\documentclass[12pt,oneside]{amsart}

\usepackage{amssymb,amsxtra,mathtools}

\usepackage[
  paper=a4paper,
  headsep=20pt,
  footskip=18pt,
  text={160mm,245mm},
  centering,
  includehead
]{geometry}

\usepackage{newtxtext,newtxmath}

\usepackage[bookmarks,bookmarksdepth=3,pdfencoding=auto,psdextra]{hyperref}
\usepackage{enumitem}
\setlist[enumerate,1]{label=(\arabic*),font=\upshape}
\usepackage[capitalize,noabbrev]{cleveref}

\crefname{enumi}{}{}
\crefname{equation}{}{}
\usepackage{thmtools}
\usepackage{graphicx,xcolor}
\usepackage{tikz}
\usetikzlibrary{cd,calc,arrows,decorations.pathreplacing}
\tikzcdset{arrow style=math font}
\usetikzlibrary{arrows,spath3,intersections,nfold}
\makeatletter
\tikzset{
  offset now/.code=
    \pgfgetpath\tikz@temp\pgfsetpath\pgfutil@empty\pgfoffsetpath\tikz@temp{#1}
}
\makeatother

\usepackage[mode=buildnew]{standalone}
\usepackage{comment}

\iftrue
  \makeatletter
  \def\@settitle{%
    \begin{flushleft}%
      \LARGE\bfseries
      \strut\@title\strut
    \end{flushleft}%
  }
  \def\@setauthors{%
    \begingroup
    \def\thanks{\protect\thanks@warning}%
    \trivlist
    \raggedright
    \large \@topsep30\p@\relax
    \advance\@topsep by -\baselineskip
    \item\relax
    \author@andify\authors
    \def\\{\protect\linebreak}%
    \authors
    \ifx\@empty\contribs
    \else
    ,\penalty-3 \space \@setcontribs
    \@closetoccontribs
    \fi
    \normalfont
    \endtrivlist
    \endgroup
  }
  \def\@setaddresses{\par
    \nobreak \begingroup
    \small\raggedright
    \def\author##1{\nobreak\addvspace\smallskipamount}%
    \def\\{\unskip, \ignorespaces}%
    \interlinepenalty\@M
    \def\address##1##2{\begingroup
      \par\addvspace\bigskipamount\noindent
      \@ifnotempty{##1}{(\ignorespaces##1\unskip) }%
      {\ignorespaces##2}\par\endgroup}%
    \def\curraddr##1##2{\begingroup
      \@ifnotempty{##2}{\nobreak\noindent\curraddrname
        \@ifnotempty{##1}{, \ignorespaces##1\unskip}\/:\space
        ##2\par}\endgroup}%
    \def\email##1##2{\begingroup
      \@ifnotempty{##2}{\nobreak\noindent E-mail address%
        \@ifnotempty{##1}{, \ignorespaces##1\unskip}\/:\space
        \ttfamily##2\par}\endgroup}%
    \def\urladdr##1##2{\begingroup
      \def~{\char`\~}%
      \@ifnotempty{##2}{\nobreak\noindent\urladdrname
        \@ifnotempty{##1}{, \ignorespaces##1\unskip}\/:\space
        \ttfamily##2\par}\endgroup}%
    \addresses
    \endgroup
    \global\let\addresses=\@empty
  }
  \def\@setabstracta{%
    \ifvoid\abstractbox
    \else
    \skip@30pt \advance\skip@-\lastskip
    \advance\skip@-\baselineskip \vskip\skip@
    \box\abstractbox
    \prevdepth\z@ 
    \vskip-20pt
    \fi
  }
  \renewenvironment{abstract}{%
    \ifx\maketitle\relax
    \ClassWarning{\@classname}{Abstract should precede
      \protect\maketitle\space in AMS document classes; reported}%
    \fi
    \global\setbox\abstractbox=\vtop \bgroup
    \normalfont\small
    \list{}{\labelwidth\z@
      \leftmargin0pc \rightmargin\leftmargin
      \listparindent\normalparindent \itemindent\z@
      \parsep\z@ \@plus\p@
      
    }%
  \item[\hskip\labelsep\bfseries\abstractname.]%
  }{%
    \endlist\egroup
    \ifx\@setabstract\relax \@setabstracta \fi
  }

  \def\ps@headings{\ps@empty
    \def\@evenhead{%
      \setTrue{runhead}%
      \normalfont\scriptsize
      \rlap{\thepage}\hfill
      \def\thanks{\protect\thanks@warning}%
      \leftmark{}{}}%
    \def\@oddhead{%
      \setTrue{runhead}%
      \normalfont\scriptsize
      \def\thanks{\protect\thanks@warning}%
      \rightmark{}{}\hfill \llap{\thepage}}%
    \let\@mkboth\markboth
  }\ps@headings

  \expandafter\def\csname section\endcsname{\@startsection{section}{1}%
    \z@{-1.4\linespacing\@plus-.5\linespacing}{.8\linespacing\@plus.3\linespacing}%
    {\normalfont\bfseries\Large}}
  \expandafter\def\csname subsection\endcsname{\@startsection{subsection}{2}%
    \z@{-.8\linespacing\@plus-.3\linespacing}{.5\linespacing\@plus.3\linespacing}%
    {\normalfont\bfseries\large}}
  \expandafter\def\csname subsubsection\endcsname{\@startsection{subsubsection}{3}%
    \z@{-.7\linespacing\@plus-.2\linespacing}{.3\linespacing\@plus.2\linespacing}%
    {\normalfont\bfseries\normalsize}}
  \expandafter\def\csname paragraph\endcsname{\@startsection{paragraph}{4}%
    \z@{.7\linespacing\@plus.2\linespacing}{-1.5ex}%
    {\normalfont\itshape}}
  \def\@secnumfont{\bfseries}

  \renewcommand\contentsnamefont{\bfseries}
  \def\@starttoc#1#2{\begingroup
    \setTrue{#1}%
    \par\removelastskip\vskip\z@skip
    \@startsection{}\@M\z@{\linespacing\@plus\linespacing}%
    {.5\linespacing}{
      \contentsnamefont}{#2}%
    \ifx\contentsname#2%
    \else \addcontentsline{toc}{section}{#2}\fi
    \makeatletter
    \@input{\jobname.#1}%
    \if@filesw
    \@xp\newwrite\csname tf@#1\endcsname
    \immediate\@xp\openout\csname tf@#1\endcsname \jobname.#1\relax
    \fi
    \global\@nobreakfalse \endgroup
    \addvspace{32\p@\@plus14\p@}%
    \let\tableofcontents\rela\x
  }
  \def\contentsname{Contents}
  \def\l@section{\@tocline{2}{.5ex}{0mm}{5pc}{}}
  \def\l@subsection{\@tocline{2}{0pt}{2em}{5pc}{}}
  \makeatother

\fi

\def\to{\mathchoice{\longrightarrow}{\rightarrow}{\rightarrow}{\rightarrow}}
\makeatletter
\newcommand{\shortxra}[2][]{\ext@arrow 0359\rightarrowfill@{#1}{#2}}
\def\longrightarrowfill@{\arrowfill@\relbar\relbar\longrightarrow}
\newcommand{\longxra}[2][]{\ext@arrow 0359\longrightarrowfill@{#1}{#2}}
\renewcommand{\xrightarrow}[2][]{\mathchoice{\longxra[#1]{#2}}%
  {\shortxra[#1]{#2}}{\shortxra[#1]{#2}}{\shortxra[#1]{#2}}}

\makeatother

\makeatletter
\def\addtagsub#1{\let\oldtf=\tagform@\def\tagform@##1{\oldtf{##1}\hbox{$_{#1}$}}}
\makeatother

\makeatletter
\def\Nopagebreak{\@nobreaktrue\nopagebreak}
\makeatother

\foreach \n in 
{Z,Q,R,C}
{\expandafter\xdef\csname\n\endcsname{\noexpand\mathbb{\n}}}

\foreach \n in 
{A,B,F,S}
{\expandafter\xdef\csname b\n\endcsname{\noexpand\mathbb{\n}}}

\foreach \n in 
{A,...,Z}
{\expandafter\xdef\csname c\n\endcsname{\noexpand\mathcal{\n}}}

\foreach \n in 
{dim,rk,rank,sign,sgn,Ker,Coker,Im,Hom,End,GL,SO,
 Tor,Ext,Gal,Irr,dis,Wh,Hopf,tr,Tr,vol,colim,grk,
 lk,st,Mod,Aut,SL,Dax,dax,FQ,Homeo,Diff,Emb,ks,
 CAT,TOP,DIFF,PL,O}
{\expandafter\xdef\csname\n\endcsname{\noexpand\operatorname{\n}}}

\foreach \n/\s in 
{inte/int}
{\expandafter\xdef\csname\n\endcsname{\noexpand\operatorname{\s}}}

\foreach \n in 
{id,pr,PI,Pr}
{\expandafter\xdef\csname\n\endcsname{\noexpand\mathrm{\n}}}

\foreach \n/\s in 
{csum/\#,bcsum/\natural/}
{\expandafter\xdef\csname\n\endcsname{\noexpand\mathbin{\s}}}

\def\RP{\R P}
\def\CP{\C P}

\def\rel{\ifmmode\;{\mathgroup\symoperators rel}\;\else rel\ \fi}
\def\sm{\smallsetminus}

\def\tildetimes{\mathbin{\tilde\times}}
\def\Homeo{\operatorname{Homeo}}
\def\Homeoh{\operatorname{Homeo^{\mathnormal{h}}}}

\def\Homeopi{\operatorname{Homeo^{\mathnormal{pi}}}}
\def\Diff{\operatorname{Diff}}
\def\Diffh{\operatorname{Diff^{\mathnormal{h}}}}

\def\Diffpi{\operatorname{Diff^{\mathnormal{pi}}}}
\def\DeltaD{\Delta_{D}}
\def\DeltaDh{\Delta_D^{\mathnormal{h}}}

\def\tDeltaD{\widetilde\Delta_D}

\def\DeltaGR{\Delta_{G,R}}
\def\tDeltaGR{\widetilde\Delta_{G,R}}

\declaretheoremstyle[
  name=\ignorespaces,numbered=no,notefont=\bfseries,notebraces={}{},bodyfont=\itshape
]{theorem-entitled}
\declaretheoremstyle[
  name=\ignorespaces,numbered=no,notefont=\bfseries,notebraces={}{},bodyfont=\normalfont
]{definition-entitled}

\declaretheorem[name=Theorem,refname={Theorem,Theorems}]{theoremalpha}
\declaretheorem[name=Corollary,sibling=theoremalpha]{corollaryalpha}

\declaretheorem[numberwithin=section]{theorem}
\foreach \n in {
  question,corollary,proposition,lemma,conjecture,observation
}{\declaretheorem[sibling=theorem]{\n}}
\foreach \n in {
  definition,example,remark,notation,case
}{\declaretheorem[style=definition,sibling=theorem]{\n}}
\foreach \n in {
  assertion,claim
}{\declaretheorem[numbered=no]{\n}}
\declaretheorem[numbered=no,name=Question]{question*}

\declaretheorem[style=theorem-entitled]{theorem-named}
\foreach \n in {
  definition-named,conjecture-named,case-named
}{\declaretheorem[style=definition-entitled]{\n}}

\numberwithin{equation}{section}

\def\labelignored#1{\ignorespaces}

\begin{document}

\title
{Topological mapping class groups of non-simply connected 4-manifolds}

\author{Jae Choon Cha}
\address{
  Center for Research in Topology and Department of Mathematics\\%
  POSTECH\\%
  Pohang Gyeongbuk 37673\\%
  Republic of Korea
}
\email{jccha@postech.ac.kr}


\begin{abstract}
We construct topological isotopy invariants of homeomorphisms of non-simply connected 4-manifolds using the topological Dax invariant of disks and the Freedman--Quinn invariant of spheres.
Our invariants are effective for various non-smoothable 4-manifolds without assuming that the fundamental groups are good.
They detect infinite-rank free abelian direct factors of the pseudoisotopy-trivial and homotopy-trivial mapping class groups, and give lower bounds on the Pr\"ufer rank of the Torelli groups modulo homotopy-trivial mapping classes, including examples of infinite Pr\"ufer rank.
Our method is also useful for smooth 4-manifolds.
We construct stably smoothable homeomorphisms of the Enriques surface and many Dolgachev surfaces that act trivially on homology but are not homotopic to the identity.
\end{abstract}

\maketitle

\section{Introduction}
\label{section:introduction}

Mapping class groups of 4-manifolds are the subject of active research in low-dimensional topology.
We study topological mapping class groups of non-simply connected 4-manifolds, with an emphasis on the subgroups of pseudoisotopy-trivial and homotopy-trivial mapping classes and Torelli subgroups.
For compact simply connected 4-manifolds, isotopy, pseudoisotopy, and homotopy rel boundary coincide on self-homeomorphisms~\cite{Perron:1986-1,Quinn:1986-1,Orson-Powell:2022-1,Gabai-Gay-Hartman-Krushkal-Powell:2023-1}, and this underlies elegant algebraic classifications of the mapping class group~\cite{Freedman:1982-1,Perron:1986-1,Quinn:1986-1,Orson-Powell:2022-1}.
In the non-simply connected case, by contrast, understanding the structure of the mapping class group of a 4-manifold remains a major challenge.

In this paper, we develop the topological Dax invariant of disks introduced in~\cite{Cha-Kim:2023-1} into a tool for detecting structures in mapping class groups.
We construct common free abelian retracts and direct factors of the pseudoisotopy-trivial and homotopy-trivial mapping class groups of many 4-manifolds, with quantitative rank estimates and infinite-rank examples.
We also detect nontrivial homotopy classes of homologically trivial homeomorphisms, including examples in which the Torelli group modulo homotopy-trivial mapping classes has infinite Pr\"ufer rank.
The proofs combine homomorphisms defined using the Dax invariant and its reduction to the Freedman--Quinn invariant with geometric realization of their values under a geometric dual sphere hypothesis.

There has been considerable recent activity and substantial progress in the study of topological mapping class groups of non-simply connected 4-manifolds.
These advances include infinite-generation results for the mapping class group of $S^1\times D^3$ and the homotopy-trivial mapping class groups of closed hyperbolic 4-manifolds~\cite{Budney-Gabai:2025-1}, and infinite-rank results for the pseudoisotopy-trivial mapping class groups of $Y\times I$, where $Y$ is a compact connected oriented 3-manifold with infinite fundamental group~\cite{Lin-Xie-Zhang:2026-1}.
Diffeomorphisms homotopic but not topologically pseudoisotopic to the identity are given in~\cite{Krannich-Kupers:2024-1}, and non-smoothable homotopy-trivial homeomorphisms were detected by the Casson--Sullivan invariant in~\cite{Galvin:2024-1}.
These results treat only smooth underlying 4-manifolds.
Topological analogues of the Hatcher--Wagoner pseudoisotopy invariants for arbitrary compact topological 4-manifolds were developed in~\cite{Galvin-Nonino:2025-1}.
Their full realization theorem assumes that the fundamental group is good in the sense of Freedman--Quinn~\cite{Freedman-Quinn:1990-1}.

We also remark that, in the smooth category, the Dax and Freedman--Quinn invariants have already been used to study mapping class groups.
See \cite[Corollary~1.3]{Schneiderman-Teichner:2022-1}, \cite[Theorem~2.3]{Kosanovic-Teichner:2024-2} and~\cite[Theorem~1.3]{Lin-Wu-Xie-Zhang:2026-1}.

Under the geometric dual hypothesis, our method combines detection and realization without assuming either a global smooth structure or a good fundamental group.
In particular, we obtain common $\Z^\infty$ direct factors in the pseudoisotopy-trivial and homotopy-trivial mapping class groups of non-smoothable 4-manifolds with free fundamental groups of arbitrary rank and with either value of the Kirby--Siebenmann invariant.
The results on the Torelli groups also apply to Enriques surfaces and many Dolgachev surfaces, yielding stably smoothable homeomorphisms which induce the identity on integral homology but are not homotopic to the identity.

\subsection{Main results}

Throughout this paper, we assume that 4-manifolds are oriented and connected.
A 4-manifold means a topological 4-manifold unless stated to be smooth.
Embeddings, submanifolds and their isotopies are locally flat.

We use the following notation.
Let $M$ be a 4-manifold.
Denote by $\Homeo(M)$ the group of orientation-preserving homeomorphisms $f\colon M\to M$ fixing $\partial M$ pointwise, equipped with the compact-open topology.
We assume all isotopies of such homeomorphisms are rel~$\partial M$.
The group $\pi_0\Homeo(M)$ of isotopy classes is the \emph{mapping class group} of~$M$.
Define the groups of homology-trivial, homotopy-trivial and pseudoisotopy-trivial homeomorphisms by
\[
  \begin{aligned}
    I(M) &= \{f\in\Homeo(M)\mid f_*=\id\text{ on }H_*(M;\Z)\},
    \\
    \Homeoh(M) &= \{f\in\Homeo(M)\mid f \simeq \id_M \rel \partial M\},
    \\
    \Homeopi(M) &= \{f\in\Homeo(M)\mid f\text{ is pseudoisotopic to the identity}\}.
  \end{aligned}
\]
We write $f\simeq g$ when $f$ and $g$ are properly homotopic.
When $M$ is compact, homotopies are automatically proper.
Also, $f,g\in\Homeo(M)$ are \emph{pseudoisotopic} if there is a homeomorphism $F\colon M\times I\to M\times I$ such that $F(x,0)=(f(x),0)$, $F(x,1)=(g(x),1)$ and $F=\id$ on $\partial M\times I$.

The \emph{Torelli group} of $M$ is $\cI(M)=\pi_0 I(M)$.
The homotopy-trivial and pseudoisotopy-trivial mapping class groups fit into the inclusions
\[
  \{1\}\subset\pi_0\Homeopi(M)\subset\pi_0\Homeoh(M)\subset\cI(M)\subset\pi_0\Homeo(M).
\]
When $M$ is smooth, we define the analogues $\pi_0\Diff(M)$, $\cI^s(M)=\pi_0 I^s(M)$, $\pi_0\Diffh(M)$ and $\pi_0\Diffpi(M)$ by replacing homeomorphisms by diffeomorphisms.

\subsubsection{The invariant}
\label{subsubsection:intro-invariant}

Let $M$ be a 4-manifold with nonempty boundary.
For neat disks $D_0,D_1\subset M$ which are homotopic rel boundary, the topological Dax invariant $\Dax(D_0,D_1)$ is constructed in~\cite{Cha-Kim:2023-1}.
In \cref{subsection:review-topological-dax}, we review its definition and the properties we use.
To obtain an invariant of homeomorphisms, fix a neat disk $D\subset M$.
For $f\in\Homeo(M)$ such that $D$ is homotopic to $f(D)$ rel~$\partial D$, define
\[
  \DeltaD(f) = \Dax(D,f(D)).
\]
One may view $\DeltaD(f)$ as measuring the displacement of $D$ under $f$ up to topological isotopy.

The invariant takes values in the abelian group
\[
  D_M = \Z[\pi_1(M)\sm 1]^\sigma/d(\pi_3(M)).
\]
Here $\Z[\pi_1(M)\sm 1]$ is the free abelian group on the nontrivial elements of~$\pi_1(M)$, the superscript $\sigma$ denotes the subgroup fixed by the involution induced by $g\mapsto g^{-1}$, and $d$ is the Dax homomorphism.
We recall the definitions in~\cref{subsection:review-topological-dax}.
A subgroup $\Sigma_M\subset D_M$, generated by the classes represented by $g+g^{-1}$ for $g\in\pi_1(M)\sm\{1\}$, also plays a key role in our results.
The quotient $D_M/\Sigma_M$ is a $\Z_2$-vector space, which can be nontrivial only when $\pi_1(M)$ has elements of order~2.
For more details, see~\cref{equation:D_M,equation:Sigma_M,remark:D_M/Sigma_M-is-Z_2-vector-space}.

The following theorem describes how $\DeltaD$ interacts with composition and homotopy.

\begin{theoremalpha}
  \label{theorem:intro-DeltaDh-homomorphism}
  The invariant $\DeltaD$ induces a homomorphism
  \[
    \DeltaDh\colon \pi_0 \Homeoh(M) \to \Sigma_M.
  \]
\end{theoremalpha}

The image restriction to $\Sigma_M$ is a consequence of the vanishing of the Freedman--Quinn invariant for $(D,f(D))$ when $f\simeq \id_M \rel \partial M$.
We establish this vanishing in \cref{lemma:fq-homotopy-trivial-homeomorphism} and give a direct proof of the naturality needed for additivity in \cref{lemma:top-dax-pi1-trivial-homeomorphism}.

\subsubsection{Realization}
\label{subsubsection:intro-realization}

For a neat disk $D$ embedded in a 4-manifold $M$, a 2-sphere $G$ immersed in $\partial M$ is called a \emph{geometric dual in $\partial M$} if $G\cap\partial D$ is exactly one point and the intersection is transverse.
The 2-sphere $G$ may be chosen to be embedded in $\partial M$ by the loop theorem~\cite[Lemma~2.7]{Cha-Kim:2023-1}.

A 4-manifold is \emph{stably smoothable} if it admits a smooth structure after connected sum with finitely many copies of $S^2\times S^2$.
A homeomorphism of a stably smoothable 4-manifold is \emph{stably smoothable} if, after stabilization by the identity on sufficiently many $S^2\times S^2$ summands, it is topologically isotopic rel boundary to a diffeomorphism.
Our precise conventions are given at the beginning of \cref{section:realization}.

\begin{theoremalpha}
  \label{theorem:intro-realization-Sigma_M}
  If $M$ is a 4-manifold and $D$ is a neat disk in $M$ that has a geometric dual in~$\partial M$, then there exists an abelian subgroup $H\subset \pi_0 \Homeopi(M) \subset \pi_0\Homeoh(M)$ which maps onto $\Sigma_M$ under~$\DeltaDh$.
  In addition, if $M$ is smooth, then all elements of $H$ may be represented by diffeomorphisms which commute up to smooth isotopy;
  if $M$ is stably smoothable, then all elements of $H$ may be represented by stably smoothable homeomorphisms.
\end{theoremalpha}

The realizing homeomorphisms are constructed explicitly using barbell maps of Budney and Gabai~\cite{Budney-Gabai:2019-1,Budney-Gabai:2025-1}.
The commutativity in~\cref{theorem:intro-realization-Sigma_M} is essential in realizing free abelian retracts in the results below.

\subsubsection{Structures in the mapping class groups}
\label{subsubsection:intro-structures}

We construct concrete free abelian subgroups of mapping class groups by estimating the size of the target group~$\Sigma_M$.
Our estimates are expressed in terms of the inversion orbits of the fundamental group, Wall's reduced self-intersection invariant, and the third homology of the universal cover.
The inversion orbits give generators and the latter two bound the loss of rank caused by the Dax relations.
Write
\[
  \mu_2\colon\pi_2(M)\to\Z[\pi_1(M)]/\langle 1,g-g^{-1}\rangle
\]
for Wall's reduced self-intersection invariant~\cite[Chapter~5]{Wall-Ranicki:1999-1}, \cite[Section~1.7]{Freedman-Quinn:1990-1}, and let $\rk\mu_2$ be the rank of the subgroup generated by its image.
Let $\widetilde M$ be the universal cover of $M$ and $\beta_3(\widetilde M)$ its third Betti number.
Let $T_M$ be the set of nontrivial elements of order~2 in $\pi_1(M)$, and let $\iota_M$ be the number of non-identity inversion orbits in~$\pi_1(M)$.
Thus $\iota_M=\frac12(|\pi_1(M)|+|T_M|-1)$ when $\pi_1(M)$ is finite, and $\iota_M=|\pi_1(M)|$ otherwise.
Let $\Z^r$ denote the free abelian group of rank~$r$, including when $r$ is infinite.

For subgroups $A\subset\Lambda\subset\Gamma$, we call $A$ a \emph{common retract} of $\Lambda$ and $\Gamma$ if it is a retract of $\Gamma$, and hence also of~$\Lambda$.
We define a \emph{common direct factor} similarly.

\begin{theoremalpha}
  \label{theorem:intro-retract-direct-factor}
  Let $M$ be a 4-manifold containing a neat disk with a geometric dual in~$\partial M$.
  Suppose $\rk\mu_2 < \infty$ and $\beta_3(\widetilde M) < \infty$.
  Let $r=\iota_M$ if $\iota_M$ is infinite;
  otherwise, let
  \[
    r = \iota_M-\bigl(\rk\mu_2+\beta_3(\widetilde M)\bigr)
  \]
  and suppose $r>0$.
  Then $\pi_0\Homeopi(M)$ and $\pi_0\Homeoh(M)$ have a common free abelian retract $A\cong \Z^r$ of rank~$r$.
  In addition, if $\pi_1(\partial M)\to\pi_1(M)$ is surjective, then $A$ is a common direct factor of $\pi_0\Homeopi(M)$ and $\pi_0\Homeoh(M)$.
\end{theoremalpha}

The following consequences illustrate the rank estimate.
In both statements, we continue to assume that $M$ contains a neat disk with a geometric dual in~$\partial M$.

\begin{enumerate}
  \item[(a)]
  If $M$ is non-simply connected, $\mu_2=0$ and $H_3(\widetilde M;\Q)=0$, then there is a common infinite cyclic retract of $\pi_0 \Homeopi(M)$ and $\pi_0 \Homeoh(M)$.
  \item[(b)]
  If $\pi_1(M)$ is infinite, $\rk\mu_2<\infty$ and $\beta_3(\widetilde M)<\infty$, then there is a common $\Z^\infty$ retract of $\pi_0 \Homeopi(M)$ and $\pi_0 \Homeoh(M)$.
  In particular, neither is finitely generated.
\end{enumerate}

When $M$ is smooth, these retracts admit realizations by diffeomorphisms.
The precise smooth and stably smoothable refinements are given in the addendum to~\cref{theorem:intro-retract-direct-factor} in~\cref{section:detecting-structures}.

The geometric dual hypothesis is a nontrivial restriction, but it holds for broad families of examples.
For instance, \cref{theorem:intro-retract-direct-factor} applies to $M=X\bcsum(S^2\times D^2)$ whenever $X$ is a non-contractible aspherical 4-manifold with nonempty boundary; see~\cref{lemma:aspherical-bcsum-S^2-times-D^2,subsubsection:examples-from-aspherical}.
For other types of examples, see also~\cref{example:nonaspherical-with-2-torsion-in-pi_1,,example:nonaspherical-one-D2-times-RP2-summand,,example:M_k(L)-nontrivial-ks,,example:nonsmoothable-with-2-torsion-in-pi_1,,example:nonsmoothable-with-ks-zero}.

\subsubsection{Non-smoothable 4-manifolds and freely slice links}
\label{subsubsection:intro-freely-slice}

The following result gives non-smoothable examples with free fundamental groups of arbitrary rank.
Let $\ks(M)\in\Z_2$ denote the Kirby--Siebenmann invariant of a 4-manifold~$M$.

\begin{theoremalpha}
  \label{theorem:intro-Z^infty-factor-nonsmoothable}
  For every integer $n\ge 1$ and for every $\epsilon\in\Z_2$, there exist infinitely many pairwise non-homeomorphic non-smoothable compact 4-manifolds $M$ with $\pi_1(M)$ free of rank $n$ and $\ks(M)=\epsilon$, such that $\pi_0 \Homeopi(M)$ and $\pi_0\Homeoh(M)$ have a common $\Z^\infty$ direct factor.
\end{theoremalpha}

\begin{figure}[!b]
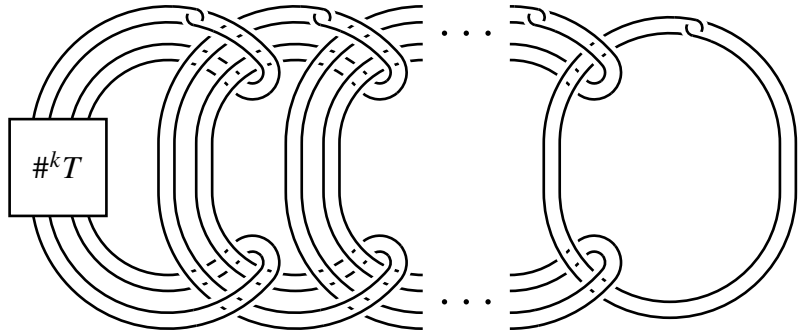

  \centering
  \includestandalone{fig-gbl-example}
  \caption{
    A family of $n$-component freely slice links ($n\ge 2$) that are not smoothly slice in any smooth homotopy 4-ball.
    Each component is the positive Whitehead double of an unknotted circle.
    The box {\setlength{\fboxrule}{1pt}\fbox{\footnotesize$\#^{\smash{k}} T$}} indicates that the $k$-fold connected sum of the right-handed trefoil $T$ is tied in, without introducing any twists among the strands involved, where $k\ge 1$ is arbitrary.}
  \label{figure:freely-slice-example}
\end{figure}

For the $\ks(M)=0$ case of~\cref{theorem:intro-Z^infty-factor-nonsmoothable}, a family of examples arises from the link concordance approach to the topological surgery conjecture, a central open problem in 4-dimensional topology~\cite[p.~243]{Freedman-Quinn:1990-1}.
A link $L\subset S^3$ is \emph{freely slice} if it bounds disjoint disks in $D^4$ whose exterior has free fundamental group generated by the meridians.
The topological surgery conjecture for arbitrary fundamental groups is equivalent to the assertion that all links in a certain class, called \emph{good boundary links}, are freely slice~\cite{Freedman-Quinn:1990-1,Freedman-Teichner:1995-2}.
This motivated the construction of large classes of freely slice links in~\cite{Freedman:1985-1,Freedman:1988-2,Freedman:1993-1,Freedman-Teichner:1995-2,Cha-Kim-Powell:2020-1}.

For our purpose, it is useful that many known freely slice links are not smoothly slice in any smooth homotopy 4-ball.
We show that, if $X$ is a free slice disk exterior for such a link, then $M=X\bcsum (S^2\times D^2)$ is non-smoothable, has $\ks(M)=0$ and has the mapping class group property in~\cref{theorem:intro-Z^infty-factor-nonsmoothable}.
An explicit family of such links is shown in~\cref{figure:freely-slice-example}.
See \cref{section:non-smoothable-4-manifolds} for details and further discussion.

The $\epsilon=1$ case and further non-smoothable examples with either Kirby--Siebenmann invariant are treated in~\cref{subsection:non-smoothable-given-homotopy-type}.

\subsubsection{Applications to Torelli groups}

For simply connected 4-manifolds $M$, the topological Torelli group $\cI(M)$ is well understood.
If $M$ is closed, $\cI(M)$ is trivial by~\cite[Theorem~1.1]{Quinn:1986-1}; see also~\cite{Gabai-Gay-Hartman-Krushkal-Powell:2023-1}.
For compact $M$ with nonempty boundary, a full computation of $\cI(M)$ is given by~\cite[Corollary~D]{Orson-Powell:2022-1}.
We focus on the non-simply connected case.
Since $\pi_0\Homeoh(M)\subset\cI(M)$, \cref{theorem:intro-retract-direct-factor} already detects free abelian subgroups in~$\cI(M)$.
To investigate homologically trivial homeomorphisms which are not homotopy-trivial, we study the quotient $\cI(M)/\pi_0\Homeoh(M)$, the group of proper homotopy classes rel~$\partial M$ of homologically trivial homeomorphisms.

The invariant $\DeltaD$ is also a homomorphism on the larger subgroup of homeomorphisms inducing the identity on $\pi_1(M)$ and preserving the relative homotopy class of~$D$.
Its reduction modulo $\Sigma_M$ is the Freedman--Quinn invariant of disks~\cite{Kosanovic-Teichner:2024-2}, and it vanishes on~$\Homeoh(M)$.
In the presence of a boundary dual, arc pushing maps realize every value of $D_M$ by Torelli homeomorphisms;
see~\cref{theorem:realization-DeltaD}.
This gives abelian quotients of subgroups of $\cI(M)/\pi_0\Homeoh(M)$.

To state a consequence of this, we use the notion of the Pr\"ufer rank, which is a dimension-type measure of the complexity of the group structure.
For a group $\Gamma$, the \emph{Pr\"ufer rank} $\rk_{\Pr}(\Gamma)$ is defined to be the supremum of the minimal numbers of generators of its finitely generated subgroups.
For example, the free group on two generators contains a free group of arbitrarily large finite rank, and thus has infinite Pr\"ufer rank.

\begin{theoremalpha}
  \label{theorem:intro-torelli-prufer-rank}
  Let $M$ be a 4-manifold containing a neat disk with a geometric dual in~$\partial M$.
  Suppose \(H_3(\widetilde M)\otimes\Z_2\) is finitely generated.
  Let $s=|T_M|$ if $T_M$ is infinite;
  otherwise, let $s = |T_M| - \dim_{\Z_2}(H_3(\widetilde M)\otimes\Z_2)$.
  Then
  \[
    \rk_{\Pr}(\cI(M)) \ge \rk_{\Pr}(\cI(M)/\pi_0\Homeoh(M)) \ge s.
  \]
  In addition, if $M$ is smooth, then
  \[
    \begin{aligned}
      \rk_{\Pr}(\cI^s(M)) & \ge \rk_{\Pr}\bigl(\cI^s(M)/\pi_0\Diffh(M)\bigr) \ge s.
    \end{aligned}
  \]
\end{theoremalpha}

Indeed, when $s>0$, the proof constructs a subgroup $H\subset\cI(M)/\pi_0\Homeoh(M)$ admitting an epimorphism onto $(\Z_2)^s$, the $\Z_2$-vector space of dimension~$s$.
When $M$ is smooth, $H$ can be chosen inside the image of $\cI^s(M)/\pi_0\Diffh(M)$, so the lower bound holds for the Pr\"ufer rank of this image as well.
See~\cref{theorem:subgroup-in-torelli,corollary:torelli-modulo-homotopy-prufer-rank} for these refinements and the stably smoothable case.
Applications to smooth and non-smoothable manifolds are given in~\cref{example:nonaspherical-with-2-torsion-in-pi_1,,example:nonaspherical-one-D2-times-RP2-summand,,example:nonaspherical-X-D2-times-RP2-summands,,example:nonsmoothable-with-2-torsion-in-pi_1,,example:nonsmoothable-with-ks-zero}.

\paragraph{Torelli groups of closed 4-manifolds}

The same approach applies to closed 4-manifolds containing an embedded sphere with a framed geometric dual, using the Freedman--Quinn invariant of spheres in place of the disk invariant.
In this case, a geometric dual in the interior is not necessarily preserved by a homeomorphism, causing additional restrictions.
However, we can still obtain an analogue of~\cref{theorem:intro-torelli-prufer-rank}, which is given in~\cref{corollary:torelli-prufer-rank-closed} (see also \cref{corollary:subgroup-in-torelli-closed}), with applications in~\cref{subsection:applications-closed}.

In particular, our results apply to compact complex surfaces, as described below.

\subsubsection{Applications to complex surfaces}
\label{subsubsection:intro-complex-surface}

Smooth Torelli groups are known to have rich structures in many 4-manifold examples.
For instance, see~\cite{Ruberman:1998-1,Ruberman:1999-1,Baraglia:2026-1,Baraglia-Tomlin:2025-1}.
In contrast, the $K3$ surface and the Enriques surface $E$ are expected to exhibit Torelli rigidity:
it is conjectured that the smooth Torelli groups $\cI^s(K3)$ and $\cI^s(E)$ are trivial~\cite{Looijenga:2021-1,Raman:2025-1}.
Raman proved that $\cI^s(K3)$ is trivial if and only if $\cI^s(E)$ is trivial, using his Birman--Hilden type theorem relating the mapping class group of $E$ to the symmetric mapping class group of its K3 cover~\cite{Raman:2025-1}.
The difficulty of these conjectures also motivates the study of the corresponding topological Torelli groups, as discussed in~\cite[Section~9]{Raman:2025-1}.

In the topological category, $\cI(K3)=\pi_0\Homeoh(K3)=\{1\}$ since $K3$ is simply connected~\cite{Quinn:1986-1,Gabai-Gay-Hartman-Krushkal-Powell:2023-1}.
We show, however, that $\cI(E)/\pi_0\Homeoh(E)$ has a nontrivial element represented by a stably smoothable homeomorphism.
In fact, our method applies to the following class of compact complex surfaces.

\begin{theoremalpha}
  \label{theorem:intro-torelli-complex-surfaces}
  Let $M$ be a compact complex surface with a nontrivial finite cyclic fundamental group of even order and with Euler characteristic not equal to~$4$.
  Then $M$ admits a stably smoothable homeomorphism $f$ which represents a nontrivial element of $\cI(M)/\pi_0\Homeoh(M)$, i.e., $f_* = \id$ on $H_*(M)$ and $f$ is not homotopic to~$\id_M$.
\end{theoremalpha}

The proof combines the sphere version of our invariant with the topological splitting results of Hambleton and Kreck~\cite{Hambleton-Kreck:1988-1}.

Let $D_{p,q}$ denote a Dolgachev surface of type $(p,q)$, whose elliptic fibration has two multiple fibers of multiplicities $p$ and~$q$.

\begin{corollaryalpha}
  \label{corollary:intro-torelli-dolgachev-enriques}
  If $M$ is a Dolgachev surface $D_{p,q}$ with $p$, $q$ even, or if $M$ is an Enriques surface, then there exists a stably smoothable homeomorphism $f$ on $M$ which represents a nontrivial element in $\cI(M)/\pi_0\Homeoh(M)$.
\end{corollaryalpha}

We remark that Galvin's work~\cite[Theorem~1.9]{Galvin:2024-1} gives a nontrivial element of the subgroup $\pi_0\Homeoh(E)\subset\cI(E)$ detected by the Casson--Sullivan invariant.
In contrast, the homeomorphisms in~\cref{theorem:intro-torelli-complex-surfaces,corollary:intro-torelli-dolgachev-enriques} represent nontrivial classes in the quotient by $\pi_0\Homeoh(M)$ and have vanishing Casson--Sullivan invariant.

Our method for Torelli groups is also effective for non-smoothable closed 4-manifolds.
For instance, the star-partner $*E$ of an Enriques surface $E$ is non-smoothable, has the homotopy type of $E$, and satisfies $\ks(*E)=1$.
We show that $\cI(*E)/\pi_0\Homeoh(*E)$ is nontrivial.
See~\cref{theorem:torelli-star-enriques}.

From the above results, a natural question arises:
\emph{is the stably smoothable homeomorphism $f$ in~\cref{theorem:intro-torelli-complex-surfaces,corollary:intro-torelli-dolgachev-enriques} homotopic to a diffeomorphism?}
An affirmative answer for Enriques surfaces would imply that both Enriques and K3 surfaces have nontrivial \emph{smooth} Torelli groups, disproving the above Torelli-triviality conjectures.

\subsection{Organization of the paper}

In~\cref{section:topological-dax,section:invariant-of-homeo}, we review the topological Dax invariant and establish its properties for homeomorphisms, proving~\cref{theorem:intro-DeltaDh-homomorphism}.
In~\cref{section:realization,section:detecting-structures}, we construct the realizing homeomorphisms and establish the results on retracts, direct factors and Torelli groups, proving \cref{theorem:intro-realization-Sigma_M,theorem:intro-retract-direct-factor,theorem:intro-torelli-prufer-rank}.
\Cref{section:non-smoothable-4-manifolds,section:more-examples} contain the smooth and non-smoothable applications, including the proof of~\cref{theorem:intro-Z^infty-factor-nonsmoothable}.
In~\cref{section:closed-4-manifolds}, we develop the sphere version and its applications to closed 4-manifolds, proving~\cref{theorem:intro-torelli-complex-surfaces,corollary:intro-torelli-dolgachev-enriques}.

\subsubsection*{Acknowledgements}

This work was supported by the National Research Foundation grant RS-2019-NR039996.
The author used AI tools to improve the language and readability of the manuscript.
The author takes full responsibility for its content.

\section{Topological Dax invariant of disks}
\label{section:topological-dax}

In this section, we briefly review the topological Dax invariant, set up the relevant notation, and summarize the properties that we use.
We also discuss the relationship between the Dax invariant and concordance via the Freedman--Quinn invariant.
The key notation and properties are given in \cref{equation:group-ring-and-symmetric-subgroup,equation:dax-self-intersection,equation:dax-homomorphism,equation:D_M,equation:Sigma_M,definition:topological-dax,equation:self-intersection-mu,proposition:top-dax-properties,remark:D_M/Sigma_M-is-Z_2-vector-space}.

Throughout this section, let $M$ be a 4-manifold, and fix an almost smooth structure $\Gamma$ on~$M$.
This means that $\Gamma$ is a smooth structure on the complement of an interior point $p_0$ in~$M$.
Every 4-manifold $M$ admits an almost smooth structure~\cite[Section~8.2]{Freedman-Quinn:1990-1}.
We assume that the boundary of a neat disk in $M$ is smooth in the smooth 3-manifold $\partial M = \partial(M\sm p_0)_\Gamma$.

We write $\Z[G]$ for the integral group ring of a group $G$, $1\in G$ for the identity, and let
\begin{equation}
  \begin{aligned}
    \Z[G \sm 1] &= \textstyle \bigl\{ \sum n_g g \in \Z[G] \mid n_1 = 0 \bigr\},
    \\
    \Z[G \sm 1]^\sigma &= \textstyle \bigl\{\sum n_g g \in \Z[G\sm 1] \mid n_g = n_{g^{-1}} \text{ for all } g\ne 1 \bigr\}.
  \end{aligned}
  \label{equation:group-ring-and-symmetric-subgroup}
\end{equation}

\subsection{A quick review of the topological Dax invariant}
\label{subsection:review-topological-dax}

\subsubsection{Dax self-intersection for a smooth 4-manifold}
\label{subsubsection:smooth-dax-self-intersection}

We begin with the Dax self-intersection invariant of a homotopy $H\colon I^3\to X$ in a smooth 4-manifold $X$, following~\cite{Dax:1972-1,Gabai:2021-1,Kosanovic-Teichner:2024-2}.
Fix $z_0$,~$z_1\in \partial X$, $z_0\ne z_1$.
We use $z_0$ as the basepoint.
Suppose that $H$ has the following property on the boundary:
$H$ is smooth on $\partial I^3$, $H(i\times I^2)=\{z_i\}$ for $i=0,1$, and for each $(t,u)\in \partial I^2$, the map $I\to X$ defined by $s\mapsto H(s,t,u)$ is a smooth embedding.
Associated with $H$, define $\overline H\colon I^3 \to X \times I^2$ by $\overline H(s,t,u)=(H(s,t,u),t,u)$.
Note that $\overline H$ smoothly embeds~$\partial I^3$.
By~\cite[Chapters III and~VI]{Dax:1972-1}, $H$ is homotopic rel $\partial I^3$, via an arbitrarily small homotopy, to a smooth map $F\colon I^3\to X$ such that $\overline F\colon I^3 \to X \times I^2$ is a smooth generic immersion, i.e., a smooth immersion which is an embedding except at finitely many transverse double points.
Note that, for each double point $p=\overline F(x_0) = \overline F(x_1)$, the preimages are of the form $x_0=(s_0,t,u)$ and $x_1=(s_1,t,u)$ and lie in the interior of~$I^3$.
The key is that $x_0$ and $x_1$ are canonically ordered by the condition $s_0 < s_1$.
This ordering determines the intersection sign $\epsilon(p) \in \{1,-1\}$ and the double point loop label $g(p) \in \pi_1(X)$ without the usual ambiguity.
Choose a neighborhood $U_i\subset I^3$ of $x_i$.
The product of the local orientations of $\overline F(U_0)$ and $\overline F(U_1)$ gives a local orientation of $X\times I^2$ at~$p$, and $\epsilon(p)=1$ if it agrees with the preferred orientation of $X\times I^2$;
$\epsilon(p)=-1$ otherwise.
The element $g(p)\in \pi_1(X \times I^2)=\pi_1(X)$ is the class of the loop $\overline F|_{[0,s_0]\times t\times u} \cdot (\overline F|_{[0,s_1]\times t\times u})^{-1}$.
Define the \emph{Dax self-intersection} of $H$ by
\begin{equation}
  \dax(H) = \sum_{g(p)\ne 1} \epsilon(p)\cdot g(p) \in \Z[\pi_1(X) \sm 1].
  \label{equation:dax-self-intersection}
\end{equation}
It is known that $\dax(H)$ is well-defined, independent of the choice of the perturbation~$F$.
Also, $\dax(H)$ is invariant under homotopy of $H$ rel~$\partial$.

In what follows, we will consider $\dax(H)$ defined in the smooth 4-manifold $X=(M\sm p_0)_\Gamma$.
(Recall that $M$ is a given topological 4-manifold and $\Gamma$ is a smooth structure on $M\sm p_0$.)
Note that $\pi_1(M)=\pi_1(X)$ and thus $\dax(H)$ is an element in $\Z[\pi_1(M) \sm 1]$.

\subsubsection{Dax homomorphism of \texorpdfstring{$\pi_3(M)$}{\pi₃(M)} for a topological 4-manifold}

For a smooth 4-manifold $X$, a homomorphism of $\pi_3(X)$ is defined using \cref{equation:dax-self-intersection} in \cite{Gabai:2021-1,Kosanovic-Teichner:2024-2}, and it is generalized to the case of topological 4-manifolds (which may be non-smoothable) in~\cite{Cha-Kim:2023-1}.
We recall the definition below.
Fix a smooth embedding $\zeta\colon I \to \partial M = \partial(M\sm p_0)_\Gamma$, and identify $\pi_3(M)$ with the collection of homotopy classes of maps $H\colon I^3 \to M$ rel $\partial I^3$ such that $H(s,t,u)=\zeta(s)$ for each $(s,t,u)\in I\times \partial I^2$ and $H(i,s,t)=\zeta(i)$ for $(i,s,t)\in \{0,1\} \times I^2$.
For such a map $H$, one can perturb $H$ by homotopy rel $\partial I^3$ so that $H$ maps $I^3$ into the smooth 4-manifold $(M\sm p_0)_\Gamma$, using topological transversality.
Then $\dax(H) \in \Z[\pi_1(M)\sm 1]$ in \cref{equation:dax-self-intersection} is defined in the smooth 4-manifold $(M\sm p_0)_\Gamma$.
The basepoint is $z_0=\zeta(0)$ as before.
By~\cite[Corollary~4.13]{Kosanovic-Teichner:2024-2}, we have $\dax(H) \in \Z[\pi_1(M)\sm 1]^\sigma \subset \Z[\pi_1(M)\sm 1]$.
By~\cite[Lemma~4.12]{Cha-Kim:2023-1}, the assignment $[H] \mapsto \dax(H)$ is a well-defined group homomorphism
\begin{equation}
  d=d_\zeta\colon \pi_3(M) \to \Z[\pi_1(M)\sm 1]^\sigma.
  \label{equation:dax-homomorphism}
\end{equation}
We call $d$ the \emph{Dax homomorphism}.
We remark that the composition of $\pi_3((M\sm p_0)_\Gamma) \to \pi_3(M)$ and $d$ in \cref{equation:dax-homomorphism} is the Dax homomorphism of the smooth manifold $(M\sm p_0)_\Gamma$ defined in~\cite{Gabai:2021-1,Kosanovic-Teichner:2024-2}.

\subsubsection{Topological Dax invariant of disks}
\label{subsubsection:topological-dax-definition}

Let $D_0$ and $D_1$ be neat disks in the given topological 4-manifold~$M$, with common boundary $\partial D_0 = \partial D_1$.
Suppose $D_0$ and $D_1$ are homotopic rel~$\partial$.
Fix a smooth quotient map $I^2\to D^2$ that collapses each of $0\times I$ and $1\times I$ and maps $I\times 1$ to the upper semicircle of~$D^2$ such that the map $I\to D^2$ given by $s\mapsto [s,t]$ is a smooth embedding for all~$t$.
Compose the given embedding $D_i\colon D^2 \to M$ with the quotient map $I^2 \to D^2$, to view the disk $D_i$ as a map $D_i \colon I^2 \to M$.
Let $\zeta_+\colon I \to M$ be the upper semicircle of~$\partial D_0=\partial D_1$ and let $d=d_{\zeta_+}$ be the associated Dax homomorphism in~\cref{equation:dax-homomorphism}.
Choose a homotopy $H\colon I^2\times I \to M$ rel $\partial $ from $D_0$ to~$D_1$.

Fix a small smooth 4-ball $B$ in $(M\sm p_0)_\Gamma$, which will be used later to form a connected sum.
By topological isotopy of $D_i$, we may assume that each $D_i$ is topologically embedded in the smooth 4-manifold $(M\sm \{p_0\})_\Gamma$ and disjoint from~$B$.
By~\cite[Theorem~D]{Cha-Kim:2023-1}, there is a topological isotopy (rel $\partial$) of $D_i$ to a smooth disk $D_i'$ in the smooth 4-manifold
\[
  M' \coloneqq (M\sm p_0)_\Gamma \csum k (S^2\times S^2)
\]
for some $k\ge 0$, where the connected sum is formed at~$B \subset (M\sm p_0)_\Gamma$.
Denote this isotopy by $h_i\colon D_i \approx D_i'$.
We can also assume that the homotopy $H$ avoids $p_0$ and~$B$.
Concatenate the inverse isotopy $h_0^{-1}\colon D_0' \approx D_0$, the homotopy $H\colon D_0\simeq D_1$ and the isotopy $h_1\colon D_1 \approx D_1'$, to obtain a homotopy $H'\colon D_0' \simeq D_1'$ in the smooth manifold~$M'$.
Since $D_0'$ and $D_1'$ are smooth disks, $H'$ satisfies the boundary property required to define the Dax self-intersection $\dax(H') \in \Z[\pi_1(M')\sm 1] = \Z[\pi_1(M)\sm 1]$ in the smooth manifold~$M'$.
By~\cite[Corollary~4.13]{Kosanovic-Teichner:2024-2}, $\dax(H')$ lies in $\Z[\pi_1(M)\sm 1]^\sigma$.

Let $D_M$ be the cokernel of the Dax homomorphism~$d$:
\begin{equation}
  D_M \coloneqq \Z[\pi_1(M)\sm 1]^\sigma / d(\pi_3(M)).
  \label{equation:D_M}
\end{equation}

\begin{definition}[\cite{Cha-Kim:2023-1}]
  \label{definition:topological-dax}
  The \emph{topological Dax invariant} for $(D_0,D_1)$ in $M$ is defined to be the class of $\dax(H')$ modulo $d(\pi_3(M))$:
  \[
    \Dax_{M,\Gamma}(D_0,D_1) \coloneqq \dax(H') \in D_M.
  \]
\end{definition}

By~\cite[Theorem~4.7]{Cha-Kim:2023-1}, $\Dax_{M,\Gamma}(D_0,D_1)$ is well-defined, independent of the choice of the smoothings $D_0'$, $D_1'$ in $M'$ and the choice of the homotopy $H$ in~$M$.

We often write $\Dax(D_0,D_1)$ for $\Dax_{M,\Gamma}(D_0,D_1)$ when the choices of $M$ and $\Gamma$ are understood from the context.

\subsubsection{Properties of the topological Dax invariant}
\label{subsubsection:topological-dax-properties}

We summarize basic properties of $\Dax(D_0,D_1)$ in the following proposition.

\begin{proposition}
  \label{proposition:top-dax-properties}
  \leavevmode\Nopagebreak
  \begin{enumerate}
    \item\label{item:top-dax-isotopy-obstruction} If $D_0$ and $D_1$ are isotopic rel $\partial$, then $\Dax(D_0,D_1)=0$.
    \item\label{item:top-dax-additivity} $\Dax(D_0,D_1) + \Dax(D_1,D_2) = \Dax(D_0,D_2)$ and $\Dax(D_1,D_0)=-\Dax(D_0,D_1)$.
    \item\label{item:top-dax-isotopy-invariance} If $D_i'$ is isotopic to $D_i$ rel $\partial$ ($i=0,1$), then $\Dax(D_0,D_1)=\Dax(D_0',D_1')$.
    \item\label{item:top-dax-equals-smooth-dax} If $M$ is a smooth 4-manifold equipped with a smooth structure $\Gamma$, then for smooth neat disks in $M$, the topological Dax invariant $\Dax_{M,\Gamma}$ is equal to the smooth Dax invariant in \cite{Gabai:2021-1,Kosanovic-Teichner:2024-2}.
  \end{enumerate}
\end{proposition}

\begin{proof}
  For \cref{item:top-dax-isotopy-obstruction}, \cref{item:top-dax-additivity} and \cref{item:top-dax-isotopy-invariance}, see~\cite[Theorems 4.14 and 4.15]{Cha-Kim:2023-1}.
  For \cref{item:top-dax-equals-smooth-dax}, see~\cite[Theorem~4.9(3)]{Cha-Kim:2023-1}.
\end{proof}

\subsection{Concordance of disks and the Freedman--Quinn invariant}
\label{subsection:disk-concordance-and-fq}

Let $D_0$ and $D_1$ be neat disks in the given 4-manifold $M$ with common boundary $\partial D_0=\partial D_1$.
View $D_i$ as an embedding $D_i\colon D^2 \to M$.
Recall that $D_0$ and $D_1$ are \emph{concordant} if there is a neat embedding $C\colon D^2\times I \to M\times I$ such that $C(x,i) = (D_i(x),i) \in M\times i$ for $(x,i)\in D^2\times \{0,1\}$ and $C(x,t) = (D_0(x),t)$ for $(x,t)\in \partial D^2 \times I$.
Note that concordance implies homotopy rel $\partial$, and thus $\Dax(D_0,D_1)$ is defined whenever $D_0$ and $D_1$ are concordant.

Let $\Sigma_M$ be the subgroup of $D_M$ generated by the cosets of $g+g^{-1}$ where $g\in \pi_1(M) \sm \{1\}$:
\begin{equation}
  \Sigma_M = \langle g+g^{-1} \mid g\in \pi_1(M),\; g\ne 1 \rangle \subset D_M = \Z[\pi_1(M)\sm 1]^\sigma / d(\pi_3(M)).
  \label{equation:Sigma_M}
\end{equation}

\begin{proposition}
  \label{proposition:dax-concordant-disks}
  If two neat disks $D_0$ and $D_1$ in $M$ are concordant, then $\Dax(D_0,D_1)$ lies in the subgroup~$\Sigma_M$.
\end{proposition}

\Cref{proposition:dax-concordant-disks} holds in both smooth and topological categories.
\Cref{proposition:concordance-disk-pseudoisotopy} tells us that the converse of \cref{proposition:dax-concordant-disks} is true if $D_i$ admits a geometric dual in~$\partial M$.

\begin{proof}
  \Cref{proposition:dax-concordant-disks} is an immediate consequence of the following statements:
  \begin{enumerate}[label=(\roman{*})]
    \item\label{item:dax-reduction-fq}
    The reduction of the value of the Dax invariant modulo $\Sigma_M$ is equal to the Freedman--Quinn invariant for disks.
    \item\label{item:fq-concordance-disk}
    The Freedman--Quinn invariant vanishes for concordant disks.
  \end{enumerate}
  \let\qed\relax
\end{proof}

In what follows, we discuss the Freedman--Quinn invariant for disks and how one verifies \cref{item:dax-reduction-fq,item:fq-concordance-disk}.

In the smooth category, the disk version of the Freedman--Quinn invariant is defined by Kosanovi\'c and Teichner~\cite[Definition~4.12]{Kosanovic-Teichner:2024-2}, adapting the approach of Schneiderman and Teichner~\cite{Schneiderman-Teichner:2022-1} for spheres.
They proved \cref{item:dax-reduction-fq} in~\cite[Corollary~4.13]{Kosanovic-Teichner:2024-2}, and \cref{item:fq-concordance-disk} follows directly from their definition, analogously to the sphere case explicitly treated in~\cite{Schneiderman-Teichner:2022-1}.

It is known that the sphere version of the Freedman--Quinn invariant can also be defined in the topological category.
For instance, in \cite[Theorem~6.1 and Appendix~A]{Cha-Kim:2023-1}, the smooth category approach of~\cite{Schneiderman-Teichner:2022-1} is adapted using the Kirby--Siebenmann smoothing theory and topological immersion theory.
A similar argument applies to topological disks, and the topological versions of \cref{item:dax-reduction-fq,item:fq-concordance-disk} can be derived from the smooth analogues.
As we are not aware of an explicit treatment of the topological disk case in the literature, we include a concise argument below for completeness.

\subsubsection{Freedman--Quinn invariant for topological disks}
\label{subsubsection:freedman-quinn-top-disks}

Let $D_0$ and $D_1$ be neat disks in an oriented topological 4-manifold~$M$ with common boundary $\partial D_0=\partial D_1$.
Suppose $D_0$ and $D_1$ are homotopic rel $\partial$ in~$M$.
Choose a homotopy $H\colon D^2 \times I \to M$ rel $\partial$ between $D_0$ and $D_1$ and define $\widehat H \colon D^2 \times I \to M\times I^2$ by $\widehat H(x,t) = (H(x,t),t,\frac12)$.
Using arguments in \cite[Theorem~6.1 and Appendix~A]{Cha-Kim:2023-1}, one can perturb $\widehat H$, by an arbitrarily small homotopy rel $\partial$, to a topological generic immersion satisfying the local properties of a smooth generic immersion.
By counting double points weighted by associated intersection signs and $\pi_1$ elements, the \emph{self-intersection} $\mu(\widehat H)$ is defined as an element in $\Z[\pi_1(M)]/\langle 1,\; g+g^{-1} \mid g\in \pi_1(M)\rangle$.
By arguments in \cite[Lemma~4.11]{Kosanovic-Teichner:2024-2}, $\mu(\widehat H)$ is fixed under the involution $g\mapsto g^{-1}$.
Thus, the self-intersection is an element
\begin{equation}
  \mu(\widehat H) \in \Z[\pi_1(M)\sm 1]^\sigma/\langle g+g^{-1} \mid g\in \pi_1(M)\sm 1 \rangle \cong \Z_2[T_M].
  \label{equation:self-intersection-mu}
\end{equation}
Here, recall that $T_M$ is the set of order 2 elements in~$\pi_1(M)$.
The isomorphism in \cref{equation:self-intersection-mu} is induced by the inclusion $T_M \subset \pi_1(M) \sm \{1\}$.
Compared to the Dax self-intersection $\dax(H)$ in \cref{subsection:review-topological-dax}, $\mu(\widehat H)$ is well-defined only modulo the additional ambiguity of $g+g^{-1}$, since the two preimage points of a double point are not canonically ordered.

A \emph{self-intersection homomorphism}
\begin{equation}
  \mu\colon \pi_3(M) \to \Z_2[T_M]
  \label{equation:homomorphism-mu}
\end{equation}
is defined similarly, by perturbing a representative $S^3 \to M\subset M\times I^2$ to a topological generic immersion and counting self-intersections as in~\cite[Appendix~A]{Cha-Kim:2023-1}.
It turns out that the value of $\mu(\widehat H)$ in \cref{equation:self-intersection-mu} modulo the image $\mu(\pi_3(M))$ is determined by $(D_0,D_1)$, independent of the choices of the homotopy and the perturbation.
The \emph{topological Freedman--Quinn invariant} for $(D_0,D_1)$ is defined by
\[
  \FQ_M(D_0,D_1) = \mu(\widehat H)\in \Z_2[T_M]/\mu(\pi_3(M)).
\]

We remark that if $M$, $D_0$ and $D_1$ are smooth, then the homomorphism $\mu$ and the invariant $\FQ_M(D_0,D_1)$ agree with the smooth versions defined in~\cite{Schneiderman-Teichner:2022-1,Kosanovic-Teichner:2024-2}.
We also remark that the definition of the topological Freedman--Quinn invariant does not require a choice of an almost smooth structure, while the topological Dax invariant does.

\begin{proof}[Proof of \cref{proposition:dax-concordant-disks}, continued]
  We have already observed that it suffices to prove \cref{item:dax-reduction-fq,item:fq-concordance-disk} in the topological category.
  We restate \cref{item:dax-reduction-fq} in a more precise form below.
  Let $q\colon \Z[\pi_1(M) \sm 1]^\sigma \to \Z_2[T_M]$ be the quotient map that the isomorphism in \cref{equation:self-intersection-mu} gives rise to.
  Then the statement \cref{item:dax-reduction-fq} means that $q$ induces a quotient map
  \[
    \overline q \colon D_M = \Z[\pi_1(M) \sm 1]^\sigma / d(\pi_3(M)) \to \Z_2[T_M]/\mu(\pi_3(M))
  \]
  with $\Ker \overline q = \Sigma_M$ and
  \begin{equation}
    \overline q(\Dax_M(D_0,D_1)) = \FQ_M(D_0,D_1).
    \label{equation:dax-reduces-to-fq}
  \end{equation}
  Denote temporarily by $d_{(M\sm p_0)_\Gamma}$ and $\mu_{M\sm p_0}$ the homomorphisms on $\pi_3(M\sm p_0)$ defined as in \cref{equation:dax-homomorphism,equation:homomorphism-mu}, respectively.
  By applying \cite[Lemma~4.10]{Kosanovic-Teichner:2024-2} to the smooth manifold $(M\sm p_0)_\Gamma$, we obtain $q\circ d_{(M\sm p_0)_\Gamma} = \mu_{M\sm p_0}$.
  By definition, $d_{(M\sm p_0)_\Gamma}$ and $\mu_{M\sm p_0}$ induce the homomorphisms $d$ and $\mu$ on the quotient $\pi_3(M)$ of $\pi_3(M\sm p_0)$.
  Thus, $q\circ d_M = \mu_M$ on~$\pi_3(M)$.
  From this, it follows that $q$ induces $\overline q$ with $\Ker \overline q = \Sigma_M$, as claimed.

  To verify \cref{equation:dax-reduces-to-fq}, we proceed as follows.
  As in the setting of \cref{definition:topological-dax}, choose a homotopy $H\colon D_0\simeq D_1$ rel $\partial$ in $M\sm p_0$ and a topological isotopy $h_i\colon D_i \approx D_i'$ ($i=0,1$) rel $\partial$ to a smooth disk $D_i'$ in a smooth stabilization $M' = (M\sm p_0)_\Gamma \csum k(S^2\times S^2)$.
  Let $H' = h_0^{-1}H h_1 \colon D_0' \simeq D_1'$ be the concatenated homotopy.
  Then $\Dax_M(D_0,D_1) = \dax_{M'}(H')$ by definition.
  In the smooth manifold $M'$, \cite[Lemma~4.10]{Kosanovic-Teichner:2024-2} yields $q(\dax_{M'}(H')) = \mu_{M'}(\widehat H')$
  where $\mu_{M'}$ is the self-intersection in \cref{equation:self-intersection-mu} for~$M'$.
  Since $h_0^{-1}$ and $h_1$ are isotopies, they make no contribution to~$\mu_{M'}$, and thus $\mu_{M'}(\widehat H') = \mu_M(\widehat H)$.
  By definition, $\FQ_M(D_0,D_1) = \mu_M(\widehat H)$.
  Combining these, we obtain \cref{equation:dax-reduces-to-fq}.

  The statement \cref{item:fq-concordance-disk} in the proof of \cref{proposition:dax-concordant-disks} is an immediate consequence of this definition of~$\FQ_M$.
  In fact, for a concordance $C\colon D^2\times I \hookrightarrow M\times I$ between $D_0$ and $D_1$, if we define a homotopy $H$ to be the composition of $C$ and the projection onto~$M$, then the associated map $\widehat H$ is homotopic to the embedding $D^2\times I \to M\times I^2$, $(x,t) \mapsto (C(x,t),\frac12)$, whose self-intersection $\mu$ automatically vanishes.

  This completes the proof of \cref{proposition:dax-concordant-disks}.
\end{proof}

\begin{remark}
  \label{remark:D_M/Sigma_M-is-Z_2-vector-space}
  As we described in the above proof, \cite[Lemma~4.10]{Kosanovic-Teichner:2024-2} yields an isomorphism
  \[
    D_M / \Sigma_M \cong \Z_2[T_M]/\mu(\pi_3(M)).
  \]
  In particular, $D_M/\Sigma_M$ is a $\Z_2$-vector space.
\end{remark}

\section{Topological isotopy invariants of homeomorphisms}
\label{section:invariant-of-homeo}

Throughout this section, $M$ denotes a 4-manifold with nonempty boundary.
Fix an almost smooth structure $\Gamma$ on~$M$, i.e., $\Gamma$ is a smooth structure on the complement $M\sm p_0$ of an interior point~$p_0$.

Fix a neat disk $D$ topologically embedded in~$M$.
Define
\begin{equation}
  H_D(M) = \{ f \in \Homeo(M) \mid f_* = \id \text{~on~}\pi_1(M) \text{~and~} f(D) \simeq D \rel \partial \}.
  \label{equation:H_D(M)}
\end{equation}
It is readily seen that $\Homeoh(M) \subset H_D(M)\subset \Homeo(M)$ and $\pi_0\Homeoh(M) \subset \pi_0 H_D(M) \subset \pi_0 \Homeo(M)$ as subgroups.

\begin{definition}
  \label{definition:DeltaD}
  For $f\in H_D(M)$, define
  \[
    \DeltaD(f) = \Delta_{D,\Gamma}(f) = \Dax_{M,\Gamma}(D, f(D)) \in D_M.
  \]
\end{definition}

We write $\DeltaD(f)$ for $\Delta_{D,\Gamma}(f)$, omitting $\Gamma$ from the notation, when the choice of $\Gamma$ has been fixed in the context.
All results on $\Delta_{D,\Gamma}$ in this paper hold for any choice of~$\Gamma$.

In what follows, an isotopy rel $\partial$ means an isotopy that fixes the boundary pointwise.

By the isotopy invariance of the Dax invariant (\cref{proposition:top-dax-properties}~\cref{item:top-dax-isotopy-invariance}), it readily follows that $\DeltaD(f)$ is invariant under isotopy of $f$ rel~$\partial$.
That is, $\DeltaD$ induces a well-defined map $\DeltaD \colon \pi_0 H_D(M) \to D_M$.

\subsection{The values of \texorpdfstring{$\DeltaD$}{\Delta\_D} for homeomorphisms homotopic to the identity}
\label{subsection:DeltaD-on-homeoh}

Note that for \emph{every} neat disk $D\subset M$, we have $\Homeoh(M) \subset H_D(M) \subset \Homeo(M)$, and thus $\DeltaD(f)$ is always defined for $f\in \Homeoh(M)$.
The goal of this subsection is to show the following result, which gives a restriction on the values of $\DeltaD$ on $\Homeoh(M)$.

\begin{proposition}
  \label{proposition:DeltaDh-on-homeoh}
  If a homeomorphism $f$ of a 4-manifold $M$ is properly homotopic to $\id_M$ rel~$\partial M$, then $\DeltaD(f) \in \Sigma_M$.
\end{proposition}

\begin{proof}
  By \cref{lemma:fq-homotopy-trivial-homeomorphism} below, $\FQ(D,f(D))=0$.
  Since $\FQ$ is the reduction of $\Dax$ modulo $\Sigma_M$, it follows that $\DeltaD(f) = \Dax(D,f(D)) \in \Sigma_M$.
\end{proof}

\begin{lemma}
  \label{lemma:fq-homotopy-trivial-homeomorphism}
  If $f$ is a homeomorphism of a 4-manifold $M$ properly homotopic to $\id_M$ rel~$\partial M$, then $\FQ_M(D,f(D))=0$ for any neat disk $D\subset M$.
\end{lemma}

\begin{proof}
  Fix a collar $N(\partial M)\cong \partial M\times [0,\epsilon)$ of $\partial M$.
  By isotopy, we may assume that $f=\id$ on $N(\partial M)$ and $D\cap N(\partial M)$ is the product $\partial D \times [0,\epsilon)$.
  Choose a proper homotopy $H\colon M\times I\to M$ from $\id_M$ to $f^{-1}$ rel~$N(\partial M)$ such that $H(\inte(M)\times I) \subset \inte(M)$, parametrized so that it is stationary near $M\times\partial I$.
  There is a topological normal $\R^2$-bundle of $D$ in $M$ \cite[\S9.3]{Freedman-Quinn:1990-1}.
  By topological transversality, a compactly supported homotopy makes $H$ transverse to the normal bundle of $D$, with $N(\partial M)\times I$ and a neighborhood of $M\times \partial I$ fixed \cite[Section~9.6]{Freedman-Quinn:1990-1}.
  The resulting $H$ is proper.
  Then $H^{-1}(D)$ is a compact, orientable, neat $3$-manifold in $M\times I$ with trivial normal bundle, bounded by $S=(D\times\{0\})\cup(\partial D\times I)\cup(f(D)\times\{1\})$.
  Let $W$ be the component of $H^{-1}(D)$ containing~$S$.
  Note that the inclusion induces the trivial map $\pi_1(W)\to\pi_1(M\times I)$, since $H(W)$ lies in the contractible subspace $D\subset M$ and $H_*\colon \pi_1(M\times I) \to \pi_1(M)$ is an isomorphism.
  View $W\subset M\times I = M\times I\times \frac12 \subset M\times I^2$.

  There is a framed link $L\subset\inte W$ such that surgery on $W$ along $L$ gives a 3-ball, by Lickorish--Wallace \cite{Lickorish:1962-1}.
  Each component of $L$ is null-homotopic in $M\times I^2$, since $\pi_1(W)\to\pi_1(M\times I)$ is trivial.
  Thus, by topological transversality, there exist disjoint embedded disks $\Delta_i$ in the interior of $M\times I^2$ such that $\bigsqcup \partial\Delta_i = L\subset W$, $\inte\Delta_i \cap W =\emptyset$ and the product $\partial\Delta_i\times [\frac12,\frac12+\epsilon] \subset (M\times I)\times I$ is a collar of~$\partial\Delta_i$.

  We claim that we can perform surgery on $W$ ambiently in $M\times I^2$ along the disks~$\Delta_i$.
  Details are as follows.
  The disk $\Delta_i\subset M\times I^2$ has a normal $\R^4$-bundle $\nu{\Delta_i}$ \cite[\S 9.4]{Freedman-Quinn:1990-1}.
  The bundle $\nu{\Delta_i}$ is trivial since $\Delta_i$ is contractible.
  By the arrangement of the collar of $\partial \Delta_i\subset \Delta_i$, the normal bundle of $\partial\Delta_i$ in $W$ is a subbundle of $\nu{\Delta_i}|_{\partial\Delta_i}$.
  Thus, the surgery framing on $L$ determines a section, on $\partial\Delta_i$, of the associated trivial bundle with fiber $V^{\TOP}_{4,2}=\TOP(4)/\TOP(4,2)$.
  Here $\TOP(4)$ is the group of homeomorphisms $\R^4\to \R^4$ fixing the origin, and $\TOP(4,2) = \{h\in \TOP(4)\mid h=\id$ on $\R^2\times 0 \subset \R^4\}$.
  It is known that the natural maps $\O(2)=\O(4,2) \to \TOP(4,2)$ and $\O(4)\to \TOP(4)$ are 2-connected;
  see~\cite[proof of Theorem~2.1]{Salvatore-Turchin:2026-1}, \cite[Corollary 2.2.3]{Quinn:1982-1} and~\cite[Theorem~8.3B]{Freedman-Quinn:1990-1}.
  It follows that $\pi_1 V^{\TOP}_{4,2} \cong \pi_1 V^{\O}_{4,2}$, where $V^{\O}_{4,2} = \O(4)/\O(2)$ is the smooth Stiefel space.
  Since $V^{\O}_{4,2}$ is simply connected, we have $\pi_1 V^{\TOP}_{4,2}=0$.
  Thus, the section on $\partial\Delta_i$ extends over~$\Delta_i$, and there is a thickening $\Delta_i\times D^2\hookrightarrow M\times I^2$ which allows us to perform the surgery on $W$ ambiently in $M\times I^2$.

  The result of ambient surgery is a neat 3-ball $C$ embedded in $M\times I^2$ with $\partial C=S\times \frac12$.
  By the Alexander trick, the given parametrization $\partial(D^2\times I)\cong S$ extends to a homeomorphism $c\colon D^2\times I\to C$.
  Write $c(x,t)=(K(x,t),u(x,t),v(x,t))\in M\times I^2$.
  Note that $u(x,t)=t$ and $v(x,t)=\frac12$ on $\partial(D^2\times I)$, and $K\colon D^2\times I \to M$ is a homotopy $D\simeq f(D)$ rel $\partial$.
  The track $\widehat K\colon D^2\times I\to M\times I^2$ of $K$ defined by $\widehat K(x,t)=(K(x,t),t,\frac12)$ is homotopic to $c$, rel $\partial(D^2\times I)$, via a straight-line homotopy between $(t,\frac12)$ and $(u,v)$.
  By the definition of $\FQ_M$ and by the homotopy invariance of the self-intersection~$\mu$, as in \cref{subsubsection:freedman-quinn-top-disks} and \cite[Appendix~A]{Cha-Kim:2023-1}, we have $\FQ_M(D,f(D)) = \mu(\widehat K) = \mu(c)$.
  Since $c$ is an embedding, $\mu(c)=0$.
\end{proof}

\subsection{Additivity of \texorpdfstring{$\DeltaD$}{\Delta\_D} under composition}
\label{subsection:DeltaD-addivity}

\begin{proposition}
  \label{proposition:DeltaD-homomorphism}
  The map $\DeltaD \colon \pi_0 H_D(M) \to D_M$ is a homomorphism.
\end{proposition}

\begin{proof}
  For $f$, $g\in H_D(M)$, we have
  \[
    \begin{aligned}
      \Dax(D,fg(D)) &= \Dax(D,f(D)) + \Dax(f(D),fg(D))
      && \text{by \cref{proposition:top-dax-properties}~\cref{item:top-dax-additivity}}
      \\
      &= \Dax(D,f(D)) + \Dax(D,g(D))
      && \text{by \cref{lemma:top-dax-pi1-trivial-homeomorphism} below}
    \end{aligned}
  \]
  where $\Dax$ designates $\Dax_{M,\Gamma}$.
  From this, it follows that
  \[
    \DeltaD(f \circ g) = \DeltaD(f) + \DeltaD(g).
    \qedhere
  \]
\end{proof}

\begin{lemma}
  \label{lemma:top-dax-pi1-trivial-homeomorphism}
  Let $D_0$, $D_1\subset M$ be topological neat disks with common boundary, which are homotopic rel boundary.
  Then, for $f\in\Homeo(M)$ inducing $f_*=\id$ on~$\pi_1(M)$, we have
  \[
    \Dax_{M,\Gamma}(f(D_0),f(D_1)) = \Dax_{M,\Gamma}(D_0,D_1).
  \]
\end{lemma}

\begin{remark}
  In the above proof of \cref{proposition:DeltaD-homomorphism}, the use of \cref{lemma:top-dax-pi1-trivial-homeomorphism} is essential.
  From the definition, it can be seen that $\Dax_{M,\Gamma}(f(D_0),f(D_1)) = f_* \Dax_{M,f^*\Gamma}(D_0,D_1)$ holds, where $f^*\Gamma$ is the pullback of the almost smooth structure $\Gamma$ under~$f$.
  This gives $\Dax_{M,\Gamma}(f(D),fg(D)) = \Dax_{M,f^*\Gamma}(D,g(D))$.
  However, the almost smooth structure $f^*\Gamma$ is not necessarily equivalent to $\Gamma$ in general.
\end{remark}

\begin{remark}
  After the results of this paper were proved and announced, Lin, Xie and Zhang proved the topological invariance of the smooth Dax isomorphism \cite{Lin-Xie-Zhang:2026-2}.
  Also, a homological definition of the Dax invariant was developed in~\cite{Cha:2026-2}.
  In the setting of \cref{lemma:top-dax-pi1-trivial-homeomorphism}, their methods can be used to prove that $\Dax_{M,f^*\Gamma}(D_0,D_1) = \Dax_{M,\Gamma}(D_0,D_1)$, yielding an alternative proof of the lemma.
\end{remark}

To prove \cref{lemma:top-dax-pi1-trivial-homeomorphism}, we use the following.

\begin{lemma}[Topological computation of Dax self-intersection]
  \label{lemma:dax-topological-computation}
  Let $X$ be a smooth oriented 4-manifold and let $H\colon I^3\to X$ be a homotopy satisfying the boundary condition required to define $\dax(H)$ in~\cref{equation:dax-self-intersection}.
  We allow $H$ to be non-smooth.
  Suppose that the honest track $\overline H\colon I^3\to X\times I^2$ is a topological embedding except at finitely many transverse double points.
  Then
  \[
    \dax(H)=\sum_{g(p)\ne 1}\epsilon(p)g(p) \in \Z[\pi_1(X)\sm 1]
  \]
  where $p$ varies over double points of $\overline H$ and $\epsilon(p)$ and $g(p)$ are defined as in~\cref{equation:dax-self-intersection}.
\end{lemma}

\begin{proof}
  For brevity, let $\pr_1\colon I^3 \to I$ be the projection $\pr_1(s,t,u)=s$.
  For each double point $p=\overline H(x_p)=\overline H(y_p)$ with $\pr_1(x_p) < \pr_1(y_p)$, choose an open 6-ball neighborhood $U_p\subset X\times I^2$ of $p$, and choose closed smooth 3-ball neighborhoods $B_p$,~$C_p\subset\inte I^3$ of $x_p$,~$y_p$, respectively, such that all of the $B_p$ and $C_p$ are pairwise disjoint, $\pr_1(x) < \pr_1(y)$ whenever $x\in B_p$ and $y\in C_p$, $\overline H|_{B_p}$, $\overline H|_{C_p}$ are embeddings into $U_p$, and $\overline H(B_p) \cap \overline H(C_p) = \{p\}$.
  Orient $B_p$, $C_p$ and $U_p$ by the orientations of $I^3$ and~$X\times I^2$.
  Choose a sufficiently small neighborhood $N$ of the diagonal $\Delta_{I^3} \subset I^3\times I^3$ which is disjoint from all $B_p\times C_p$ and $C_p\times B_p$.

  Using Dax's smooth perturbation \cite[Chapters III and~VI]{Dax:1972-1}, choose a small homotopy $\{H_r\}_{0\le r\le 1}$ rel $\partial I^3$ from $H_0=H$ to a smooth map $H_1=F$ such that the honest track $\overline F$ is a smooth generic immersion.
  For a double point $q$ of $\overline F$, write $q=\overline F(z_q)=\overline F(w_q)$ with $\pr_1(z_q)<\pr_1(w_q)$.
  By choosing $F$ sufficiently close to $H$, we may assume that $(z_q,w_q)$ lies either in $N$ or in $\inte(B_p\times C_p)$ for some~$p$.

  Suppose that $(z_q,w_q)\in N$.
  Write $z_q=(s_0,t,u)$ and $w_q=(s_1,t,u)$ where $s_0 < s_1$.
  Choosing $N$ sufficiently small, we may assume that the loop $F|_{[s_0,s_1]\times(t,u)}$ is contained in a contractible coordinate neighborhood in~$X$.
  Then, $g(q)=1$ whenever $(z_q,w_q)\in N$.

  Let $S_p$ be the set of double points $q$ of $\overline F$ such that $(z_q,w_q) \in \inte(B_p\times C_p)$.
  For all $q\in S_p$, we have $g(q)=g(p)$, since the homotopy is sufficiently small.
  Furthermore, the following holds.

  \begin{claim}
    For each $p$, $\sum_{q\in S_p}\epsilon(q)=\epsilon(p)$.
  \end{claim}

  Assuming the claim, we have
  \[
    \dax(H) = \dax(F) = \sum_{g(q)\ne 1} \epsilon(q)g(q)
    = \sum_{g(p)\ne 1}\;\sum_{q\in S_p}\epsilon(q)g(q)
    = \sum_{g(p)\ne 1}\epsilon(p)g(p).
  \]
  This proves the desired conclusion.

  To prove the claim, we use a homological intersection number defined below.
  For maps $\alpha$,~$\beta\colon D^3\to \R^6$ satisfying
  \begin{equation}
    \alpha(\partial D^3)\cap\beta(D^3)=\emptyset,\quad \alpha(D^3)\cap\beta(\partial D^3)=\emptyset,
    \label{equation:boundary-condition-for-intersection}
  \end{equation}
  the product map
  \[
    \alpha\times\beta \colon (D^3\times D^3,\partial(D^3\times D^3)) \to (\R^6\times \R^6,(\R^6\times \R^6)\sm \Delta_{\R^6})
  \]
  is defined, where $\Delta_{\R^6} \subset \R^6\times \R^6$ is the diagonal.
  The map $(\R^6,\R^6\sm 0) \to (\R^6\times \R^6, (\R^6\times \R^6) \sm \Delta_{\R^6})$ given by $u\mapsto (u,0)$ sends the standard local orientation of $\R^6$ to a generator $\zeta$ of $H_6(\R^6\times \R^6, (\R^6\times \R^6) \sm \Delta_{\R^6}) \cong \Z$.
  Define $I(\alpha,\beta)\in \Z$ by
  \[
    (\alpha\times\beta)_*[D^3\times D^3,\partial(D^3\times D^3)] = -I(\alpha,\beta)\cdot \zeta.
  \]
  The sign convention is chosen so that $I(\id_{D^3\times0},\id_{0\times D^3})=1$.
  The value of $I(\alpha,\beta)$ is unchanged under homotopies through maps satisfying the boundary condition~\cref{equation:boundary-condition-for-intersection}.
  If $\alpha$ and $\beta$ are topologically transverse with finitely many intersections, then $I(\alpha,\beta)$ is the sum of their ordered intersection signs.

  In our case, recall that $\{H_r\colon I^3 \to X\}$ is a homotopy between $H_0=H$ and $H_1=F$.
  Fix orientation-preserving identifications of $B_p$ and $C_p$ with $D^3$ and $U_p$ with~$\R^6$.
  Let $\alpha_r=\overline{H_r}|_{B_p}$ and $\beta_r=\overline{H_r}|_{C_p}$. Since $\alpha_0$ and $\beta_0$ satisfy~\cref{equation:boundary-condition-for-intersection}, we may assume that $\alpha_r$ and $\beta_r$ send $B_p$ and $C_p$ to $U_p$ and satisfy~\cref{equation:boundary-condition-for-intersection} for all~$r$, by choosing a sufficiently small homotopy~$\{H_r\}$.
  Thus, $I(\alpha_r,\beta_r)$ is defined.
  Furthermore, by the transverse intersection interpretation and homotopy invariance,
  \[
    \sum_{q\in S_p}\epsilon(q)
    = I(\alpha_1,\beta_1)
    = I(\alpha_0,\beta_0)
    = \epsilon(p).
  \]
  This proves the claim.
\end{proof}

\begin{proof}[Proof of~\cref{lemma:top-dax-pi1-trivial-homeomorphism}]
  Let $p_0$ be the singular point of~$\Gamma$.
  By isotopy invariance, we may assume that $D_0,D_1$ avoid~$p_0$.
  Choose a homotopy $H\colon D_0 \simeq D_1 \rel \partial$ in $M\sm p_0$.
  We may assume that $f=\id$ on an open 4-ball neighborhood $U$ of~$p_0$ such that $U\cap H(I^3) = \emptyset$, by isotopy of $f$ rel~$\partial M$.
  Note that $fH\colon f(D_0) \simeq f(D_1)$ is a homotopy rel $\partial$ in $M\sm U$.
  Let $X=(M\sm p_0)_\Gamma\csum k(S^2\times S^2)$, where the connected sum is formed at a smooth 4-ball contained in $U\sm p_0$.
  Since $f=\id$ on $U$, $f$ extends by $\id_{k(S^2\times S^2)}$ to a homeomorphism of $X$, which we denote by~$f$ for brevity.
  This extension induces the identity on $\pi_1(X)=\pi_1(M)$.
  Using \cite[Theorem~D]{Cha-Kim:2023-1}, choose $k$ such that there are topological isotopies $h_i\colon D_i\approx D_i'$ and $k_i\colon f(D_i')\approx D_i''$ rel~$\partial$ in $X$, where $D_i'$, $D_i''\subset X$ are smooth neat disks.

  Choose a smooth homotopy $G\colon D_0'\simeq D_1'$ rel $\partial$ such that $G$ is homotopic to the concatenation $h_0^{-1} H h_1$ rel~$\partial I^3$ and the honest track $\overline G$ is a generic immersion, using Dax's perturbation.
  By definition, $\Dax(D_0,D_1)=\dax(G)$.

  In what follows, when a concatenation is formed, make the homotopies involved stationary near endpoints.
  Let $P = k_0^{-1} (f G) k_1\colon D_0'' \simeq D_1''$, a homotopy rel~$\partial$.
  Then $P$ is homotopic, rel $\partial I^3$, to the homotopy $Q=(k_0^{-1}(fh_0^{-1}))(fH)((fh_1)k_1)$.
  Note that $(fh_i)k_i\colon f(D_i) \approx D_i''$ is a topological isotopy to the smooth disk~$D_i''$ in~$X$.
  Thus, by the definition of the topological Dax invariant and by the homotopy invariance of $\dax$, we have $\Dax(f(D_0),f(D_1))=\dax(Q)=\dax(P)$.

  Thus, it suffices to show that $\dax(P)=\dax(G)$.
  Since $k_0$ and $k_1$ are isotopies, $\overline P$ is a topological immersion whose self-intersections are finitely many topologically transverse double points in bijective correspondence with those of~$\overline{G}$.
  Since $f$ preserves the $\pi_1$ elements and orientation, corresponding double points $q$ of $\overline P$ and $p$ of~$\overline G$ satisfy $g(q) = g(p)$ and $\epsilon(q) = \epsilon(p)$.
  By \cref{lemma:dax-topological-computation}, it follows that $\dax(P) = \dax(G)$.
\end{proof}

\begin{proof}[Proof of \cref{theorem:intro-DeltaDh-homomorphism}]
  By \cref{proposition:DeltaD-homomorphism}, $\DeltaD$ is additive on $\pi_0 \Homeoh(M) \subset \pi_0 H_D(M)$ under composition.
  By \cref{proposition:DeltaDh-on-homeoh}, $\DeltaD(\pi_0 \Homeoh(M)) \subset \Sigma_M$.
  Thus $\DeltaD$ induces a homomorphism $\pi_0\Homeoh(M) \to \Sigma_M$.
\end{proof}

\begin{corollary}
  \label{corollary:homomorphism-modulo-homotopy}
  The homomorphism $\DeltaD\colon \pi_0 H_D(M) \to D_M$ induces a homomorphism
  \[
    \tDeltaD\colon \pi_0 H_D(M) / \pi_0 \Homeoh(M) = H_D(M) / ({\simeq}\rel \partial M) \to D_M/\Sigma_M.
  \]
\end{corollary}

\begin{proof}
  By \cref{theorem:intro-DeltaDh-homomorphism}, $\DeltaD$ sends the subgroup $\pi_0 \Homeoh(M) \subset \pi_0 H_D(M)$ to~$\Sigma_M$.
  The conclusion readily follows.
\end{proof}

\section{Realization of the values of the invariant~\texorpdfstring{$\DeltaD$}{\Delta\_D}}
\label{section:realization}

In this section, we study the values of $\DeltaD$ realized by homeomorphisms, in the presence of a geometric dual.
When the underlying 4-manifold is smooth or stably smoothable, we realize these values by diffeomorphisms or stably smoothable homeomorphisms, respectively.
We begin with our conventions on stable smoothability.

\paragraph{Stably smoothable 4-manifolds and almost smooth structures}

We say that a 4-manifold $M$ is \emph{stably smoothable} if $M\csum k(S^2\times S^2)$ admits a smooth structure $\Theta$ for some $k \ge 0$.
We call $\Theta$ a \emph{stable smooth structure} for~$M$.
By~\cite[10.2B]{Freedman-Quinn:1990-1}, $M$ is stably smoothable if and only if $\ks(M)=0$.
In this paper, a stably smoothable 4-manifold $M$ is always equipped with a fixed stable smooth structure~$\Theta$.

We say that a homeomorphism $f\in\Homeo(M)$ on a stably smoothable 4-manifold $M$ is \emph{stably smoothable} if the following holds:
after isotoping $f$ so that $f = \id$ near the point at which the connected sum is formed, there exists $r\ge k$ such that the homeomorphism $f\csum \id_{r(S^2\times S^2)}$ on the smooth 4-manifold $M \csum r(S^2\times S^2) = (M \csum k(S^2\times S^2))_\Theta \csum (r-k)(S^2\times S^2)$ is topologically isotopic rel $\partial$ to a diffeomorphism.
When $f$ is stably smoothable, we also say that the mapping class $[f] \in \pi_0\Homeo(M)$ is stably smoothable.
When it is necessary to specify the choice of $\Theta$ explicitly, we say that a homeomorphism $f$ (or its mapping class $[f]$) is \emph{stably smoothable with respect to~$\Theta$}.

By convention, for a stably smoothable 4-manifold $M$ equipped with a stable smooth structure $\Theta$ on $M\csum k(S^2\times S^2)$, we fix an almost smooth structure $\Gamma$ on $M$ as follows.
View $M\sm p_0$ as an open subset of $M\csum k(S^2\times S^2)$, where $p_0$ is an interior point of~$M$ at which the connected sum is performed.
Use the restriction $\Gamma=\Theta|_{M \sm p_0}$ as an almost smooth structure on $M$, in particular for the purpose of defining the invariants $\Dax=\Dax_{M,\Gamma}$ and $\Delta_D = \Delta_{D,\Gamma}$.

As a special case, for a smooth 4-manifold $M$, we use the given smooth structure as a stable smooth structure, and use its restriction to $M\sm p_0$ as an almost smooth structure.

\subsection{Pseudoisotopy and realization of values in~\texorpdfstring{$\Sigma_M$}{\Sigma\_M}}
\label{subsection:pseudoisotopy-and-disk-concordance}

The goal of this subsection is to prove \cref{theorem:intro-realization-Sigma_M}, which asserts the following:
for a 4-manifold $M$ with nonempty boundary and a neat disk $D\subset M$ with a geometric dual in $\partial M$, there is an abelian subgroup $H$ in $\pi_0 \Homeopi(M)$ that surjects onto~$\Sigma_M$ under the homomorphism $\DeltaDh \colon \pi_0 \Homeoh(M) \to \Sigma_M$.
In addition, if $M$ is smooth, then we may assume that all elements of $H$ are represented by diffeomorphisms that commute up to smooth isotopy;
if $M$ is stably smoothable, then we may assume that all elements of $H$ are represented by stably smoothable homeomorphisms.

Our proof of \cref{theorem:intro-realization-Sigma_M} is built on the following result.
Recall that $\Sigma_M$ is the subgroup of $D_M = \Z[\pi_1(M)\sm 1]^\sigma/d(\pi_3(M))$ generated by $\{g+g^{-1} \mid g\in \pi_1(M),\; g\ne 1\}$.
Let $\Sigma_M^\partial \subset \Sigma_M$ be the subgroup generated by $\{j_*(g)+j_*(g)^{-1} \mid g\in \pi_1(\partial M),\; j_*(g)\ne 1\}$, where $j_*\colon \pi_1(\partial M) \to \pi_1(M)$ is the inclusion-induced homomorphism.

\begin{proposition}
  \label{proposition:concordance-disk-pseudoisotopy}
  Let $M$ be a 4-manifold with nonempty boundary.
  Suppose that $D_0$ and $D_1$ are neat disks embedded in $M$ with $\partial D_0 = \partial D_1$, which are homotopic rel~$\partial$ and have a geometric dual in~$\partial M$.
  Then the following are equivalent.

  \begin{enumerate}
    \item\label{item:disks-concordant}
    The disks $D_0$ and $D_1$ are concordant.
    \item\label{item:dax-symmetrized-value}
    The value of $\Dax(D_0,D_1)$ lies in the subgroup $\Sigma_M \subset D_M$.
    \item\label{item:disks-pseudoisotopy}
    For some homeomorphism $f_{D_0,D_1}$ of $M$ pseudoisotopic to $\id_M$, $D_1 = f_{D_0,D_1}(D_0)$.
    \item\label{item:disks-homotopy}
    For some homeomorphism $g_{D_0,D_1}$ of $M$ properly homotopic to $\id_M$ rel $\partial M$, $D_1 = g_{D_0,D_1}(D_0)$.
  \end{enumerate}
  When \cref{item:disks-concordant,,item:dax-symmetrized-value,,item:disks-pseudoisotopy,,item:disks-homotopy} are satisfied, we may assume that $f_{D_0,D_1}$ has compact support and that any two of these homeomorphisms $f_{D_0,D_1}$ and $f_{D_0',D_1'}$ commute up to isotopy.
  Furthermore, if $\Dax(D_0,D_1)\in \Sigma_M^\partial$, then $[f_{D_0,D_1}]$ is central in $\pi_0 \Homeo(M)$.

  If $M$ is stably smoothable, then we may assume that $f_{D_0,D_1}$ is stably smoothable.

  If $M$, $D_0$ and $D_1$ are smooth, then the conclusions hold in the smooth category:
  $f_{D_0,D_1}$ is a diffeomorphism smoothly pseudoisotopic to~$\id_M$, and $f_{D_0,D_1}$ and $f_{D_0',D_1'}$ commute up to smooth isotopy;
  furthermore, if $\Dax(D_0,D_1)\in \Sigma_M^\partial$, then $[f_{D_0,D_1}]$ is central in $\pi_0 \Diff(M)$.
\end{proposition}

\begin{remark}
  In \cref{proposition:concordance-disk-pseudoisotopy}, if $M$ is smooth, then $D_0$ and $D_1$ are topologically isotopic to smooth disks, by the light bulb smoothing theorem~\cite[Theorem~C]{Cha-Kim:2023-1}.
\end{remark}

\begin{proof}[Proof of~\cref{proposition:concordance-disk-pseudoisotopy}]
  By \cref{proposition:dax-concordant-disks}, \cref{item:disks-concordant}~$\Rightarrow$~\cref{item:dax-symmetrized-value}.
  Obviously \cref{item:disks-pseudoisotopy}~$\Rightarrow$~\cref{item:disks-homotopy}.
  By \cref{proposition:DeltaDh-on-homeoh}, \cref{item:disks-homotopy}~$\Rightarrow$~\cref{item:dax-symmetrized-value}.
  Also, the implication \cref{item:disks-pseudoisotopy}~$\Rightarrow$~\cref{item:disks-concordant} is a consequence of the definitions:
  if $F\colon M\times I \to M\times I$ is a pseudoisotopy from $\id_M$ to $f$, i.e., $F=\id$ on $M\times 0 \cup \partial M \times I$ and $F(x,1)=(f(x),1)$ for $x\in M$, and if $i\colon D_0\times I \hookrightarrow M\times I$ is the inclusion, then the composition $F\circ i\colon D_0 \times I \hookrightarrow M\times I$ is a concordance from $D_0$ to $D_1=f(D_0)$.

  Therefore, to show that \cref{item:disks-concordant,,item:dax-symmetrized-value,,item:disks-pseudoisotopy,,item:disks-homotopy} are equivalent, it suffices to show \cref{item:dax-symmetrized-value}~$\Rightarrow$~\cref{item:disks-pseudoisotopy}.

  We first treat the smooth case.
  Suppose that $M$ is smooth and $D_0$ and $D_1$ are smoothly embedded neat disks in~$M$, such that $v = \Dax(D_0,D_1)$ lies in~$\Sigma_M$.
  We will construct $f=f_{D_0,D_1}\in \Diffpi(M)$ satisfying $f(D_0)=D_1$, by using barbell maps defined in~\cite{Budney-Gabai:2019-1}.

  We begin with a brief review of barbell maps.
  A \emph{barbell} $\bB$ in a 4-manifold $M$ is an embedded image of $(S^2\times D^2) \bcsum (S^2 \times D^2)$ in~$M$.
  We view a barbell $\bB$ as a regular neighborhood of an embedded 2-complex $B = S_1 \cup b\cup S_2 \subset M$ where $S_1$ and $S_2$ are the images of the core 2-spheres of the two $S^2\times D^2$ factors and $b$ is an arc joining $S_1$ and~$S_2$.
  Abusing the terminology, we often refer to the 2-complex $B$ as a \emph{barbell}, and write $N(B)=\bB$.
  The spheres $S_1$ and $S_2$ are called the \emph{cuffs} and the arc $b$ is called the \emph{bar}.
  In~\cite[Section~5]{Budney-Gabai:2019-1}, Budney and Gabai defined a diffeomorphism $\phi_B \in \Diff(M)$ supported in $N(B)$, which is called a \emph{barbell map}.
  For a detailed description, see~\cite[Section~5]{Budney-Gabai:2019-1} and~\cite[Section~2]{Budney-Gabai:2025-1}.

  We will use the following fact about a barbell map.
  Let $R$ be a surface in $M$ which intersects the cuff $S_1$ transversely at exactly one point and does not meet $b \cup S_2$.
  Extend $b$ to an arc $\tau = \tau_0 * b$ by attaching an arc $\tau_0$ on $S_1$ from $R\cap S_1$ to $b\cap S_1$.
  Then the surface $\phi_B(R)$ is obtained by tubing $R$ to the cuff $S_2$ along the arc~$\tau$.
  For a proof, refer to \cite[Construction~5.3, Figure~30]{Budney-Gabai:2019-1}.

  We apply the above fact to our case, to prove the following claim:
  for all non-identity elements $g\in \pi_1(M)\sm \{1\}$ and for all $\epsilon \in \{1,-1\}$, there is a barbell map $\phi_B$ such that $\Dax(D_0, \phi_B(D_0)) = \epsilon\cdot (g+g^{-1})$.

  To show the claim, we proceed as follows.
  Let $S_1\subset \inte M$ be a parallel push-off of the geometric dual sphere $G\subset \partial M$ of the given neat disk~$D_0$, which intersects $D_0$ transversely at exactly one point.
  Choose a small embedded 3-disk $D^3 \subset \inte M$ disjoint from $D_0 \cup S_1$.
  Choose an arc $b_0$ from $S_1$ to a point $x_0 \in \inte D^3$ such that $(D_0 \cup S_1 \cup D^3) \cap \inte b_0 = \emptyset$.
  Extend $b_0$ to an arc $b=b_0 * b_1$ which is transverse to~$D^3$, by choosing an arc $b_1$ from $x_0$ to $\partial D^3$ such that $(D_0 \cup S_1 \cup b_0 \cup D^3) \cap \inte b_1 = \emptyset$ and $b_1$ represents a given non-identity element $g\in \pi_1(M)$;
  here we use the simply connected subset $D_0 \cup S_1 \cup b_0 \cup D^3 \subset M$ as a basepoint.
  Let $\phi_B\colon M \to M$ be the barbell map associated with the barbell $B=S_1 \cup b \cup \partial D^3$.
  By the above paragraph, $\phi_B(D_0)$ is obtained by tubing $D_0$ to $\partial D^3$ along an arc $\tau=\tau_0 * b$ that extends $b$, where $\tau_0 \subset S_1$.
  This tubing construction producing $\phi_B(D_0)$ is equivalent to Gabai's self-referential disk construction \cite[Definition~2.1]{Gabai:2021-1}.
  Gabai showed that $\Dax(D_0, \phi_B(D_0)) = \epsilon\cdot (g+g^{-1})$, where $\epsilon\in \{1,-1\}$ is determined by the orientation of $D^3$ \cite[Theorem~4.9]{Gabai:2021-1};
  the sign $\epsilon$ is negated if the orientation of~$D^3$ is reversed.
  (We remark that $\phi_B(D_0)$ is the disk $D_{\epsilon\cdot g}$ in~\cite[Notation~4.4]{Gabai:2021-1}.)
  This verifies the claim.

  Recall that $D_1$ is a given disk such that $v=\Dax(D_0,D_1)$ lies in~$\Sigma_M$.
  Since $\Sigma_M$ is generated by $\{g+g^{-1} \mid g\in \pi_1(M), \; g\ne 1\}$ and since $\Dax$ is additive (\cref{proposition:top-dax-properties}~\cref{item:top-dax-additivity}), there is a composition $f_{D_0,D_1} \in \Diff(M)$ of finitely many barbell maps $\phi_B$ such that $\Dax(D_0,f_{D_0,D_1}(D_0)) = v$.
  By \cref{proposition:top-dax-properties}~\cref{item:top-dax-additivity} again, we have
  \[
    \Dax(f_{D_0,D_1}(D_0),D_1) = -\Dax(D_0,f_{D_0,D_1}(D_0)) + \Dax(D_0,D_1) = 0.
  \]
  By~\cite{Kosanovic-Teichner:2024-2}, this implies that $f_{D_0,D_1}(D_0)$ is smoothly isotopic to $D_1$ rel~$\partial$ in~$M$, since $D_1$ has a geometric dual in~$\partial M$.
  Thus, we may assume that $f_{D_0,D_1}(D_0)=D_1$ by smooth isotopy of~$f_{D_0,D_1}$.

  By~\cite[Proposition~2.6]{Budney-Gabai:2025-1}, each barbell map $\phi_B$ used above is smoothly pseudoisotopic to~$\id_M$, since the cuff $\partial D^3$ of our barbell $B$ bounds the 3-disk $D^3$.
  Consequently, $f_{D_0,D_1}$ is smoothly pseudoisotopic to~$\id_M$.
  Since $\phi_B$ is supported in the compact set $N(B)$, $f_{D_0,D_1}$ has compact support.

  We claim that $f_{D_0,D_1}$ and $f_{D_0',D_1'}$ commute up to smooth isotopy.
  To show the claim, it suffices to show that any two barbell maps $\phi_B$ and $\phi_{B'}$ used above commute up to smooth isotopy.
  Since we may assume that $N(B)$ and $N(B')$ are disjoint by their construction and since $\phi_B$ and $\phi_{B'}$ are supported in $N(B)$ and $N(B')$, $\phi_B$ and $\phi_{B'}$ commute.

  Furthermore, if $v = \Dax(D_0,D_1)\in \Sigma_M^\partial$, then $f_{D_0,D_1}$ is smoothly isotopic to a composition of barbell maps $\phi_B$ associated with barbells $B=S_1 \cup b \cup \partial D^3$ which are contained in a collar neighborhood $U$ of~$\partial M$.
  Such a barbell map $\phi_B$ commutes up to smooth isotopy with every $f\in \Diff(M)$, since $f$ can be smoothly isotoped so that $f|_U=\id_U$.
  It follows that our $f_{D_0,D_1}$ commutes up to smooth isotopy with every $g\in \Diff(M)$.

  This completes the proof of the smooth case.

  When $M$ is topological, we proceed as follows.
  Let $M' = (M\sm p_0)_\Gamma$, where $\Gamma$ is our fixed almost smooth structure.
  We may assume $D_i\subset M'$, by topological isotopy.
  Furthermore, since $D_i$ has a geometric dual in~$\partial M'$, we may assume that $D_i$ is smoothly embedded in the smooth 4-manifold $M'$ by the light bulb smoothing theorem~\cite[Theorem~C]{Cha-Kim:2023-1}.
  By definition, the smooth Dax invariant $\Dax_{M'}(D_0,D_1)$ is equal to the topological Dax invariant $\Dax_M(D_0,D_1)$, and thus $\Dax_{M'}(D_0,D_1) \in \Sigma_{M'}=\Sigma_M$ by the hypothesis.
  By the smooth case, it follows that there is a diffeomorphism $f'\colon M' \to M'$ with compact support such that $f'(D_0)=D_1$.
  Since $f'$ has compact support, $f'$ extends to a self-homeomorphism $f_{D_0,D_1}$ on $M$ by defining $f_{D_0,D_1}$ to be the identity near~$p_0$.
  We have $f_{D_0,D_1}(D_0)=f'(D_0)=D_1$, as desired.

  The remaining promised properties of $f_{D_0,D_1}$ can be verified as in the smooth case.
  The homeomorphism $f_{D_0,D_1}$ is pseudoisotopic to $\id_M$ by \cite[Proposition~2.6]{Budney-Gabai:2025-1}.
  The homeomorphisms $f_{D_0,D_1}$ and $f_{D_0',D_1'}$ commute up to isotopy since they are isotopic to compositions of barbell maps with disjoint supports.
  If $\Dax(D_0,D_1)\in \Sigma_M^\partial$, then $f_{D_0,D_1}$ is isotopic to a homeomorphism supported in a collar neighborhood of $\partial M$, and thus commutes with any $f\in \Homeo(M)$ up to isotopy.

  Finally, suppose that $M$ is stably smoothable.
  Recall that, in this case, $M\sm p_0$ is equipped with a smooth structure $\Gamma$ that extends to a smooth structure $\Theta$ on $M\csum k(S^2\times S^2)$.
  The argument in the topological case above yields a homeomorphism $f_{D_0,D_1}$ of~$M$, which restricts to a diffeomorphism of $M'=(M\sm p_0)_\Gamma$, and it extends to the diffeomorphism $f\csum \id_{k(S^2\times S^2)}$ on $(M\csum k(S^2\times S^2))_\Theta$.
  This shows that $f_{D_0,D_1}$ is stably smoothable.
\end{proof}

\begin{proof}[Proof of \cref{theorem:intro-realization-Sigma_M}]
  Recall that $D$ is a fixed neat disk in~$M$.
  For an element $v\in \Sigma_M \subset D_M$, there is a neat disk $D_v$ in $M$ such that $\partial D_v = \partial D$, $D_v \simeq D$ rel $\partial$ and $\Dax(D,D_v)=v$, by \cite[Corollary~H]{Cha-Kim:2023-1}.
  By \cref{item:dax-symmetrized-value}~$\Rightarrow$~\cref{item:disks-pseudoisotopy} in \cref{proposition:concordance-disk-pseudoisotopy}, there is $f_v\in \Homeopi(M)$ such that $f_v(D)=D_v$.
  Let $H\subset \pi_0 \Homeopi(M)\subset \pi_0 \Homeoh(M)$ be the subgroup generated by $\{[f_v]\mid v\in \Sigma_M\}$.
  We have $\DeltaD(f_v) = \Dax(D,f_v(D)) = \Dax(D,D_v) = v$.
  In addition, the homeomorphisms $f_v$ commute up to isotopy by \cref{proposition:concordance-disk-pseudoisotopy}.
  It follows that $H$ is abelian and surjects onto~$\Sigma_M$ under~$\DeltaDh\colon \pi_0 \Homeoh(M) \to \Sigma_M$.

  If $M$ is smooth, then we may assume $D$ and $D_v$ are smooth by the light bulb smoothing theorem~\cite[Theorem~C]{Cha-Kim:2023-1}, and then we may assume that the maps $f_v$ are diffeomorphisms that commute up to smooth isotopy by \cref{proposition:concordance-disk-pseudoisotopy}.
  Similarly, if $M$ is stably smoothable, then we may assume that $f_v$ is stably smoothable by \cref{proposition:concordance-disk-pseudoisotopy}.
\end{proof}

\subsection{Realization of the values in~\texorpdfstring{$D_M$}{D\_M}}
\label{subsection:realization-D_M}

Recall that the homomorphism $\Delta_D$ is defined on $\pi_0 H_D(M) \subset \pi_0 \Homeo(M)$.
The following result shows it is onto the abelian group~$D_M$, in the presence of a geometric dual.
In particular, $\DeltaD(f)$ need not lie in the subgroup $\Sigma_M$ in general.

Recall that the Torelli group is defined by $\cI(M) = \pi_0 I(M)\subset \pi_0 \Homeo(M)$, where $I(M) = \{f\in \Homeo(M) \mid f_* = \id \text{~on~} H_*(M) \}$.
When $M$ is smooth, the smooth Torelli group is defined by $\cI^s(M) = \pi_0 I^s(M) \subset \pi_0 \Diff(M)$ where $I^s(M) = \{f\in \Diff(M) \mid f_* = \id \text{~on~} H_*(M) \}$.

\begin{theorem}
  \label{theorem:realization-DeltaD}
  Suppose that $M$ is a 4-manifold and $D$ is a neat disk in $M$ with a geometric dual in~$\partial M$.
  Then the restriction $\DeltaD \colon \cI(M) \cap \pi_0 H_D(M) \to D_M$ is onto.
  In addition, if $M$ is smooth, then the precomposition of $\DeltaD$ with $\cI^s(M) \cap \pi_0(\Diff(M) \cap H_D(M)) \to \cI(M) \cap \pi_0 H_D(M)$ is onto~$D_M$;
  if $M$ is stably smoothable, then for any $v\in D_M$, there is a stably smoothable $[f]\in \cI(M) \cap \pi_0 H_D(M)$ such that $\DeltaD(f)=v$.
\end{theorem}

To prove \cref{theorem:realization-DeltaD}, we use a known construction of a homeomorphism which is called an arc pushing map.
For neat disks $D_0$ and $D_1$ in a 4-manifold $M$ which are homotopic rel~$\partial$ and admit an embedded geometric dual $G$ in $\partial M$, the construction gives a homeomorphism $f\in \Homeo(M)$ supported in an arbitrarily small neighborhood of $D_0 \cup D_1\cup G$ such that $f(D_0)=D_1$.
We call $f$ an \emph{arc pushing map} sending $D_0$ to~$D_1$.
See~\cite[Lemma~5.3]{Gabai:2021-1} for a detailed description of the arc pushing map~$f$.
When $M$ is smooth and $D_0$, $D_1$ are smoothly embedded, $f$ can be chosen to be a diffeomorphism.

\begin{lemma}
  \label{lemma:arc-pushing-torelli}
  If $f$ is an arc pushing map sending $D_0$ to $D_1$, then $f_* = \id$ on $H_i(M)$ for all~$i$ and $f_*=\id$ on $\pi_1(M)$.
\end{lemma}

\begin{proof}
  Let $G\subset \partial M$ be an embedded geometric dual of the disks~$D_i$.
  We use the following fact:
  let $X=M\cup H^3$ be obtained by attaching a 3-handle $H^3$ along~$G$.
  Then the homeomorphism $f\cup \id_{H^3} \in \Homeo(X)$ is isotopic to $\id_X$ rel~$\partial X$ by the construction in~\cite[Lemma~5.3]{Gabai:2021-1}.

  From this, it follows that $f_*=\id$ on $\pi_1(M)\cong \pi_1(X)$.
  For $i\ne 2$, it also follows that $f_* = \id$ on $H_i(M)\cong H_i(X)$.
  Thus, it remains to show that $f_*=\id$ on~$H_2(M)$.

  Let $\alpha\in H_2(M)$.
  Since $f\cup \id_{H^3}$ induces the identity on $H_2(X)=H_2(M)/\langle [G] \rangle$, we have $f_*(\alpha) = \alpha + k \cdot [G]$ for some $k \in \Z$.
  Let $\lambda\colon H_2(M) \times H_2(M,\partial M) \to \Z$ be the intersection pairing of~$M$.
  Note that $\lambda([G],[D_i])=1$ and $f_*[D_0] = [D_1] = [D_0]$.
  Thus we have
  \[
    k=\lambda(f_*(\alpha)-\alpha,[D_0]) = \lambda(f_*(\alpha),f_*[D_0]) - \lambda(\alpha,[D_0]) = 0.
  \]
  This shows that $f_*=\id$ on $H_2(M)$.
\end{proof}

\begin{proof}[Proof of \cref{theorem:realization-DeltaD}]
  Let $v \in D_M$ be an arbitrary element.
  Then, by~\cite[Corollary~H]{Cha-Kim:2023-1}, there is a neat disk $D_v$ in $M$ such that $\partial D_v = \partial D$, $D_v \simeq D \rel \partial$ and $\Dax(D,D_v)=v$.
  (This is true for any choice of an almost smooth structure on $M$ used to define the topological Dax invariant.)
  By the arc pushing construction, choose $f\in \Homeo(M)$ such that $f(D)=D_v$ and $f_*=\id$ on $\pi_1(M)$.
  Note that $f\in I(M) \cap H_D(M)$ by \cref{lemma:arc-pushing-torelli}.
  By the definition of $\DeltaD$, we have $\DeltaD(f)=\Dax(D,D_v)=v$.
  This shows the conclusion for a general topological 4-manifold~$M$.

  If $M$ is smooth, then we may assume that $D$ and $D_v$ are smoothly embedded in $M$, by the light bulb smoothing theorem~\cite[Theorem~C]{Cha-Kim:2023-1}.
  Thus, the above $f$ can be chosen to be a diffeomorphism.

  Suppose that $M$ is stably smoothable, i.e., $M\csum k(S^2\times S^2)$ is equipped with a smooth structure~$\Theta$.
  Recall that, in this case, we view $M\sm p_0 \subset M\csum k(S^2\times S^2)$ as an open subset and use the restriction $\Gamma=\Theta|_{M \sm p_0}$ as an almost smooth structure on~$M$.
  We may assume that the above disks $D$ and $D_v$ are contained in $M\sm p_0$ by topological isotopy.
  By the light bulb smoothing theorem~\cite[Theorem~C]{Cha-Kim:2023-1}, we may assume that $D$ and $D_v$ are smooth in $(M\sm p_0)_\Gamma$.
  The arc pushing construction gives a diffeomorphism $g$ of $(M\sm p_0)_\Gamma$ such that $g(D)=D_v$.
  In addition, $g$ is supported in an arbitrary neighborhood of $D\cup D_v \cup \partial M$, and thus $g=\id$ near~$p_0$.
  It follows that $g$ extends to a homeomorphism $f\in \Homeo(M)$ by defining $f$ to be $\id$ near~$p_0$.
  Note that $f(D)=g(D)=D_v \simeq D \rel \partial$.
  By \cref{lemma:arc-pushing-torelli}, $g_*=\id$ on $\pi_1(M\sm p_0)$ and on $H_i(M\sm p_0)$.
  Therefore $f\in I(M)\cap H_D(M)$, since $\pi_1(M)\cong \pi_1(M\sm p_0)$, $H_i(M)= H_i(M\sm p_0)$ for $i \ne 3$ and $H_3(M)$ is a quotient of $H_3(M\sm p_0)$.
  Also, $f\csum \id_{k(S^2\times S^2)}$ is a diffeomorphism of $(M\csum k(S^2\times S^2))_\Theta$.
  This shows that $f$ is stably smoothable.
\end{proof}

\section{Detecting structures in the mapping class groups}
\label{section:detecting-structures}

\subsection{Structures in \texorpdfstring{$\pi_0 \Homeopi(M)$}{\pi₀ Homeo\textasciicircum pi(M)} and \texorpdfstring{$\pi_0 \Homeoh(M)$}{\pi₀ Homeo\textasciicircum h(M)}}
\label{subsection:structures-in-pi_0_Homeoh-pi_0_Homeopi}

In this subsection, we prove \cref{theorem:intro-retract-direct-factor}, together with an addendum given below.

Recall that $T_M$ denotes the set of nontrivial elements of order 2 in~$\pi_1(M)$.
Denote the universal cover of~$M$ by $\widetilde M$, and its Betti numbers by~$\beta_i(\widetilde M)$.
Let $\mu_2 \colon \pi_2(M) \to \Z[\pi_1(M)] / \langle 1, g-g^{-1} \rangle$ be Wall's self-intersection invariant \cite[Chapter~5]{Wall-Ranicki:1999-1}, \cite[Section~1.7]{Freedman-Quinn:1990-1}, where $\langle 1, g-g^{-1} \rangle$ is the subgroup of $\Z[\pi_1(M)]$ generated by $\{1\} \cup \{g-g^{-1} \mid g\in \pi_1(M)\}$.
Denote by $\rk\mu_2$ the rank of the subgroup generated by the image~$\mu_2(\pi_2(M))$.

\begin{theorem-named}[\cref{theorem:intro-retract-direct-factor}]
  Let $M$ be a 4-manifold containing a neat disk with a geometric dual in~$\partial M$.
  Suppose $\rk\mu_2 < \infty$ and $\beta_3(\widetilde M) < \infty$.
  Let $r=|\pi_1(M)|$ if $\pi_1(M)$ is infinite.
  Otherwise, let
  \[
    r=\frac{|\pi_1(M)|+|T_M|-1}{2} - \bigl(\rk \mu_2 + \beta_3(\widetilde M)\bigr)
  \]
  and suppose $r>0$.
  Then $\pi_0\Homeopi(M)$ and $\pi_0 \Homeoh(M)$ have a common free abelian retract $A\cong \Z^r$ of rank~$r$.
  Furthermore, $A$ is a common direct factor of $\pi_0\Homeopi(M)$ and $\pi_0 \Homeoh(M)$ if $\pi_1(\partial M)\to\pi_1(M)$ is onto.
\end{theorem-named}

\begin{theorem-named}[Addendum to \cref{theorem:intro-retract-direct-factor}]
  If $M$ is stably smoothable, then every element of $A$ is represented by a stably smoothable homeomorphism.
  If $M$ is smooth, then $\pi_0 \Diffpi(M)$ and $\pi_0 \Diffh(M)$ have a common free abelian retract $B$ of rank $r$, which maps isomorphically onto $A$ under the natural map $\pi_0 \Diffpi(M) \to \pi_0 \Homeopi(M)$.
  The retract $B$ is a common direct factor of $\pi_0 \Diffpi(M)$ and $\pi_0 \Diffh(M)$ if $\pi_1(\partial M) \to \pi_1(M)$ is onto.
\end{theorem-named}

The following lemma estimates the size of $D_M$ and~$\Sigma_M$.
Recall that the quotient $D_M/\Sigma_M = \Z_2[T_M] / \mu(\pi_3(M))$ is a vector space over~$\Z_2$.

\begin{lemma}
  \label{lemma:free-part-of-D_M-Sigma_M}
  Let $M$ be a 4-manifold with nonempty boundary.
  Suppose $\rk \mu_2 < \infty$ and $\beta_3(\widetilde M) < \infty$.
  Define $r$ as in \cref{theorem:intro-retract-direct-factor} and suppose $r>0$.
  Then each of $D_M$ and $\Sigma_M$ surjects onto~$\Z^r$.
\end{lemma}

\begin{proof}
  The following diagram is commutative and has exact rows, due to~\cite[Theorem~3.16]{Kosanovic-Teichner:2024-2}.
  \begin{equation}
    \begin{tikzcd}[column sep=normal,row sep=large]
      0 \ar[r]
      & \Gamma(\pi_2(M)) \ar[r] \ar[d,"\Gamma(\mu_2)"]
      & \pi_3(M) \ar[r] \ar[d,"d"] \ar[dr,"\mu"]
      & H_3(\widetilde M) \ar[r] \ar[d,"\overline\mu"]
      & 0
      \\
      0 \ar[r]
      & \frac{\Z[\pi_1(M)]}{\langle 1, g-g^{-1}\rangle} \ar[r,"g\mapsto g+g^{-1}"]
      & \Z[\pi_1(M) \sm 1]^\sigma \ar[r]
      & \Z_2[T_M] \ar[r]
      & 0
    \end{tikzcd}
    \label{equation:whitehead-dax-wall}
  \end{equation}
  In fact, $\mu$ on $\pi_3(M)=\pi_3(\widetilde M)$ factors through $H_3(\widetilde M)$ via the Hurewicz map by~\cite[Lemma~4.2]{Schneiderman-Teichner:2022-1}, and $\overline\mu$ in \cref{equation:whitehead-dax-wall} is the homomorphism induced by~$\mu$.
  The top row is an exact sequence of Whitehead;
  $\Gamma(\pi_2(M))$ is Whitehead's quadratic group and $\Gamma(\mu_2)$ is the homomorphism induced by the quadratic map~$\mu_2$~\cite{Whitehead:1950-1}.
  We will use the fact that the image of $\Gamma(\mu_2)$ is generated by the image of $\mu_2$.

  We remark that although \cite[Theorem~3.16]{Kosanovic-Teichner:2024-2} treats the case of a smooth 4-manifold, the diagram \cref{equation:whitehead-dax-wall} for our topological 4-manifold $M$ follows immediately from the case of the smooth 4-manifold $(M\sm p_0)_\Gamma$.
  Indeed, the top row is exact for any 4-manifold $M$ with nonempty boundary, and the homomorphisms $d$ and $\mu$ for $M$ are induced by those of~$(M\sm p_0)_\Gamma$.

  We claim that the torsion-free quotient $D_M/tD_M$ is free abelian and $\rk D_M/tD_M = \rk D_M \ge r$.
  From this, it follows that $D_M$ surjects onto $D_M/tD_M\cong \Z^{\rk D_M}$, a free abelian group of rank${}\ge r$.

  To prove the claim, we use the following facts.
  \begin{enumerate}[label=(\roman*)]
    \item\label{item:torsion-free-part-free-by-fg}
    If $G$ is a quotient of a free abelian group by a finitely generated\ subgroup, then $G/tG$ is free abelian and $tG$ is finite.
    \item\label{item:torsion-free-part-torsion-extension}
    If $0\to H\to G\to \Gamma\to 0$ is an exact sequence of abelian groups, $H/tH$ is free abelian, and $\Gamma$ is bounded torsion (i.e., $n\Gamma=0$ for some $n>0$), then $G/tG$ is free abelian.
  \end{enumerate}
  Indeed, \cref{item:torsion-free-part-free-by-fg}~follows from the classification of finitely generated abelian groups and from the fact that a finitely generated\ subgroup of a free abelian group is contained in a finitely generated summand;
  \cref{item:torsion-free-part-torsion-extension}~follows from the fact that $H/tH \subset G/tG \subset \frac1n (H/tH)$ in~$G\otimes\Q$.
  We leave the details to the reader.

  Returning to our case, apply the snake lemma to \cref{equation:whitehead-dax-wall}, to obtain the following diagram:
  \begin{equation}
    \begin{tikzcd}[column sep=tiny,row sep=4ex]
      & & C
      & & D_M
      \\[-5ex]
      & & \rotatebox[origin=c]{-90}{$\coloneqq$}
      & & \rotatebox[origin=c]{-90}{$=$}
      \\[-5ex]
      \Ker \overline\mu \ar[rr,"\delta"] \ar[rd,two heads,start anchor={[xshift=-1em]south east}]
      & & \Coker \Gamma(\mu_2) \ar[rr] \ar[rd,two heads,start anchor={[xshift=-1em]south east}]
      & & \Coker d \ar[rr]
      & & \Coker \overline\mu \ar[rr]
      & & 0
      \\
      & \hbox to .5em{\hss\hbox to 0mm{\hss$D \coloneqq{}$}$\Im \delta$\hss} \ar[ru,hook,end anchor={[xshift=1em]south west}]
      & & \hbox to .5em{\hss\hbox to 0mm{\hss$B \coloneqq{}$}$\Coker \delta$\hss} \ar[ru,hook,end anchor={[xshift=1em]south west}]
    \end{tikzcd}
    \label{equation:snake-diagram}
  \end{equation}
  Note that $\Z[\pi_1(M)]/\langle 1, g-g^{-1}\rangle \cong \Z^R$ is a free abelian group of rank $R=\iota_M$, the number of non-identity inversion orbits in~$\pi_1(M)$.
  In the rank calculations below, we use the convention that $x-y=x$ when $x$ is infinite and $y$ is finite.

  Since $\rk\mu_2 < \infty$, $C\coloneqq\Coker\Gamma(\mu_2)$ is a quotient of $\Z^R$ by a finitely generated\ subgroup.
  It follows that $C/tC$ is free abelian and $tC$ is finite by~\cref{item:torsion-free-part-free-by-fg}.
  Note that, for $D \coloneqq \Im\{\delta\colon \Ker\overline\mu\to C\}$, $D/(D\cap tC) \cong (D+tC)/tC$ is free abelian since it is a subgroup of~$C/tC$.
  Also we have
  \[
    \rk D/(D\cap tC) \le \rk D \le \rk \Ker\overline\mu \le \beta_3(\widetilde M) < \infty.
  \]
  Thus $D/(D\cap tC)$ is finitely generated.
  Since $|tC| < \infty$, it follows that $D$ is finitely generated.
  Therefore, $B \coloneqq \Coker\delta = C/D$ is a quotient of $\Z^R$ by a finitely generated\ subgroup, and thus $B/tB$ is free abelian by~\cref{item:torsion-free-part-free-by-fg}.
  Since $\Coker \overline\mu$ is a vector space over $\Z_2$, $D_M/tD_M$ is free abelian by~\cref{item:torsion-free-part-torsion-extension}.
  In addition, from \cref{equation:snake-diagram}, we obtain the first two equalities below:
  \[
    \rk D_M = \rk B = \rk C - \rk D \ge (R - \rk \mu_2) - \beta_3(\widetilde M) = r.
  \]
  The above inequality follows from the fact that $C=\Coker \Gamma(\mu_2)$ and that $\Ker \overline\mu \subset H_3(\widetilde M)$ surjects onto~$D$.
  This proves the claim that $\rk D_M \ge r$.

  To prove the conclusion for $\Sigma_M$, recall that $D_M/\Sigma_M$ is a $\Z_2$-vector space, and thus $\rk \Sigma_M = \rk D_M$.
  Moreover, $t\Sigma_M = \Sigma_M \cap tD_M$, and hence the torsion-free quotient $\Sigma_M/t\Sigma_M$ is a subgroup of $D_M/tD_M$.
  Since $D_M/tD_M$ is free abelian, $\Sigma_M/t\Sigma_M$ is free abelian.
  Thus $\Sigma_M$ surjects onto $\Sigma_M/t\Sigma_M$, a free abelian group of rank $\rk \Sigma_M = \rk D_M \ge r$.
\end{proof}

\begin{remark}
  A similar argument which uses \cref{equation:whitehead-dax-wall,equation:snake-diagram} yields the following:
  if $\mu_2=0$ on $\pi_2(M)$ and $H_3(\widetilde M) = 0$, then $d=0$ on $\pi_3(M)$.
\end{remark}

\begin{proof}[Proof of \cref{theorem:intro-retract-direct-factor}]
  By \Cref{lemma:free-part-of-D_M-Sigma_M}, there exists a surjection $\eta \colon \Sigma_M \to \Z^r$.
  By \cref{theorem:intro-realization-Sigma_M}, there is an abelian subgroup $H \subset \pi_0\Homeopi(M) \subset \pi_0\Homeoh(M)$ that surjects onto $\Sigma_M$ under $\DeltaDh\colon \pi_0\Homeoh(M) \to \Sigma_M$.
  Since $\Z^r$ is free abelian and $H$ is abelian, there is a monomorphism $i\colon \Z^r \hookrightarrow H$ satisfying $(\eta\circ\DeltaDh) \circ i = \id_{\Z^r}$.
  Then $A \coloneqq i(\Z^r)\subset \pi_0 \Homeopi(M)$ is a common retract of $\pi_0\Homeopi(M)$ and $\pi_0\Homeoh(M)$.

  Now, suppose that $\pi_1(\partial M) \to \pi_1(M)$ is surjective.
  To show that the common retract $A$ is a common direct factor, it suffices to prove that $A$ is a central subgroup of $\pi_0\Homeoh(M)$.
  Recall from the proof of \cref{theorem:intro-realization-Sigma_M} that the abelian subgroup $H\subset \pi_0 \Homeopi(M)$ is generated by homeomorphisms $f_v$ ($v \in \Sigma_M$) which are constructed using \cref{proposition:concordance-disk-pseudoisotopy}.
  If $\pi_1(\partial M) \to \pi_1(M)$ is onto, then we have $\Sigma_M = \Sigma_M^\partial$, and thus every $f_v$ is central in $\pi_0 \Homeo(M)$ by \cref{proposition:concordance-disk-pseudoisotopy}.
  Since $A\subset H$, it follows that $A$ is central in $\pi_0 \Homeoh(M) \subset \pi_0 \Homeo(M)$.

  If $M$ is smooth, then by \cref{theorem:intro-realization-Sigma_M}, $H\subset \pi_0\Homeopi(M)$ is the image of an abelian subgroup $K\subset \pi_0\Diffpi(M)$ under the natural map $\Phi \colon \pi_0 \Diff(M) \to \pi_0 \Homeo(M)$, and thus $i\colon \Z^r \to H$ lifts to $j\colon \Z^r \to K$ such that $i=\Phi\circ j$.
  Then $j$ is injective, and $B\coloneqq j(\Z^r) \subset \pi_0 \Diffpi(M)$ is a common retract of $\pi_0\Diffpi(M)$ and $\pi_0\Diffh(M)$ which maps isomorphically onto $A$ under~$\Phi$.
  Furthermore, if $\pi_1(\partial M)\to \pi_1(M)$ is surjective, then $B$ is a central subgroup of $\pi_0\Diff(M)$ by \cref{proposition:concordance-disk-pseudoisotopy}.
  This proves the addendum for the smooth case.

  If $M$ is stably smoothable, every element in $A\subset H\subset \pi_0 \Homeopi(M)$ is stably smoothable by \cref{theorem:intro-realization-Sigma_M}.
  This proves the addendum for the stably smoothable case.
\end{proof}

\subsection{Structures in the Torelli group}
\label{subsection:structures-in-torelli}

In this subsection, we apply our method to study structures in the Torelli group.

Our first observation is an immediate consequence of \cref{theorem:intro-retract-direct-factor}, using that the Torelli group $\cI(M)$ contains the homotopy-trivial mapping class group $\pi_0 \Homeoh(M)$ as a subgroup.
Recall that $\rk_{\Pr}(\Gamma)$ is the Pr\"ufer rank of a group~$\Gamma$.

\begin{corollary}
  Let $M$ and $r$ be as in \cref{theorem:intro-retract-direct-factor}.
  Then $\rk_{\Pr}(\cI(M)) \ge r$.
\end{corollary}

\begin{proof}
  By \cref{theorem:intro-retract-direct-factor}, there is a subgroup $H \subset \pi_0 \Homeoh(M) \subset \cI(M)$ which is free abelian of rank~$r$.
  It follows that $\rk_{\Pr}(\cI(M)) \ge r$.
\end{proof}

In what follows, we focus on the quotient $\cI(M) / \pi_0 \Homeoh(M) = I(M) / {\simeq}\rel\partial M$, where $I(M) = \{f\in \Homeo(M) \mid f_*=\id \text{~on~} H_*(M)\}$.
Note that this quotient measures the gap between homology triviality and homotopy triviality.

\begin{theorem}
  \label{theorem:subgroup-in-torelli}
  Let $M$ be a 4-manifold containing a neat disk with a geometric dual in~$\partial M$.
  Suppose $H_3(\widetilde M) \otimes \Z_2$ is finitely generated.
  Let $s=|T_M|$ if $T_M$ is infinite.
  Otherwise, let
  \[
    s = |T_M| - \dim_{\Z_2} H_3(\widetilde M) \otimes \Z_2
  \]
  and suppose $s>0$.
  Then there exists an epimorphism of a subgroup $H\subset \cI(M)/\pi_0 \Homeoh(M)$ onto $(\Z_2)^s$, the $\Z_2$-vector space of dimension~$s$.

  In addition, if $M$ is smooth, then every element of $H$ is represented by a diffeomorphism, i.e., $H$ is the image of a subgroup of $\cI^s(M)/\pi_0 \Diffh(M)$ under the natural map $\cI^s(M)/\pi_0 \Diffh(M) \to \cI(M)/\pi_0 \Homeoh(M)$;
  if $M$ is stably smoothable, then every element of $H$ is represented by a stably smoothable homeomorphism.
\end{theorem}

The following is an immediate consequence of \cref{theorem:subgroup-in-torelli}.

\begin{corollary}
  \label{corollary:torelli-modulo-homotopy-prufer-rank}
  Let $M$ and $s$ be as in \cref{theorem:subgroup-in-torelli}.
  Then $\rk_{\Pr}\bigl(\cI(M)/\pi_0 \Homeoh(M)\bigr) \ge s$.
  In addition, if $M$ is smooth, then
  \[
    \rk_{\Pr}\biggl(\frac{\cI^s(M)}{\pi_0 \Diffh(M)}\biggr)
    \ge \rk_{\Pr}\biggl( \Im\biggl\{ \frac{\cI^s(M)}{\pi_0 \Diffh(M)} \rightarrow \frac{\cI(M)}{\pi_0 \Homeoh(M)} \biggr\} \biggr)
    \ge s.
  \]
\end{corollary}

\begin{proof}[Proof of \cref{theorem:intro-torelli-prufer-rank}]
  The conclusions immediately follow from \cref{corollary:torelli-modulo-homotopy-prufer-rank}, using the fact that $\rk_{\Pr}(\Gamma/N) \le \rk_{\Pr}(\Gamma)$ for any normal subgroup $N \subset \Gamma$.
\end{proof}

To prove~\cref{theorem:subgroup-in-torelli}, we use the following estimate of the size of the quotient target~$D_M/\Sigma_M$.
Recall that $D_M/\Sigma_M = \Z_2[T_M]/\mu(\pi_3(M))$ where $\mu\colon \pi_3(M) \to \Z_2[T_M]$ is the self-intersection homomorphism in~\cref{equation:homomorphism-mu}.

\begin{lemma}
  \label{lemma:dimension-of-D_M/Sigma_M}
  Let $M$ be a 4-manifold.
  Suppose $H_3(\widetilde M) \otimes \Z_2$ is finitely generated.
  Let $s=|T_M|$ if $T_M$ is infinite;
  otherwise, let $s = |T_M| - \dim_{\Z_2} H_3(\widetilde M) \otimes \Z_2$.
  Then $\dim_{\Z_2} \Z_2[T_M]/\mu(\pi_3(M)) \ge s$.
\end{lemma}

\begin{proof}
  By~\cite[Lemma~4.2]{Schneiderman-Teichner:2022-1}, $\mu$ factors through $H_3(\widetilde M)$ via the Hurewicz map.
  Since the target $\Z_2[T_M]$ is a vector space over $\Z_2$, $\mu$ factors through $H_3(\widetilde M)\otimes \Z_2$.
  From this, the conclusion follows.
\end{proof}

\begin{proof}[Proof of \cref{theorem:subgroup-in-torelli}]
  \Cref{corollary:homomorphism-modulo-homotopy} gives a homomorphism
  \[
    \tDeltaD\colon H_D(M)/{\simeq}\rel\partial M \to D_M/\Sigma_M.
  \]
  By \cref{theorem:realization-DeltaD}, the restriction of $\tDeltaD$ to a subgroup $H$ of $(I(M)\cap H_D(M))/{\simeq}\rel\partial M$ is onto~$D_M/\Sigma_M$.
  By \cref{lemma:dimension-of-D_M/Sigma_M}, $D_M/\Sigma_M$ surjects onto~$(\Z_2)^s$.
  It follows that the subgroup $H$ surjects onto $(\Z_2)^s$.

  The additional assertions for the smooth and stably smoothable cases follow from the corresponding cases of \cref{theorem:realization-DeltaD}.
\end{proof}

\section{Non-smoothable examples from freely slice links}
\label{section:non-smoothable-4-manifolds}

In this section, we construct from freely slice links an infinite family of 4-manifolds $M$ for which the vanishing Kirby--Siebenmann invariant case of \cref{theorem:intro-Z^infty-factor-nonsmoothable} holds;
that is, $M$ is non-smoothable, $\ks(M)=0$, $\pi_1(M)$ is a free group of arbitrary prescribed rank, and $\pi_0 \Homeopi(M)$ and $\pi_0 \Homeoh(M)$ have a common $\Z^\infty$ direct factor.

\subsection{Freely slice links and mapping class groups}

A link $L$ with $n$ components in $S^3$ is \emph{freely slice} in $D^4$ if there exist disjoint disks $D_1,\ldots,D_n$ topologically embedded in $D^4$ such that $\bigsqcup \partial D_i = L$ and $\pi_1(D^4 \sm \bigsqcup D_i)$ is free of rank $n$, generated by meridians.
In this paper, as a convention, a freely slice link $L$ is always equipped with free slice disks $\bigsqcup D_i$.
Let $X(L) = D^4 \sm \nu(\bigsqcup D_i)$ be the slice disk exterior.
The boundary $\partial X(L)$ is the zero-surgery manifold of~$L\subset S^3$.

Let
\begin{equation}
  M_k(L) = X(L) \bcsum (\bcsum^k S^2\times D^2), \quad k\ge 1.
  \label{equation:M_k(L)}
\end{equation}
We remark that for any collection of $k$ pairwise disjoint boundary-parallel neat arcs $\alpha_i \subset X(L)$, the exterior $X(L)\sm \nu(\bigsqcup \alpha_i)$ is homeomorphic to $M_k(L)$.

\begin{theorem}
  \label{theorem:examples-from-freely-slice-links}
  If $L$ is a freely slice link with $n$ components and $k\ge 1$, then $\ks(M_k(L))=0$, $\pi_1(M_k(L))$ is a free group of rank $n$, and $\pi_0 \Homeopi(M_k(L))$ and $\pi_0 \Homeoh(M_k(L))$ have a common direct factor isomorphic to $\Z^\infty$ generated by stably smoothable classes.
  Furthermore, if $L$ is not smoothly slice in any smooth homotopy 4-ball, then the 4-manifold $M_k(L)$ is not smoothable.
\end{theorem}

The following observation will be used in the proof of \cref{theorem:examples-from-freely-slice-links}.
Since we will also use it in the next section, we state it as a lemma.

\begin{lemma}
  \label{lemma:aspherical-bcsum-S^2-times-D^2}
  Suppose that $M=X\bcsum (\bcsum^k S^2\times D^2)$ with $k\ge 1$, where $X$ is a non-contractible aspherical 4-manifold with nonempty boundary.
  Then $\pi_0 \Homeopi(M)$ and $\pi_0 \Homeoh(M)$ have a common retract $A\cong \Z^\infty$, which is a common direct factor when $\pi_1(\partial X) \to \pi_1(X)$ is onto.
  In addition, if $\ks(X)=0$, then all elements of $A$ are stably smoothable.
  If $X$ is smooth, then $\pi_0 \Diffpi(M)$ and $\pi_0 \Diffh(M)$ have a common retract $B\cong \Z^\infty$ that maps isomorphically onto~$A$;
  $B$ is a common direct factor when $\pi_1(\partial X) \to \pi_1(X)$ is onto.
\end{lemma}

\begin{proof}
  Since $X$ is aspherical, $\pi_1(M) \cong \pi_1(X)$ is a torsion-free infinite group.
  Since $M\simeq X \vee (\bigvee^k S^2)$, the universal cover $\widetilde M$ has the homotopy type of $\bigvee^\infty S^2$, the wedge sum of infinitely many 2-spheres.
  Thus, $H_3(\widetilde M)=0$.
  Also, since the 2-sphere factors of $\widetilde M$ are represented by disjointly embedded framed 2-spheres in $M$, we have $\mu_2=0$ on~$\pi_2(M)$.
  Thus, by \cref{theorem:intro-retract-direct-factor} and its addendum, the conclusions follow.
\end{proof}

\begin{proof}[Proof of \cref{theorem:examples-from-freely-slice-links}]
  By \cite[Proposition 11.6C(1)]{Freedman-Quinn:1990-1}, $X(L)$ is homotopy equivalent to $\bigvee^n S^1$.
  Thus, $X(L)$ is aspherical and $\pi_1(M_k(L))\cong \pi_1(X(L))$ is free of rank~$n$.
  Also, $\pi_1(\partial M_k(L)) \to \pi_1(M_k(L)) \cong \pi_1(X(L))$ is onto, since $\pi_1(X(L))$ is generated by the meridians of~$L$.

  It is known that the Kirby--Siebenmann invariant is additive in the following strong sense:
  if $M = M_0 \cup M_1$ and $M_0 \cap M_1$ is a codimension-zero submanifold of each of $\partial M_0$ and $\partial M_1$, then $\ks(M) = \ks(M_0) + \ks(M_1)$~\cite[Section~10.2B]{Freedman-Quinn:1990-1}.
  By applying this to $D^4 = X(L) \cup (\bigsqcup^n D^2\times D^2)$, it follows that $\ks(X(L))=0$ since $D^2\times D^2$ and $D^4$ are smoothable.
  Applying additivity to $M_k(L) = X(L) \bcsum (\bcsum^k S^2\times D^2)$, we obtain $\ks(M_k(L))=0$.

  By \cref{lemma:aspherical-bcsum-S^2-times-D^2}, $\pi_0 \Homeopi(M_k(L))$ and  $\pi_0 \Homeoh(M_k(L))$ have a common $\Z^\infty$ direct factor, whose elements are stably smoothable.

  Suppose that $L$ is not smoothly slice in any smooth homotopy 4-ball.
  We will show that $M_k(L)$ is not smoothable.
  Assume, to the contrary, that $M_k(L)$ admits a smooth structure.
  The smooth structure restricts to the standard smooth structure on $\partial M_k(L) = \partial X(L) \csum (\csum^k S^2\times S^1)$, since every 3-manifold has a unique smooth structure.
  Attach $k$ smooth 3-handles to $M_k(L)$ along the 2-spheres $S^2 \times \{\text{pt}\}$ in the $S^2\times S^1$ summands of $\partial M_k(L)$, to obtain a smooth 4-manifold~$W$.
  Note that $W$ is homeomorphic to~$X(L)$ and $\partial W$ is diffeomorphic to $\partial X(L)$, the surgery manifold for~$L$.
  Attach $n$ smooth 2-handles to $W$ along the meridians of $L$, to obtain a smooth 4-manifold~$V$.
  Then $V$ is a smooth homotopy 4-ball in which $L$ is smoothly slice:
  the cocore disks of the 2-handles are slicing disks for $L$ in~$V$, and the slice disk exterior is~$W$.
  This contradicts the hypothesis.
\end{proof}

In what follows, we discuss some examples of links to which \cref{theorem:examples-from-freely-slice-links} applies.

For the one-component case ($n=1$), a knot is freely slice if and only if the Alexander polynomial is trivial~\cite{Freedman:1984-1,Freedman-Quinn:1990-1}.
(A freely slice knot is also called a \emph{$\Z$-slice knot}.)
For instance, let $\Wh^+(J)$ be the (untwisted) positive Whitehead double of a knot~$J$.
The knot $\Wh^+(J)$ is freely slice for all $J$, but for many $J$, it is known that $\Wh^+(J)$ is not smoothly slice in any smooth homotopy 4-ball.
For instance, see \cite{Hedden:2007-1,Manolescu-Owens:2007-1,Cha-Livingston-Ruberman:2006-1,Hedden-Kirk:2010-1,Hedden-Pinzon-Caicedo:2021-1,Nozaki-Sato-Taniguchi:2024-1,Dai-Hedden-Mallick-Stoffregen:2024-1,Daemi-Imori-Sato-Scaduto-Taniguchi:2025-1}.
Note that, in the knot case, a free slice disk exterior has fundamental group $\Z$, which is a good group.

The multi-component case ($n \ge 2$) is closely connected to the link-theoretic formulation of the topological surgery conjecture in dimension 4, a major open problem.
Recall that an $n$-component link $L$ is a \emph{boundary link} if it admits an epimorphism $\phi_L$ of $\pi_1(S^3\sm L)$ onto the free group of rank $n$ sending meridians to free generators.
All knots are boundary links.
The multi-component analogue of the trivial Alexander polynomial condition for knots is the notion of a good boundary link.
We say that a boundary link $L$ is \emph{good} if the free cover of $S^3\sm L$ associated with the above $\phi_L$ has trivial first homology.
For a Seifert matrix characterization, see~\cite[Corollary 2.2]{Cha-Kim-Powell:2020-1}.
It is known that topological surgery in dimension 4 works if and only if all good boundary links are freely slice~\cite[p.~243]{Freedman-Quinn:1990-1}.

This has motivated substantial efforts to prove that good boundary links are freely slice.
Specifically, results in~\cite{Freedman:1985-1,Freedman:1988-2,Freedman:1993-1,Freedman-Teichner:1995-2,Cha-Kim-Powell:2020-1} yield large families of freely slice links.
See~\cite[Theorem~A]{Cha-Kim-Powell:2020-1} for a statement which subsumes earlier results.
We remark that these known results still fall short of settling the topological surgery conjecture in dimension~4.

\subsection{Non-smoothability of \texorpdfstring{$M_k(L)$}{M\_k(L)}}

Many known freely slice links $L$ can be shown to be smoothly nonslice in any smooth homotopy 4-ball, and consequently, we can use these to produce non-smoothable 4-manifolds $M_k(L)$ by \cref{theorem:examples-from-freely-slice-links}.
In addition, many of these links are non-split and have unknotted components.
We describe explicit examples below.

\begin{example}
  \label{example:Wh^+(L_n(J))}
  For a knot $J$ and $n\ge 2$, we consider the $n$-component link $\Wh^+(L_n(J))$ shown at the top of \cref{figure:L_n(J)}.
  The diagram shows a link $L_n(J)$, whose positive Whitehead double is the link $\Wh^+(L_n(J))$.
  The link $L_n(J)$ is a satellite link with companion~$J$;
  the box {\setlength{\fboxrule}{1pt}\fbox{\footnotesize$J$}} indicates that
  the knot $J$ is tied in without introducing twists on the parallel strands.

  \begin{figure}[ht]
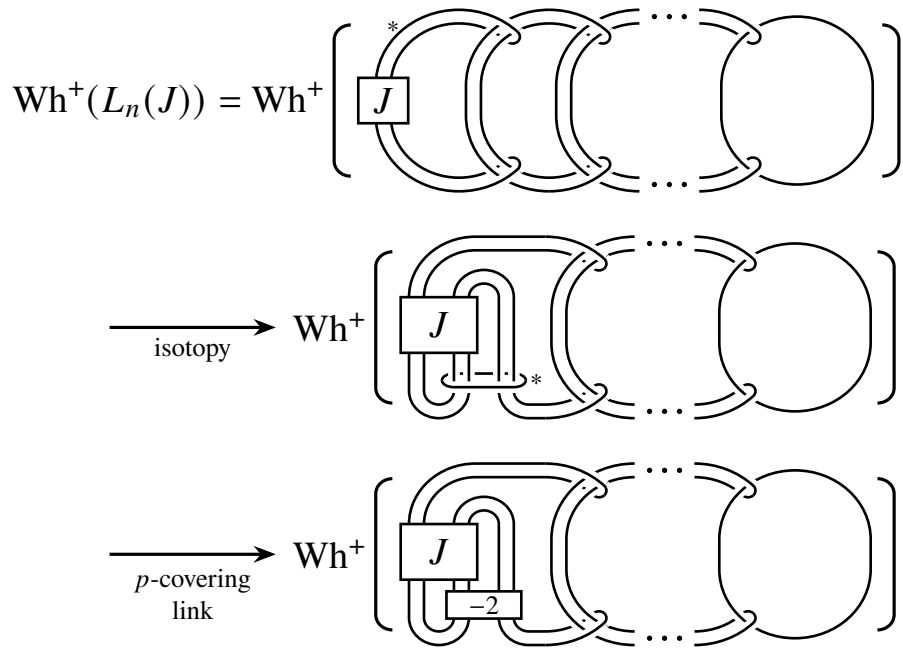

    \includestandalone{fig-covering-links}
    \caption{The link $\Wh^+(L_n(J))$ and its $p$-covering link.}
    \label{figure:L_n(J)}
  \end{figure}

  For a link $L$, the Whitehead double of $L$ (with arbitrary clasp signs) is a good boundary link if any two components of $L$ have vanishing linking number.
  In particular, our $\Wh^+(L_n(J))$ is a good boundary link.

  The link $\Wh^+(L_n(J))$ is freely slice, by a result of Freedman and Teichner~\cite{Freedman-Teichner:1995-2}.
  An $n$-component link is \emph{homotopically trivial} if there is a homotopy of the link to the trivial link through immersions of $n$ circles with disjoint images.
  (Crossing changes are allowed only between strands of the same component.)
  An $n$-component link $L$ is \emph{homotopically trivial+} if, for each component of $L$, the $(n+1)$-component link obtained by adding to $L$ a zero-linking parallel copy of the chosen component is homotopically trivial.
  By~\cite{Freedman-Teichner:1995-2}, a Whitehead double of a homotopically trivial+ link is freely slice.
  In our case, from the first and second link diagrams in \cref{figure:L_n(J)}, it is readily seen that the link $L_n(J)$ is homotopically trivial+, and consequently, $\Wh^+(L_n(J))$ is freely slice.
  When a parallel copy of the component of $L_n(J)$ marked by $*$ is added, one can verify the homotopy triviality by applying crossing changes to the second diagram in \cref{figure:L_n(J)} to eliminate the tying by~$J$.
  To see the homotopy triviality in the other cases, apply crossing changes to the first diagram to eliminate the tying by~$J$.

  Now let $J$ be any knot with $\tau(J) > 0$, where $\tau(J)$ is the Ozsv\'ath--Szab\'o $\tau$-invariant~\cite{Ozsvath-Szabo:2003-1}.
  Then, by \cref{proposition:Wh(L_n(J))-non-smoothly-slice} below, $\Wh^+(L_n(J))$ is not smoothly slice in any smooth homotopy 4-ball.
  Thus, by \cref{theorem:examples-from-freely-slice-links}, for any $k\ge 1$, the associated 4-manifold $M = M_k(\Wh^+(L_n(J)))$ is non-smoothable, $\ks(M)=0$, $\pi_1(M)$ is free of rank $n$, and $\pi_0 \Homeopi(M)$ and $\pi_0 \Homeoh(M)$ have a common $\Z^\infty$ direct factor generated by stably smoothable classes.
\end{example}

\begin{proof}[Proof of \cref{theorem:intro-Z^infty-factor-nonsmoothable}, $\epsilon=0$ case]
  In \cref{example:Wh^+(L_n(J))}, $\partial M$ is the connected sum of the zero-surgery manifold of $\Wh^+(L_n(J))$ and $k$ copies of $S^1\times S^2$.
  Thus, by varying $k$, we obtain infinitely many pairwise non-homeomorphic non-smoothable 4-manifolds $M=M_k(\Wh^+(L_n(J)))$.
  The 4-manifold $M$ has $\ks(M)=0$ and satisfies the conclusions of \cref{theorem:intro-Z^infty-factor-nonsmoothable} for $\epsilon=0$.
  For $n\ge 2$, this proves the assertion.
  For $n=1$, apply the same argument to $M=M_k(\Wh^+(J))$ where $J$ is a knot with $\tau(J)>0$.
\end{proof}

\begin{remark}
  Alternatively, one can vary $J$, with $k\ge 1$ fixed, to obtain infinitely many pairwise non-homeomorphic non-smoothable 4-manifolds $M = M_k(\Wh^+(L_n(J)))$.
  In this case, all of these $M$ have the same homotopy type:
  $M \simeq (\bigvee^n S^1) \vee (\bigvee^k S^2)$.
  We omit the proof.
\end{remark}

\begin{proposition}
  \label{proposition:Wh(L_n(J))-non-smoothly-slice}
  The link $\Wh^+(L_n(J))$ is not smoothly slice in any smooth rational homology 4-ball if $\tau(J) > 0$.
\end{proposition}

To prove \cref{proposition:Wh(L_n(J))-non-smoothly-slice}, we use the covering link calculus technique in the smooth category, following the formulation in~\cite[Section~2]{Cha-Kim:2008-1} and \cite[Section~3]{Cha-Powell:2014-1}.
Our proof is also similar to the proofs in~\cite{Cha-Livingston-Ruberman:2006-1,LevineA:2009-1,Livingston-VanCott:2011-1}.
Although \cite{Cha-Kim:2008-1,Cha-Powell:2014-1} deal with the more general case of links in $\Z_{(p)}$-homology spheres, we will use links in $S^3$ only.

Fix a prime~$p$.
Consider the following two operations:
(i)~take a sublink of~$L$;
(ii)~take the preimage of $L$ in the $p^k$-fold cyclic branched cover of $S^3$ taken along an unknotted component of~$L$.
(Note that the result is a link in~$S^3$.)
A link obtained from a given link $L$ by applying a finite sequence of these two operations is called a \emph{$p$-covering link} of~$L$.

\begin{theorem}[{\cite[Theorem~2.2]{Cha-Kim:2008-1}}]
  \label{theorem:covering-link-calculus}
  If $L$ is slice in a $\Z_{(p)}$-homology 4-ball, then every $p$-covering link of $L$ is slice in a $\Z_{(p)}$-homology 4-ball.
\end{theorem}

\begin{figure}[ht]
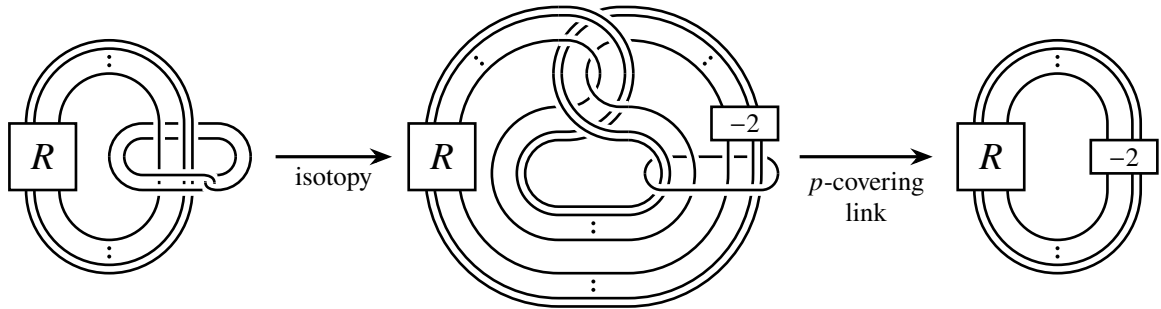

  \includestandalone{fig-wh-covering-lemma}
  \caption{
    A covering link construction along a Whitehead-doubled unknotted component.
    The box {\setlength{\fboxrule}{1pt}\fbox{\footnotesize$-2$}} indicates $-2$ full twists.
  }
  \label{figure:covering-link-along-whitehead}
\end{figure}

We will use a covering link construction shown in \cref{figure:covering-link-along-whitehead};
the leftmost diagram depicts a link, one of whose components is the positive Whitehead double of an unknotted circle.
The box {\setlength{\fboxrule}{1pt}\fbox{\footnotesize$R$}} indicates an arbitrary tangle, and the number of parallel strands connected to {\setlength{\fboxrule}{1pt}\fbox{\footnotesize$R$}} is also arbitrary.

\begin{lemma}
  \label{lemma:covering-link-along-whitehead}
  In \cref{figure:covering-link-along-whitehead}, the rightmost link is a $p$-covering link of the leftmost link for every prime~$p$.
\end{lemma}

\begin{proof}
  Isotope the leftmost diagram to the middle diagram in \cref{figure:covering-link-along-whitehead}, by transforming the Whitehead-doubled component to an unknotted circle.
  Along that component, take the $p$-fold cyclic branched cover.
  The preimage of the remaining components is a link in the branched cover.
  This link contains the link depicted in the rightmost diagram in \cref{figure:covering-link-along-whitehead} as a sublink.
\end{proof}

\begin{proof}[Proof of \cref{proposition:Wh(L_n(J))-non-smoothly-slice}]
  For a knot $K$, let $P(K)$ be the satellite knot shown in \cref{figure:iterated-covering-pattern}.
  Viewing the pattern $P$ as a satellite operator, define the iterated operator by $P^{k+1}(K) = P(P^k(K))$.

  \begin{figure}[ht]
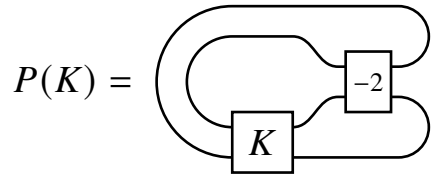

    \includestandalone{fig-iterated-covering-pattern}
    \caption{A satellite knot $P(K)$.}
    \label{figure:iterated-covering-pattern}
  \end{figure}

  We claim that the knot $\Wh^+(P^{n-1}(J))$ is a $p$-covering link of $\Wh^+(L_n(J))$ for every prime~$p$.
  To show this, start with the second diagram of $\Wh^+(L_n(J))$ in \cref{figure:L_n(J)}.
  Apply \cref{lemma:covering-link-along-whitehead}, taking the cover branched along (the Whitehead double of) the component labeled by~$*$, to obtain a $p$-covering link depicted in the last diagram in~\cref{figure:L_n(J)}.
  Observe that it is the link $\Wh^+(L_{n-1}(P(J)))$.
  Repeated application of this construction shows that the 2-component link $\Wh^+(L_2(P^{n-2}(J)))$ is a $p$-covering link of $\Wh^+(L_n(J))$.
  Note that $\Wh^+(L_2(P^{n-2}(J)))$ is the 2-component case of the first diagram in \cref{figure:L_n(J)} with $J$ replaced by $P^{n-2}(J)$.
  Apply \cref{lemma:covering-link-along-whitehead} to this link to obtain the knot~$\Wh^+(P^{n-1}(J))$ as a $p$-covering link.
  This verifies the claim.

  We claim that $\Wh^+(P^{n-1}(J))$ is not smoothly slice in any rational homology 4-ball.
  To verify this, we use the fact that if a knot $K$ satisfies $\tau(K)>0$, then $\tau(P(K)) \ge \tau(\Wh^+(K)) > 0$.
  The first inequality holds by \cite[Corollary~1.5]{Ozsvath-Szabo:2003-1}, since $\Wh^+(K)$ is obtained by changing one positive crossing of~$P(K)$.
  The second inequality holds by \cite[Theorem~1.5]{Hedden:2007-1}.
  Apply these inequalities inductively, with our hypothesis $\tau(J) > 0$ as the initial condition, to conclude that $\tau(\Wh^+(P^{n-1}(J))) > 0$ for all~$n$.
  The claim follows from this, by~\cite[Theorem~1.1]{Ozsvath-Szabo:2003-1}.

  If the link $\Wh^+(L_n(J))$ were smoothly slice in a rational homology 4-ball $W$, then since $W$ is a $\Z_{(p)}$-homology ball for some prime $p$, the $p$-covering knot $\Wh^+(P^{n-1}(J))$ of $\Wh^+(L_n(J))$ would be smoothly slice in a $\Z_{(p)}$-homology ball by \cref{theorem:covering-link-calculus}.
  This contradicts the claim in the previous paragraph.
\end{proof}

\section{More examples}
\label{section:more-examples}

\subsection{Smooth examples from aspherical 4-manifolds}
\label{subsubsection:examples-from-aspherical}

There are many examples of smooth aspherical 4-manifolds.
The simplest examples come from surfaces:
the product of two compact orientable surfaces of genus${}\ge 1$ is a smooth aspherical 4-manifold.
Similarly, $X = Y\times S^1$ is an aspherical 4-manifold if $Y$ is a compact aspherical 3-manifold, in particular, when $Y$ is a Haken 3-manifold with incompressible boundary.
More generally, surface bundles over surfaces are aspherical 4-manifolds when the surfaces have genus${}\ge 1$;
similarly for Haken 3-manifold bundles over the circle.
Another example is the family of Haken 4-manifolds in the sense of Foozwell and Rubinstein, which are higher-dimensional analogues of Haken 3-manifolds~\cite{Foozwell-Rubinstein:2011-1}.
These are aspherical 4-manifolds.
Compact negatively curved 4-manifolds with locally convex boundary are also examples.
This includes, in particular, compact hyperbolic 4-manifolds with totally geodesic boundary.

Many of the smooth aspherical 4-manifolds $X$ described above are non-contractible and have nonempty boundary.
For each such $X$ and each $k\ge 1$, the smooth 4-manifold $M=X\bcsum(\bcsum^k S^2\times D^2)$ has the following property by \cref{lemma:aspherical-bcsum-S^2-times-D^2}:
$\pi_0\Diffpi(M)$ and $\pi_0\Diffh(M)$ have a common $\Z^\infty$ retract which maps isomorphically onto a common $\Z^\infty$ retract of $\pi_0\Homeopi(M)$ and $\pi_0\Homeoh(M)$.

\subsection{Smooth examples with nontrivial Torelli quotients}

View $D^2$ as the unit disk in $\C$, and let
\[
  D^2 \tildetimes \RP^2 = (D^2 \times S^2) / (z,x) \sim (\bar z, -x).
\]
This is a twisted $D^2$-bundle over $\RP^2$ with orientable total space.

\begin{example}
  \label{example:nonaspherical-with-2-torsion-in-pi_1}
  Let $M = (\bcsum^n D^2 \tildetimes \RP^2) \bcsum (\bcsum^k S^2\times D^2)$ where $n\ge 2$, $k\ge 1$.
  We will show the following assertions:
  \begin{enumerate}[label=(\roman*)]
    \item \label{item:nonaspherical-Z^infty-factors}
    The groups $\pi_0 \Diffpi(M)$ and $\pi_0 \Diffh(M)$ have a common $\Z^\infty$ direct factor, which maps isomorphically to a common $\Z^\infty$ direct factor of $\pi_0 \Homeopi(M)$ and $\pi_0 \Homeoh(M)$.
    Consequently, none of $\pi_0 \Diffpi(M)$, $\pi_0 \Diffh(M)$, $\pi_0 \Homeopi(M)$ and $\pi_0 \Homeoh(M)$ is finitely generated.
    \item \label{item:nonaspherical-torelli}
    The image of $\cI^s(M)/\pi_0 \Diffh(M)$ in $\cI(M)/\pi_0 \Homeoh(M)$ has infinite Pr\"ufer rank.
    Consequently, both $\cI^s(M)/\pi_0 \Diffh(M)$ and $\cI(M)/\pi_0 \Homeoh(M)$ have infinite Pr\"ufer rank.
  \end{enumerate}

  Note that $(\bcsum^n D^2 \tildetimes \RP^2)$ is not aspherical, and thus \cref{lemma:aspherical-bcsum-S^2-times-D^2} does not yield~\cref{item:nonaspherical-Z^infty-factors}.

  To prove the assertions, we first compute $H_3(\tilde M)$ and Wall's self-intersection~$\mu_2$.
  Since $D^2 \tildetimes \RP^2 \simeq \RP^2$ and $S^2\times D^2 \simeq S^2$, the universal cover $\widetilde M$ has the homotopy type of a wedge of 2-spheres.
  Thus $H_3(\widetilde M)=0$.
  Note that the generator of $\pi_2(S^2\times D^2)=\Z$ is represented by an embedded 2-sphere with trivial normal bundle.
  In addition, the generator of $\pi_2(D^2 \tildetimes \RP^2) = \Z$ is represented by the embedding $S^2 \hookrightarrow D^2 \tildetimes \RP^2$ defined by $x \mapsto [\sqrt{-1}/2,x]$, which has a trivial normal bundle.
  (It lifts to $S^2 \hookrightarrow D^2\times S^2$, $x \mapsto (\sqrt{-1}/2,x)$.)
  Also, these 2-sphere representatives are disjoint.
  From this, it follows that $\mu_2$ vanishes on $\pi_2(M)$.

  Since $\pi_1(M)$ is the free product of $n$ copies of $\Z_2$ with $n\ge 2$, we have $|T_M|=|\pi_1(M)|=\infty$.
  Note that the neat disk $*\times D^2$ in an $S^2\times D^2$ summand of $M$ has a geometric dual in~$\partial M$.
  Therefore, the assertions \cref{item:nonaspherical-Z^infty-factors} and \cref{item:nonaspherical-torelli} follow from \cref{theorem:intro-retract-direct-factor} ($r=\infty$) and \cref{corollary:torelli-modulo-homotopy-prufer-rank} ($s=\infty$), respectively.
\end{example}

The following two examples are variations on \cref{example:nonaspherical-with-2-torsion-in-pi_1}.
We omit their proofs, since the argument of \cref{example:nonaspherical-with-2-torsion-in-pi_1} applies.

\begin{example}
  \label{example:nonaspherical-one-D2-times-RP2-summand}
  When $n=1$ in \cref{example:nonaspherical-with-2-torsion-in-pi_1}, we have $\pi_1(M) = \Z_2$, $r=s=1$, and thus the above argument shows the following:
  \begin{enumerate}[label=(\roman*)]
    \item The groups $\pi_0 \Diffpi(M)$, $\pi_0 \Diffh(M)$, $\pi_0 \Homeopi(M)$ and $\pi_0 \Homeoh(M)$ have a central element of infinite order.
    \item The groups $\cI^s(M)/\pi_0 \Diffh(M)$ and $\cI(M)/\pi_0 \Homeoh(M)$ have nonzero Pr\"ufer rank.
  \end{enumerate}
\end{example}

\begin{example}
  \label{example:nonaspherical-X-D2-times-RP2-summands}
  Let $M = X \bcsum (D^2 \tildetimes \RP^2) \bcsum (\bcsum^k S^2\times D^2)$ where $X$ is an arbitrary non-contractible smooth aspherical 4-manifold with nonempty boundary.
  Then \cref{example:nonaspherical-with-2-torsion-in-pi_1}~\cref{item:nonaspherical-torelli} holds.
  \Cref{example:nonaspherical-with-2-torsion-in-pi_1}~\cref{item:nonaspherical-Z^infty-factors} holds with ``direct factor'' replaced by ``retract.''
\end{example}

\subsection{Non-smoothable examples with a given homotopy type}
\label{subsection:non-smoothable-given-homotopy-type}

In this subsection, we discuss more examples of non-smoothable 4-manifolds $M$ for which our method detects interesting structures in the mapping class groups.
Recall that our first non-smoothable examples $M_k(L)$ were obtained from freely slice links $L$ in \cref{section:non-smoothable-4-manifolds}.
These 4-manifolds have $\ks(M_k(L))=0$, as stated in \cref{theorem:examples-from-freely-slice-links}.

\subsubsection{Examples with nonvanishing Kirby--Siebenmann invariant}

In the examples below, we use a construction to convert a 4-manifold $M$ with nonempty boundary into a 4-manifold $M'$ having the homotopy type of $M$ and satisfying $\ks(M')=\ks(M)+1$.

\begin{example}
  \label{example:M_k(L)-nontrivial-ks}
  For every freely slice link $L$, we construct a non-smoothable 4-manifold $M'$ homotopy equivalent to the 4-manifold $M_k(L)$ defined in~\cref{equation:M_k(L)} at the beginning of \cref{section:non-smoothable-4-manifolds}, such that $\ks(M')=1$ and $\pi_0 \Homeopi(M')$ and $\pi_0\Homeoh(M')$ have a common $\Z^\infty$ direct factor.

  Choose a homology 3-sphere $\Sigma$ with nonvanishing Rokhlin invariant $\mu(\Sigma) \in \Z_2$.
  By Freedman's work~\cite{Freedman:1982-1}, $\Sigma$ bounds a contractible 4-manifold~$W_1$.
  It is known that $\ks(W_1) = \mu(\Sigma) \in \Z_2$~\cite[10.2B]{Freedman-Quinn:1990-1}.
  Thus we have $\ks(W_1)=1$.

  Now, for a freely slice link $L$, let $M' = W_1 \bcsum M_k(L)$, $k\ge 1$.
  By the additivity of $\ks$, we have $\ks(M') = \ks(W_1) + \ks(M_k(L)) = 1$.
  Consequently, the 4-manifold $M'$ is non-smoothable.
  Since $W_1$ is contractible, $M'$ is homotopy equivalent to~$M_k(L)$.
  Since $M' = (W_1 \bcsum X(L)) \bcsum (\bcsum^k S^2\times D^2)$ and $W_1 \bcsum X(L)$ is aspherical, \cref{lemma:aspherical-bcsum-S^2-times-D^2} applies to $M'$ to yield the promised properties.
\end{example}

\begin{proof}[Proof of \cref{theorem:intro-Z^infty-factor-nonsmoothable}, $\epsilon=1$ case]
  By varying $k$ (or the freely slice link $L$) and by varying the homology 3-sphere $\Sigma$ in \cref{example:M_k(L)-nontrivial-ks}, we obtain infinitely many 4-manifolds $M'$ with $\ks(M')=1$, which satisfy the conclusion of \cref{theorem:intro-Z^infty-factor-nonsmoothable} for $\epsilon=1$.
\end{proof}

The above trick also applies to other smooth examples described above, to produce corresponding non-smoothable examples $M'$ with $\ks(M')=1$.
We describe specific cases below.

\begin{example}
  \label{example:nonsmoothable-with-2-torsion-in-pi_1}
  As in \cref{example:nonaspherical-with-2-torsion-in-pi_1,,example:nonaspherical-one-D2-times-RP2-summand,example:nonaspherical-X-D2-times-RP2-summands}, consider the 4-manifolds $M = (\bcsum^n D^2 \tildetimes \RP^2) \bcsum (\bcsum^k S^2\times D^2)$ and $M = X \bcsum (D^2 \tildetimes \RP^2) \bcsum (\bcsum^k S^2\times D^2)$, with $X$ smooth and aspherical and $n$,~$k\ge 1$.
  As in \cref{example:M_k(L)-nontrivial-ks}, choose a contractible 4-manifold $W_1$ with $\ks(W_1)=1$.
  Let $M'=W_1 \bcsum M$.
  Then $M'\simeq M$, $\ks(M')=1$ and $M'$ is non-smoothable, by the arguments in \cref{example:M_k(L)-nontrivial-ks}.
  Furthermore, by the arguments in \cref{example:nonaspherical-with-2-torsion-in-pi_1}, \cref{theorem:intro-retract-direct-factor} and \cref{corollary:torelli-modulo-homotopy-prufer-rank}  yield the following:
  \begin{enumerate}[label=(\roman*)]
    \item For $M' = W_1\bcsum (\bcsum^n D^2 \tildetimes \RP^2) \bcsum (\bcsum^k S^2\times D^2)$ with $n\ge 2$ and $k\ge 1$, $\pi_0\Homeopi(M')$ and $\pi_0\Homeoh(M')$ have a common $\Z^\infty$ direct factor.
    For $M' = W_1 \bcsum X \bcsum (D^2 \tildetimes \RP^2) \bcsum (\bcsum^k S^2\times D^2)$ with $X$ non-contractible and aspherical, $\pi_0\Homeopi(M')$ and $\pi_0\Homeoh(M')$ have a common $\Z^\infty$ retract.
    \item For $M' = W_1\bcsum (D^2 \tildetimes \RP^2) \bcsum (\bcsum^k S^2\times D^2)$, $\cI(M')/\pi_0\Homeoh(M')$ has Pr\"ufer rank${}\ge 1$, and $\pi_0\Homeopi(M')$ and $\pi_0 \Homeoh(M')$ have a central element of infinite order.
  \end{enumerate}
\end{example}

\subsubsection{Examples with vanishing Kirby--Siebenmann invariant}
\label{subsubsection:examples-with-ks-zero}

The above construction can be modified to convert the previously described smooth 4-manifold examples into non-smoothable 4-manifolds $M'$ with $\ks(M')=0$ that admit analogous structures in mapping class groups detected by our invariant.
We describe it below.

\begin{example}
  \label{example:nonsmoothable-with-ks-zero}
  We first claim that there exist infinitely many non-smoothable contractible 4-manifolds $W_0$ with $\ks(W_0)=0$.
  To verify this, take a homology 3-sphere $\Sigma$ with nonvanishing Rokhlin invariant $\mu(\Sigma)$ and a contractible 4-manifold $W_1$ bounded by~$\Sigma$, as in \cref{example:M_k(L)-nontrivial-ks}.
  Let $W_0=W_1\bcsum W_1$.
  We have $\partial W_0 = \Sigma \csum \Sigma$.
  By Manolescu's work \cite[Corollary 1.2]{Manolescu:2016-1}, $\mu(\Sigma)\ne 0$ implies that $\Sigma \csum \Sigma$ does not bound any smooth integral homology 4-ball.
  It follows that $W_0$ is non-smoothable.
  We also have $\ks(W_0)= 2\ks(W_1) = 0$ by additivity.
  By varying $\Sigma$, we obtain infinitely many pairwise non-homeomorphic~$W_0$.

  Recall that $M = (\bcsum^n D^2 \tildetimes \RP^2) \bcsum (\bcsum^k S^2\times D^2)$ is the smooth 4-manifold in \cref{example:nonaspherical-with-2-torsion-in-pi_1,example:nonaspherical-one-D2-times-RP2-summand}.
  For brevity, let $N = (\bcsum^n D^2 \tildetimes \RP^2)$.
  Let $M' = W_0 \bcsum M = W_0 \bcsum N \bcsum (\bcsum^k S^2\times D^2)$.

  We claim that $M'$ is non-smoothable.

  To prove this, suppose $M'$ is smoothable.
  Then $W_0\bcsum N$ is smoothable, since it is obtained by attaching handles to~$M'$.
  View $W_0\bcsum N$ as $(W_0\sqcup N)\cup H^1$ where $H^1$ is a 1-handle that forms the boundary connected sum, and attach a 3-handle $H^3$ along the belt sphere of $H^1$, to obtain a smoothable 4-manifold~$V$.
  Note that $V = (W_0 \bigsqcup N)\cup H^1\cup H^3$ is homeomorphic to $W_0 \csum N$.
  Let $X = V \mathbin{\cup_{\partial N}} (-N)$, a smoothable 4-manifold with $\partial X = \partial W_0 = \Sigma \csum \Sigma$.
  Note that $X$ is homeomorphic to $W_0 \csum D(N)$ where $D(N)$ is the double of~$N$.

  Since the zero section $\RP^2 = 0 \times \RP^2$ is a deformation retract of $D^2 \tildetimes \RP^2$, $\tilde H_*(D^2 \tildetimes \RP^2;\Q)=0$ and thus $\tilde H_*(N;\Q)=0$.
  In addition, $N$ is spin.
  To see this, observe that $(S^2\times\R)\times\R \cong (\R^3\sm 0)\times \R \subset \R^4$;
  using the $\R^4$ coordinates, we have $\R^2 \tildetimes \RP^2 \cong ((\R^3\sm 0)\times \R)/{x\sim -x} \cong (\RP^3 \sm \{\mathrm{pt}\}) \times (0,\infty)$.
  It follows that $\R^2 \tildetimes \RP^2$ is parallelizable, since $\RP^3$ is parallelizable.
  Therefore $D^2 \tildetimes \RP^2$ is spin, and thus $N$ is spin.

  From the above, it follows that the smoothable 4-manifold $X=W_0 \csum D(N)$ is a spin rational homology 4-ball bounded by $\Sigma \csum \Sigma$.
  Attach a 3-handle along the connected sum sphere in $\partial X=\Sigma\csum\Sigma$, to obtain a smooth spin cobordism from $-\Sigma$ to $\Sigma$ with $b_2=0$.
  By~\cite[Theorem~1.1 and Corollary~3.10]{Manolescu:2016-1}, Manolescu's $\beta$-invariant satisfies $\beta(\Sigma)=\beta(-\Sigma)=-\beta(\Sigma)$, and thus $\beta(\Sigma)=0$.
  However, $\beta(\Sigma)\equiv\mu(\Sigma)=1 \pmod 2$.
  This proves that $M'$ is non-smoothable.

  We also have $M'\simeq M$ and $\ks(M')=0$.
  \Cref{theorem:intro-retract-direct-factor} and \cref{corollary:torelli-modulo-homotopy-prufer-rank} give the following:

  \begin{enumerate}[label=(\roman*)]
    \item For $n\ge 2$, $\rk_{\Pr}(\cI(M')/\pi_0\Homeoh(M'))=\infty$, and
    the groups $\pi_0 \Homeopi(M')$ and $\pi_0 \Homeoh(M')$ have a common $\Z^\infty$ direct factor whose elements are stably smoothable.
    \item For $n=1$, $\rk_{\Pr}(\cI(M')/\pi_0\Homeoh(M')) \ge 1$, and the groups $\pi_0 \Homeopi(M')$ and $\pi_0 \Homeoh(M')$ have a central, stably smoothable element of infinite order.
  \end{enumerate}
\end{example}

\section{Structures in mapping class groups of closed 4-manifolds}
\label{section:closed-4-manifolds}

In this section, we study the mapping class groups of closed 4-manifolds.
In \cref{subsection:invariant-on-closed-4-manifold}, we define an invariant of homeomorphisms using the topological Freedman--Quinn invariant of spheres, show that it induces a homomorphism (\cref{theorem:homeomorphisms-tDeltaGR}), and prove a realization result (\cref{theorem:realization-tDeltaGR}).
The method we use is similar to that used to develop the invariant $\Delta_D$ which we defined using the Dax invariant of neat disks.
In \cref{subsection:structures-in-torelli-closed}, we show that our invariant detects structures in the Torelli group of closed 4-manifolds.
In \cref{subsection:applications-closed}, we discuss various applications.
In \cref{subsection:torelli-complex-surfaces}, we apply our results to compact complex surfaces.

In this section, homotopy of homeomorphisms of $M$ means proper homotopy rel~$\partial M$.

\subsection{Invariants of homeomorphisms of closed 4-manifolds}
\label{subsection:invariant-on-closed-4-manifold}

Let $M$ be a 4-manifold.
We allow $M$ to have either empty or nonempty boundary.
In particular, $M$ may be closed.
Suppose that $G$ is a framed 2-sphere embedded in the interior of~$M$.
Fix a basepoint $z_0 \in G \subset M$.
Fix an open tubular neighborhood $\nu G$ of~$G$.
Suppose that there is a 2-sphere $R$ embedded in $M$ such that $R\cap \nu G$ is a normal fiber disk at~$z_0$.
We say that such a 2-sphere $R$ is \emph{geometrically dual to $G$ at~$z_0$}.

Let
\[
  H_{G,R}(M) = \{f\in\Homeo(M) \mid f|_{\nu G} = \id_{\nu G},\; f(R) \simeq R,\; f_* = \id \text{~on~} \pi_1(M)\}.
\]
Note that the homotopy $f(R)\simeq R$ in the above definition is not required to be rel~$\nu G$.
It is readily seen that $H_{G,R}(M) \subset \Homeo(M)$ is a subgroup, and the set of (free) homotopy classes
\[
  H_{G,R}(M)/{\simeq} = H_{G,R}(M)/(H_{G,R}(M)\cap \Homeoh(M))
\]
is a subgroup of the quotient $\Homeo(M)/{\simeq} = \pi_0 \Homeo(M)/\pi_0 \Homeoh(M)$.

For homotopic embedded spheres $R_0$ and $R_1$ in $M$ which are geometrically dual to~$G$, the topological Freedman--Quinn invariant $\FQ(R_0,R_1) \in \Z_2[T_M]/\mu(\pi_3(M))$ is defined, following \cite[Section~4]{Schneiderman-Teichner:2022-1} and \cite[Appendix]{Cha-Kim:2023-1}.
See also~\cite{Klug-Miller:2021-1,Klug-Miller:2022-1}.
Here, $\mu\colon \pi_3(M) \to \Z_2[T_M]$ is the self-intersection homomorphism described in \cref{subsubsection:freedman-quinn-top-disks}.
We briefly review the Freedman--Quinn invariant $\FQ(R_0,R_1)$ in \cref{subsubsection:fq-spheres} below.

Abusing notation, in view of \cref{remark:D_M/Sigma_M-is-Z_2-vector-space}, we denote $\Z_2[T_M]/\mu(\pi_3(M))$ by $D_M/\Sigma_M$ even when $M$ has empty boundary, although $D_M$ in~\cref{equation:D_M} is only defined when $M$ has nonempty boundary.

\begin{definition}
  \label{definition:DeltaGR}
  For $f\in H_{G,R}(M)$, define $\DeltaGR(f) = \FQ(R,f(R)) \in D_M/\Sigma_M$.
\end{definition}

The goal of this subsection is to prove the following results.

\begin{theorem}
  \label{theorem:homeomorphisms-tDeltaGR}
  The invariant $\DeltaGR$ induces a homomorphism $\tDeltaGR\colon H_{G,R}(M)/{\simeq} \to D_M/\Sigma_M$.
\end{theorem}

\begin{theorem}
  \label{theorem:realization-tDeltaGR}
  The restriction of $\tDeltaGR\colon H_{G,R}(M)/{\simeq} \to D_M/\Sigma_M$ to the subgroup $(I(M)\cap H_{G,R}(M))/{\simeq}$ is onto $D_M/\Sigma_M$.
  In addition, if $M$ is smooth and $G$ is smoothly embedded, then the composition
  \[
    (I^s(M) \cap H_{G,R}(M))/{\simeq} \hookrightarrow (I(M)\cap H_{G,R}(M))/{\simeq} \xrightarrow{\smash{\tDeltaGR}} D_M/\Sigma_M
  \]
  is surjective;
  if $M$ is smooth but $G$ is allowed to be non-smooth, then for every $v\in D_M/\Sigma_M$, there is a stably smoothable homeomorphism $f\in I(M)\cap H_{G,R}(M)$ such that $\tDeltaGR([f])=v$.
\end{theorem}

\subsubsection{Freedman--Quinn invariant of spheres}
\label{subsubsection:fq-spheres}

We begin with a brief review of the Freedman--Quinn invariant of spheres, focusing on the definitions and properties we use.
We follow \cite[Section~4]{Schneiderman-Teichner:2022-1} and \cite[Appendix]{Cha-Kim:2023-1}.
It is analogous to the disk case treated in \cref{subsubsection:freedman-quinn-top-disks}.

Fix basepoints $(z,0)\in S^2\times I$ and $(z_0,0,\frac12)\in M\times I^2$.
For a based map $f\colon S^2\times I \to M\times I^2$ such that $f|_{S^2\times i}$ is an embedding $S^2\times i \hookrightarrow M\times i \times \frac12$ for $i=0,1$, the self-intersection invariant $\mu(f) \in \Z_2[T_M]$ is defined similarly to~\cref{equation:self-intersection-mu} in \cref{subsubsection:freedman-quinn-top-disks}: homotope $f$ rel $S^2\times \partial I$ to a generic immersion into $M\times I^2$, and count the signed self-intersections labeled by the associated $\pi_1(M)$ elements, to obtain an element in $\Z[\pi_1(M)]/\langle 1,g+g^{-1}\rangle$.
By~\cite[Lemma~4.1]{Schneiderman-Teichner:2022-1}, the value lies in the subgroup $\Z_2[T_M]$ generated by the 2-torsion elements.
The value of $\mu(f)$ is well-defined and invariant under homotopy of $f$ rel $S^2\times \partial I$ in both the smooth and topological categories~\cite{Schneiderman-Teichner:2022-1,Cha-Kim:2023-1}.

Suppose that $R_0$ and $R_1$ are embedded spheres in $M$ which are geometrically dual to $G$ at~$z_0$ and (freely) homotopic to each other.
The Freedman--Quinn invariant of $(R_0,R_1)$ is defined by
\[
  \FQ(R_0,R_1) = \mu(f) \in \Z_2[T_M] / \mu(\pi_3(M)) = D_M/\Sigma_M
\]
where $f\colon S^2\times I \to M\times I^2$ is a map such that $f(x,i) = (R_i(x),i,\frac12)$ for $i=0,1$.
For instance, if $H\colon S^2\times I \to M$ is a homotopy $H\colon R_0 \simeq R_1$, then the track $\widehat H \colon S^2\times I \to M\times I^2$ defined by $\widehat H(x,t) = (H(x,t),t,\frac12)$ can be used as~$f$.
The value of $\FQ(R_0,R_1)$ is well-defined, independent of the choice of $f$, by the arguments in~\cite[Lemma 4.4]{Schneiderman-Teichner:2022-1}, \cite[Proposition 4.9]{Klug-Miller:2021-1} and \cite[Proposition 3.3]{Klug-Miller:2022-1}.

\subsubsection{Properties of \texorpdfstring{$\DeltaGR$}{Δ\_\{G,R\}}}

\begin{lemma}
  \label{lemma:DeltaGR-additivity}
  For $f$, $g\in H_{G,R}(M)$, $\DeltaGR(f\circ g) = \DeltaGR(f) + \DeltaGR(g)$.
\end{lemma}

\begin{proof}
  It is known that $\FQ(R_0,R_2)=\FQ(R_0,R_1)+\FQ(R_1,R_2)$ if $R_0 \simeq R_1 \simeq R_2$.
  Indeed, by~\cite[Lemma~2.1]{Schneiderman-Teichner:2022-1} (see also \cite[Proposition 4.9]{Klug-Miller:2021-1}), there are based homotopies $F\colon R_0 \simeq R_1$ and $G\colon R_1 \simeq R_2$, and then we have
  $\FQ(R_0,R_2)=\mu(\widehat{F*G}) = \mu(\widehat F) + \mu(\widehat G) = \FQ(R_0,R_1)+\FQ(R_1,R_2)$.
  The second equality uses that $F$ is a based homotopy.

  Also, if $h\in \Homeo(M)$ satisfies $h(z_0)=z_0$ and $h_*=\id$ on $\pi_1(M)$, then $\FQ(h(R_0), h(R_1)) = \mu(\widehat{hF}) = \mu(\widehat F) = \FQ(R_0,R_1)$ where $F\colon R_0 \simeq R_1$.

  Applying these to our case, we obtain $\FQ(R,fg(R)) = \FQ(R,f(R)) + \FQ(f(R),fg(R))= \FQ(R,f(R))+\FQ(R,g(R))$.
\end{proof}

\begin{lemma}
  \label{lemma:DeltaGR-homotopy-trivial-vanishing}
  If $f\in H_{G,R}(M)$ is properly homotopic to~$\id_M$, then $\DeltaGR(f)=0$.
\end{lemma}

\begin{proof}
  The same argument as in the proof of \cref{lemma:fq-homotopy-trivial-homeomorphism} applies.
  We omit details.
\end{proof}

\begin{proof}[Proof of \cref{theorem:homeomorphisms-tDeltaGR}]
  The map $H_{G,R}(M) \to D_M/\Sigma_M$ defined by $f\mapsto \DeltaGR(f)$ is a homomorphism by \cref{lemma:DeltaGR-additivity}.
  Its kernel contains the subgroup $H_{G,R}(M)\cap \Homeoh(M)$ by \cref{lemma:DeltaGR-homotopy-trivial-vanishing}.
  Thus, $\DeltaGR$ induces a homomorphism $\tDeltaGR$ of the quotient $H_{G,R}(M)/ (H_{G,R}(M)\cap \Homeoh(M)) = H_{G,R}(M)/\simeq$ to~$D_M/\Sigma_M$.
\end{proof}

\subsubsection{Realization of values of \texorpdfstring{$\DeltaGR$}{\Delta\_GR}}
\label{subsubsection:realization-DeltaGR}

Let $N = M \sm \nu G$ be the exterior of the framed 2-sphere $G\subset M$.
Let $D = N\cap R$, a neat 2-disk in~$N$.
There is a natural inclusion $\Homeo(N) \hookrightarrow \Homeo(M)$ defined by $g \mapsto g \cup \id_{\nu G}$.
Note that $\pi_1(M)\cong \pi_1(N)$ since the meridian circle of $G$ bounds the disk~$D\subset N$.
Thus, the above inclusion restricts to $H_D(N) \hookrightarrow H_{G,R}(M)$;
recall that $H_D(N)\subset \Homeo(M)$ is the subgroup of homeomorphisms $g$ of $N$ such that $g|_{\partial N}=\id_{\partial N}$, $g_*=\id$ on $\pi_1(N)$, and $D \simeq g(D) \rel \partial$.

\Cref{corollary:homomorphism-modulo-homotopy} gives the homomorphism $\tDeltaD\colon H_D(N)/{\simeq}\rel\partial N \to D_N/\Sigma_N$, which sends $[g]$ to the reduction of $\Dax_N(D,g(D))\in D_N$ modulo~$\Sigma_N$.
By~\cref{equation:dax-reduces-to-fq}, $\tDeltaD([g]) = \FQ_N(D,g(D))$.

The diagram
\[
  \begin{tikzcd}
    \pi_3(N) \ar[d] \ar[r,"\mu"]
    & \Z_2[T_N] \ar[d,"\cong"]
    \\
    \pi_3(M) \ar[r,"\mu"']
    & \Z_2[T_M]
  \end{tikzcd}
\]
is commutative, by the definition of~$\mu$ in~\cref{equation:homomorphism-mu}.
Consequently, the bijection $T_N \approx T_M$ induces an epimorphism $q\colon D_N/\Sigma_N = \Z_2[T_N]/\mu(\pi_3(N)) \to D_M/\Sigma_M = \Z_2[T_M]/\mu(\pi_3(M))$.

Furthermore, from the definition of the Freedman--Quinn invariant, it is readily seen that $q(\FQ_N(D,g(D))) = \FQ_M(R,(g\cup \id_{\nu G})(R))$.
This shows that the following diagram is commutative.
\begin{equation}
  \begin{tikzcd}[]
    H_D(N)/{\simeq}\rel\partial N \ar[r] \ar[d,"\tDeltaD"']
    & H_{G,R}(M)/\simeq \ar[d,"\tDeltaGR"]
    \\
    D_N/\Sigma_N \ar[r,two heads,"q"']
    & D_M/\Sigma_M
  \end{tikzcd}
  \label{equation:tDeltaGR-tDeltaD-compatibility}
\end{equation}

\begin{proof}[Proof of \cref{theorem:realization-tDeltaGR}]
  Recall that the group of Torelli homeomorphisms is defined by $I(M)=\{f\in \Homeo(M) \mid f_*=\id \text{~on~} H_*(M)\}$.

  We claim that if $g\in I(N)\cap H_D(N)$, then $f=g\cup\id_{\nu G}$ lies in $I(M)\cap H_{G,R}(M)$.
  Indeed, we have observed that $f\in H_{G,R}(M)$.
  For $i=0,1,3$, since $g_*=\id$ on $H_i(N)$ and since $H_i(M)$ is a quotient of $H_i(N)$, $f_*=\id$ on $H_i(M)$.
  Since $f$ is orientation preserving, $f_*=\id$ on $H_4(M)$.
  Also, $H_2(M)=H_2(N)\oplus \Z$, where the $\Z$ summand is generated by the class $[R]$.
  Since $g\in I(N)$, $f_*=\id$ on $H_2(N)\subset H_2(M)$.
  Since $f\in H_{G,R}(M)$, $f_*([R])=[R]$.
  Thus $f_*=\id$ on $H_2(M)$.
  This verifies the claim.

  By the claim, the diagram~\cref{equation:tDeltaGR-tDeltaD-compatibility} restricts to the following commutative diagram.
  \begin{equation}
    \begin{tikzcd}[column sep=small]
      \bigl(I(N)\cap H_D(N)\bigr)/{\simeq}\rel\partial N \ar[r] \ar[d,"\tDeltaD"']
      & \bigl(I(M)\cap H_{G,R}(M)\bigr)/\simeq \ar[d,"\tDeltaGR"]
      \\
      D_N/\Sigma_N \ar[r,two heads]
      & D_M/\Sigma_M
    \end{tikzcd}
    \label{equation:tDeltaGR-tDeltaD-torelli-compatibility}
  \end{equation}
  By~\cref{theorem:realization-DeltaD}, $\tDeltaD$ is onto.
  It follows that $\tDeltaGR$ is onto, by the commutativity of~\cref{equation:tDeltaGR-tDeltaD-torelli-compatibility}.

  If $M$ is smooth and $G\subset M$ is smoothly embedded, then $N$ is also smooth.
  Thus, the additional assertion for the smooth case follows from the smooth case of~\cref{theorem:realization-DeltaD}.

  Suppose that $M$ is smooth but $G\subset M$ is not necessarily smooth.
  For $v\in D_M/\Sigma_M$, choose a preimage $w\in D_N/\Sigma_N$ under the epimorphism~$q$.
  By the proof of \cref{theorem:realization-DeltaD}, there is an arc pushing map $g\in I(N)\cap H_D(N) \subset \Homeo(N)$ such that $\tDeltaD(g)=w$.
  By \cref{equation:tDeltaGR-tDeltaD-torelli-compatibility}, the homeomorphism $f=g\cup \id_{\nu G}\in I(M) \cap H_{G,R}(M)$ satisfies $\tDeltaGR(f)=v$.
  By \cref{proposition:stably-smoothable-arc-pushing} below, $f$ is stably smoothable.
\end{proof}

\subsubsection{Stable smoothability of arc pushing maps}
\label{subsubsection:stably-smoothable-arc-pushing}

Let $R'$ be another 2-sphere in $M$ such that $R'\cap \nu G = R\cap \nu G$, and let $D' = R' \cap N$.
Suppose $D \simeq D'$ rel~$\partial$ in $N$.
Then an arc pushing map $g\in \Homeo(N)$ sending $D$ to $D'$ is defined, following~\cite[Lemma~5.3]{Gabai:2021-1}, as we discussed in \cref{subsection:realization-D_M}.
We call the homeomorphism $f=g \cup \id_{\nu(G)} \in \Homeo(M)$ an \emph{arc pushing map for $(R,R',G)$}.

When $M$ is smooth and $G$ is smoothly embedded in $M$, an arc pushing map $h$ for $(R,R',G)$ is topologically isotopic to a diffeomorphism, since each of $D=R\cap N$ and $D'=R'\cap N$ is topologically isotopic rel $\partial$ to a smooth disk by the light bulb smoothing theorem~\cite[Theorem~C]{Cha-Kim:2023-1}.
The following result shows that $f$ is always stably smoothable, even when $G$ is not necessarily smoothable.
The key ingredient of the proof is the stable smoothing theorem~\cite[Theorem~D]{Cha-Kim:2023-1} for topological surfaces in a smooth 4-manifold.

\begin{proposition}
  \label{proposition:stably-smoothable-arc-pushing}
  Let $M$ be a smooth 4-manifold and $G$ be a topologically embedded 2-sphere in $M$ with trivial normal bundle.
  Suppose that $R$ and $R'$ are topologically embedded 2-spheres in~$M$, which are (freely) homotopic and geometrically dual to $G$ at $z_0 \in G$.
  Then the arc pushing map for $(R,R',G)$ is stably smoothable.
\end{proposition}

\begin{proof}
  Let $f\in\Homeo(M)$ be an arc pushing map for $(R,R',G)$.
  Since $f$ is supported in an arbitrary neighborhood of $R\cup R'\cup G$, we may assume that $f=\id$ near a smooth closed 4-ball $B\subset\inte M\sm(R\cup R'\cup G)$.
  Write $M'$ for $M\csum k(S^2\times S^2)$, where the connected sum is taken at~$B$. Let $f'=f\csum\id_{k(S^2\times S^2)} \in \Homeo(M')$.
  By the stable smoothing theorem~\cite[Theorem~D]{Cha-Kim:2023-1}, for some $k\ge 0$, there is a topological ambient isotopy $\phi_t\colon M'\to M'$ rel~$\partial M'$ such that $\phi_0=\id_{M'}$ and $\phi_1(G)$ is smoothly embedded.
  We may assume that $\phi_1(\nu G)$ is a smooth normal bundle of $\phi_1(G)$ and $\phi_1$ preserves the $\R^2$-bundle structure, by uniqueness of topological normal bundles~\cite[Section~9.3]{Freedman-Quinn:1990-1}.

  Note that $\phi_t^{\vphantom1} f' \phi_t^{-1}$ is an isotopy from $f'$ to $h=\phi_1^{\vphantom1} f' \phi_1^{-1}$, which is an arc pushing map for $(\phi_1(R),\phi_1(R'),\phi_1(G))$.
  Since $\phi_1(G)$ is smoothly embedded, each of $\phi_1(R)$ and $\phi_1(R')$ is ambiently isotopic to a smoothly embedded sphere in $M'$ rel~$\phi_1(\nu G)$ by the light bulb smoothing theorem~\cite[Theorem~C]{Cha-Kim:2023-1}.
  Thus $h$ is topologically isotopic rel~$\partial M'$ to a diffeomorphism.
  From this, it follows that $f'$ is topologically isotopic rel~$\partial M'$ to a diffeomorphism, and consequently $f$ is stably smoothable.
\end{proof}

\subsection{Torelli groups of closed 4-manifolds}
\label{subsection:structures-in-torelli-closed}

In this subsection, we record results on Torelli groups of closed 4-manifolds obtained by using the homomorphism $\tDeltaGR$ in \cref{subsection:invariant-on-closed-4-manifold}.
These may be viewed as analogous to the results for the nonempty boundary case in \cref{subsection:structures-in-torelli}.

\begin{corollary}
  \label{corollary:subgroup-in-torelli-closed}
  Let $M$ be a 4-manifold containing embedded 2-spheres $R$ and $G$ such that $G$ has trivial normal bundle and $R$ is geometrically dual to~$G$.
  Suppose $H_3(\widetilde M) \otimes \Z_2$ is finitely generated.
  Let $s=|T_M|$ if $T_M$ is infinite.
  Otherwise, let
  \[
    s = |T_M| - \dim_{\Z_2} H_3(\widetilde M) \otimes \Z_2
  \]
  and suppose $s>0$.
  Then there exists an epimorphism of a subgroup $H\subset \cI(M)/\pi_0 \Homeoh(M)$ onto $(\Z_2)^s$, the $\Z_2$-vector space of dimension~$s$.

  In addition, if $M$ is smooth and $G\subset M$ is smoothly embedded, then $H$ is the image of a subgroup of $\cI^s(M)/\pi_0 \Diffh(M)$ under the natural map;
  if $M$ is smooth but $G$ is allowed to be non-smooth, then every element of $H$ is represented by a stably smoothable homeomorphism.
\end{corollary}

\begin{proof}
  By \cref{theorem:realization-tDeltaGR}, there is a subgroup $H \subset {I(M)/\simeq} = \cI(M)/\pi_0 \Homeoh(M)$ satisfying the following:
  $H$ admits an epimorphism onto~$D_M/\Sigma_M$.
  If $M$ is smooth, then all elements of $H$ are represented by stably smoothable homeomorphisms.
  If $M$ is smooth and $G\subset M$ is smoothly embedded, then all elements of $H$ are represented by diffeomorphisms.
  Since $D_M/\Sigma_M$ admits an epimorphism onto $(\Z_2)^s$ by \cref{lemma:dimension-of-D_M/Sigma_M}, the conclusion follows.
\end{proof}

The following is a consequence of \cref{corollary:subgroup-in-torelli-closed}.
\begin{corollary}
  \label{corollary:torelli-prufer-rank-closed}
  Let $M$ and $s$ be as in \cref{corollary:subgroup-in-torelli-closed}.
  Then $\rk_{\Pr}\bigl(\cI(M)/\pi_0 \Homeoh(M)\bigr) \ge s$.
  In addition, if $M$ is smooth and $G\subset M$ is smoothly embedded, then
  \[
    \rk_{\Pr}\biggl(\frac{\cI^s(M)}{\pi_0 \Diffh(M)}\biggr)
    \ge \rk_{\Pr}\biggl( \Im\biggl\{ \frac{\cI^s(M)}{\pi_0 \Diffh(M)} \rightarrow \frac{\cI(M)}{\pi_0 \Homeoh(M)} \biggr\} \biggr)
    \ge s.
  \]
  If $M$ is smooth but $G$ is allowed to be non-smooth, then
  \[
    \rk_{\Pr}(I^{ss}(M)/\simeq)
    \ge \rk_{\Pr}\biggl( \Im\biggl\{ {I^{ss}(M)/\simeq} \rightarrow \frac{\cI(M)}{\pi_0 \Homeoh(M)} \biggr\} \biggr)
    \ge s
  \]
  where $I^{ss}(M)=\{f\in I(M) \subset \Homeo(M) \mid f \text{~is stably smoothable}\}$.
\end{corollary}

\subsection{Applications}
\label{subsection:applications-closed}

Let
\[
  S^2 \tildetimes \RP^2 = (S^2 \times S^2) / (z,x) \sim (\bar z, -x)
\]
where $\bar z$ denotes $(x_1,x_2,-x_3)$ for $z=(x_1,x_2,x_3) \in S^2 \subset \R^3$.
This is a twisted $S^2$-bundle over $\RP^2$ with orientable total space.
Also, $S^2 \tildetimes \RP^2$ is the double of $D^2\tildetimes \RP^2$.

\begin{example}
  \label{example:torelli-closed-M-explicit}
  Let
  \[
    M = (S^2 \tildetimes \RP^2) \csum (S^2 \tildetimes \RP^2) \csum (\csum^k S^2\times S^2)
  \]
  where $k\ge 1$.

  The framed 2-spheres $R=*\times S^2$ and $G=S^2 \times *$ in an $S^2\times S^2$ summand of $M$ are geometrically dual to each other.
  If we write $A = (\csum^k S^2\times S^2) \sm (\text{two open $4$-balls})$ and $B = (S^2\times S^2) \sm (\text{two open $4$-balls})$, then the universal cover $\widetilde M$ is obtained by forming an infinite linear chain
  \[
    \widetilde M = \cdots \cup A \cup B \cup A \cup B \cup  \cdots
  \]
  where each copy of $A$ or $B$ is glued to the adjacent copies along the boundary 3-spheres.
  It follows that $H_3(\widetilde M)=\Z$.
  In addition, we have $|T_M|=\infty$.
  Thus, by \cref{corollary:torelli-prufer-rank-closed}, both $\cI(M)/\pi_0\Homeoh(M)$ and $\cI^s(M)/\pi_0\Diffh(M)$ have infinite Pr\"ufer rank.
\end{example}

In \cref{example:torelli-closed-M-explicit}, we used the lower bound $|T_M| - \dim_{\Z_2} \bigl( H_3(\widetilde M)\otimes \Z_2 \bigr)$ in \cref{corollary:torelli-prufer-rank-closed}.
The following example illustrates that our method can also be useful even when this lower bound is not effective.
For this purpose, the following observation is useful.

\begin{lemma}
  \label{lemma:vanishing-mu-for-connected-sum}
  Suppose that $M=M_1\csum \cdots \csum M_k$ where each $M_i$ is a 4-manifold with $H_3(\widetilde M_i)=0$.
  Then, $\mu(\pi_3(M))=0$.
\end{lemma}

\begin{proof}
  Let $\pi=\pi_1(M)$.
  Let $V_1 = M_1$ with $k-1$ open 4-balls removed, and let $V_i = M_i$ with an open 4-ball removed for $i\ge 2$.
  Then $M=V_1\cup\cdots\cup V_k$, where $V_1 \cap V_i$ is a boundary 3-sphere for $i\ge 2$.
  From the long exact sequence for $(M_i,V_i)$ with $\Z[\pi]$-coefficients, it follows that $H_3(V_i;\Z[\pi])$ is generated over $\Z[\pi]$ by the classes of the boundary 3-spheres of~$V_i$.
  From the Mayer--Vietoris sequence for $M=V_1\cup\cdots\cup V_k$ with $\Z[\pi]$-coefficients, it follows that $\bigoplus H_3(V_i;\Z[\pi])$ surjects onto $H_3(M;\Z[\pi])$.
  Consequently, $H_3(\widetilde M) = H_3(M;\Z[\pi])$ is generated by the classes of the connected sum 3-spheres.
  Since these 3-spheres are embedded and $\mu\colon \pi_3(M) \to \Z_2[T_M]$ factors through $H_3(\widetilde M)$, it follows that $\mu(\pi_3(M))=0$.
\end{proof}

\begin{example}
  \label{example:torelli-closed-M-more}
  Generalizing \cref{example:torelli-closed-M-explicit}, let $M = (\csum^n S^2 \tildetimes \RP^2) \csum (\csum^k S^2\times S^2)$ where $n\ge 3$, $k\ge 1$.
  We will show that both $\cI(M)/\pi_0\Homeoh(M)$ and $\cI^s(M)/\pi_0\Diffh(M)$ have infinite Pr\"ufer rank.

  We have $|T_M|=\infty$.
  Note that the universal cover $\widetilde M$ is obtained by an infinite construction similar to the one in \cref{example:torelli-closed-M-explicit}, but for $n\ge 3$, it is modeled on an infinite $n$-valent tree instead of a linear chain.
  Thus, $H_3(\widetilde M) \cong \Z^\infty$, and consequently the lower bound $|T_M| - \dim_{\Z_2} \bigl( H_3(\widetilde M)\otimes \Z_2 \bigr)$ is not well-defined.
  However, by \cref{lemma:vanishing-mu-for-connected-sum}, $D_M/\Sigma_M = \Z_2[T_M]/\mu(\pi_3(M)) \cong (\Z_2)^\infty$, and consequently by \cref{theorem:realization-tDeltaGR}, both $\cI(M)/\pi_0\Homeoh(M)$ and $\cI^s(M)/\pi_0\Diffh(M)$ have a subgroup which maps onto $(\Z_2)^\infty$.
  From this, the assertion follows immediately.
\end{example}

\begin{example}
  \label{example:torelli-closed-from-aspherical}
  Let $M = (S^2 \tildetimes \RP^2) \csum X \csum (S^2 \times S^2)$ where $X$ is an arbitrary closed, non-simply connected, aspherical 4-manifold (not necessarily smooth).
  Then, the arguments in \cref{example:torelli-closed-M-more} show that $\mu(\pi_3(M)) = 0$ and $D_M/\Sigma_M \cong (\Z_2)^\infty$.
  Therefore, by \cref{theorem:realization-tDeltaGR}, $\cI(M)/\pi_0\Homeoh(M)$ has a subgroup which maps onto $(\Z_2)^\infty$, and thus has infinite Pr\"ufer rank.
  In addition, if $X$ is smooth, then $\cI^s(M)/\pi_0\Diffh(M)$ also has infinite Pr\"ufer rank.
\end{example}

\subsection{Torelli groups of complex surfaces}
\label{subsection:torelli-complex-surfaces}

In this subsection, we apply our results to complex surfaces, to prove \cref{theorem:intro-torelli-complex-surfaces,corollary:intro-torelli-dolgachev-enriques}.

\begin{proof}[Proof of \cref{theorem:intro-torelli-complex-surfaces}]
  Suppose that $M$ is a compact complex surface such that $\pi_1(M)$ is a nontrivial finite cyclic group of even order and $e(M) \ne 4$.
  Our goal is to prove that there exists a stably smoothable homeomorphism of $M$ that represents a nontrivial element in $\cI(M)/\pi_0\Homeoh(M)$.

  Since $\pi_1(M)$ is cyclic of even order, $|T_M|=1$.
  Since the universal cover $\widetilde M$ is a closed, simply connected 4-manifold, $H_3(\widetilde M)=H^1(\widetilde M) = 0$ by Poincar\'e duality.
  By work of Hambleton and Kreck~\cite{Hambleton-Kreck:1988-1}, $M$ is homeomorphic (but not necessarily diffeomorphic) to $X\csum (S^2\times S^2)$ for some (topological) 4-manifold~$X$.
  Consequently, $M$ has a pair of topological spheres with trivial normal bundle which are geometrically dual to each other.
  Thus, by \cref{corollary:torelli-prufer-rank-closed} (or \cref{corollary:subgroup-in-torelli-closed}), $\cI(M)/\pi_0\Homeoh(M)$ has a nontrivial element represented by a stably smoothable homeomorphism.
  This completes the proof.

  For completeness, we describe below how one verifies the claim that $M$ has an $S^2\times S^2$ summand topologically, following the proof of Corollaries~4 and~5 in \cite[p.~58]{Hambleton-Kreck:1988-1}.

  When $w_2(\widetilde M)\ne 0$, Hambleton and Kreck proved that $M$ is homeomorphic to $\Sigma \csum r \C P^2 \csum s \overline{\C P}{}^2$ for some rational homology 4-sphere~$\Sigma$ and for some $r$, $s \ge 0$.
  In addition, they observed that $M$ has an indefinite intersection form, and thus $r$, $s \ge 1$.
  Since $e(M) = r+s+2\ne 4$, we have $r+s \ge 3$.
  It follows that $S^2\times S^2$ is a summand of $M$, since $2 \C P^2 \csum \overline{\C P}{}^2 \cong \C P^2 \csum (S^2\times S^2)$ and $\C P^2 \csum 2\overline{\C P}{}^2 \cong \overline{\C P}{}^2 \csum (S^2\times S^2)$.

  Now suppose $w_2(\widetilde M) = 0$.
  Hambleton and Kreck showed that $e(M)>2$ and $M$ is homeomorphic to either $\Sigma \csum r (\pm K) \csum s(S^2\times S^2)$ or $(\pm D_{k,k}) \csum r(\pm K) \csum s(S^2\times S^2)$, where $\Sigma$ is a rational homology 4-sphere, $K$ is a Kummer surface, $D_{k,k}$ is a Dolgachev surface of type $(k,k)$, and $r$, $s\ge 0$.
  In the former case, $e(M)$ would be 2 if $r=s=0$, and thus either $r$ or $s$ is positive.
  It follows that $M$ has an $S^2\times S^2$ summand, since $K\cong 2(-|E_8|) \csum 3(S^2\times S^2)$, where $|E_8|$ is Freedman's $E_8$ manifold.
  In the latter case, $M$ has an $S^2\times S^2$ summand since $D_{k,k} \cong \Sigma' \csum (-|E_8|) \csum (S^2\times S^2)$ for some rational homology 4-sphere $\Sigma'$ by~\cite[Proof of Theorem 3]{Hambleton-Kreck:1988-1}.
\end{proof}

\begin{proof}[Proof of \cref{corollary:intro-torelli-dolgachev-enriques}]
  For an Enriques surface $E$, we have $e(E)=12$ and $\pi_1(E)=\Z_2$.
  For a Dolgachev surface $D_{p,q}$ with $p$, $q$ even, we have $e(D_{p,q})=12$ and $\pi_1(D_{p,q})=\Z_n$ where $n=\gcd(p,q)$ is even.
  Thus, \cref{theorem:intro-torelli-complex-surfaces} applies to $E$ and $D_{p,q}$, to give the desired conclusion.
\end{proof}

\paragraph{Star-partner of the Enriques surface}

For a 4-manifold $M$, the \emph{star-partner} of $M$ is a 4-manifold $*M$ admitting a homeomorphism $M\csum (*\CP^2) \cong (*M)\csum \CP^2$ which preserves the connected sum decomposition of~$\pi_2$.
Here $*\CP^2$ is Freedman's simply connected 4-manifold with intersection form $\langle 1\rangle$, the simple homotopy type of $\CP^2$, and $\ks(*\CP^2)=1$.

By work of Teichner~\cite{Teichner:1997-1}, every closed, connected, oriented, non-spin 4-manifold $M$ with $\pi_1(M)\cong \Z_2$ has a star-partner~$*M$.
In particular, the Enriques surface $E$ admits a star-partner~$*E$.
It is a non-smoothable 4-manifold with $\ks(*E)=1$ and the simple homotopy type of~$E$.

\begin{theorem}
  \label{theorem:torelli-star-enriques}
  The quotient $\cI(*E)/\pi_0\Homeoh(*E)$ is nontrivial.
\end{theorem}

\begin{proof}
  By work of Hambleton and Kreck~\cite{Hambleton-Kreck:1988-1}, $E$ is homeomorphic to $(-|E_8|) \csum X \csum (S^2\times S^2)$, where $|E_8|$ is Freedman's $E_8$ manifold and $X$ is a rational homology sphere with $\pi_1(X)=\Z_2$~\cite{Nakamura:2010-1}.
  Since $E$ is non-spin, $X$ is non-spin.
  Thus, $X$ has a star-partner $*X$ by \cite{Teichner:1997-1}, and $*E \cong (-|E_8|) \csum (*X) \csum (S^2\times S^2)$.
  It follows that $*E$ contains two framed spheres which are geometrically dual to each other.
  Also, since $*E \simeq E$, we have $H_3(\widetilde{*E}) = H_3(\widetilde{E}) = 0$ and $\pi_1(*E) = \pi_1(E) = \Z_2$.
  Therefore, by \cref{corollary:torelli-prufer-rank-closed}, $\cI(*E)/\pi_0\Homeoh(*E)$ is nontrivial.
\end{proof}

\bibliographystyle{amsalpha}
\def\MR#1{}
\bibliography{research}

\newcommand{\etalchar}[1]{$^{#1}$}
\def\cprime{$'$}
\providecommand{\bysame}{\leavevmode\hbox to3em{\hrulefill}\thinspace}
\providecommand{\MR}{\relax\ifhmode\unskip\space\fi MR }
\providecommand{\MRhref}[2]{%
  \href{http://www.ams.org/mathscinet-getitem?mr=#1}{#2}
}
\providecommand{\href}[2]{#2}
\begin{thebibliography}{GGH{\etalchar{+}}26}

\bibitem[Bar26]{Baraglia:2026-1}
David Baraglia, \emph{On the mapping class groups of simply connected smooth 4-manifolds}, 2026, pp.~1635--1653. \MR{5079109}

\bibitem[BG19]{Budney-Gabai:2019-1}
Ryan Budney and David Gabai, \emph{Knotted 3-balls in {$S^4$}}, arXiv:1912.09029v3, 2019.

\bibitem[BG25]{Budney-Gabai:2025-1}
\bysame, \emph{On the automorphism groups of hyperbolic manifolds}, Int. Math. Res. Not. IMRN (2025), no.~7, Paper No. rnaf083, 28. \MR{4889279}

\bibitem[BT25]{Baraglia-Tomlin:2025-1}
David Baraglia and Joshua Tomlin, \emph{Exotic diffeomorphisms of 4-manifolds with $b_+ = 2$}, arXiv:2511.13976, 2025.

\bibitem[Cha26]{Cha:2026-2}
Jae~Choon Cha, \emph{A homological approach to {D}ax invariants}, preprint, 2026.

\bibitem[CK08]{Cha-Kim:2008-1}
Jae~Choon Cha and Taehee Kim, \emph{Covering link calculus and iterated {B}ing doubles}, Geom. Topol. \textbf{12} (2008), no.~4, 2173--2201. \MR{MR2431018 (2009f:57007)}

\bibitem[CK25]{Cha-Kim:2023-1}
Jae~Choon Cha and Byeorhi Kim, \emph{Light bulb smoothing for topological surfaces in 4-manifolds}, J. Eur. Math. Soc. (JEMS) (2025), published online first.

\bibitem[CKP20]{Cha-Kim-Powell:2020-1}
Jae~Choon Cha, Min~Hoon Kim, and Mark Powell, \emph{A family of freely slice good boundary links}, Math. Ann. \textbf{376} (2020), no.~3-4, 1009--1030. \MR{4081108}

\bibitem[CLR08]{Cha-Livingston-Ruberman:2006-1}
Jae~Choon Cha, Charles Livingston, and Daniel Ruberman, \emph{Algebraic and {H}eegaard-{F}loer invariants of knots with slice {B}ing doubles}, Math. Proc. Cambridge Philos. Soc. \textbf{144} (2008), no.~2, 403--410. \MR{MR2405897}

\bibitem[CP14]{Cha-Powell:2014-1}
Jae~Choon Cha and Mark Powell, \emph{Covering link calculus and the bipolar filtration of topologically slice links}, Geom. Topol. \textbf{18} (2014), no.~3, 1539--1579. \MR{3228458}

\bibitem[Dax72]{Dax:1972-1}
Jean-Pierre Dax, \emph{\'{E}tude homotopique des espaces de plongements}, Ann. Sci. \'{E}cole Norm. Sup. (4) \textbf{5} (1972), 303--377. \MR{321110}

\bibitem[DHMS24]{Dai-Hedden-Mallick-Stoffregen:2024-1}
Irving Dai, Matthew Hedden, Abhishek Mallick, and Matthew Stoffregen, \emph{Rank-expanding satellites, {W}hitehead doubles, and {H}eegaard {F}loer homology}, J. Topol. \textbf{17} (2024), no.~4, Paper No. e70008, 40. \MR{4832710}

\bibitem[DIS{\etalchar{+}}25]{Daemi-Imori-Sato-Scaduto-Taniguchi:2025-1}
Aliakbar Daemi, Hayato Imori, Kouki Sato, Christopher Scaduto, and Masaki Taniguchi, \emph{Instantons, special cycles and knot concordance}, Geom. Topol. \textbf{29} (2025), no.~8, 4189--4298. \MR{5000788}

\bibitem[FQ90]{Freedman-Quinn:1990-1}
Michael~H. Freedman and Frank Quinn, \emph{Topology of 4-manifolds}, Princeton Mathematical Series, vol.~39, Princeton University Press, Princeton, NJ, 1990. \MR{MR1201584 (94b:57021)}

\bibitem[FR11]{Foozwell-Rubinstein:2011-1}
Bell Foozwell and Hyam Rubinstein, \emph{Introduction to the theory of {H}aken {$n$}-manifolds}, Topology and geometry in dimension three, Contemp. Math., vol. 560, Amer. Math. Soc., Providence, RI, 2011, pp.~71--84. \MR{2866924}

\bibitem[Fre82]{Freedman:1982-1}
Michael~H. Freedman, \emph{The topology of four-dimensional manifolds}, J. Differential Geom. \textbf{17} (1982), no.~3, 357--453. \MR{MR679066 (84b:57006)}

\bibitem[Fre84]{Freedman:1984-1}
\bysame, \emph{The disk theorem for four-dimensional manifolds}, Proceedings of the {I}nternational {C}ongress of {M}athematicians, {V}ol.\ 1, 2 ({W}arsaw, 1983) (Warsaw), PWN, 1984, pp.~647--663. \MR{MR804721 (86m:57016)}

\bibitem[Fre85]{Freedman:1985-1}
\bysame, \emph{A new technique for the link slice problem}, Invent. Math. \textbf{80} (1985), no.~3, 453--465. \MR{MR791669 (86m:57017)}

\bibitem[Fre88]{Freedman:1988-2}
\bysame, \emph{{${\rm Whitehead}\sb 3$} is a ``slice'' link}, Invent. Math. \textbf{94} (1988), no.~1, 175--182. \MR{MR958596 (89h:57004)}

\bibitem[Fre93]{Freedman:1993-1}
\bysame, \emph{Link compositions and the topological slice problem}, Topology \textbf{32} (1993), no.~1, 145--156. \MR{MR1204412 (94m:57011)}

\bibitem[FT95]{Freedman-Teichner:1995-2}
Michael~H. Freedman and Peter Teichner, \emph{{$4$}-manifold topology. {II}. {D}wyer's filtration and surgery kernels}, Invent. Math. \textbf{122} (1995), no.~3, 531--557. \MR{MR1359603 (96k:57016)}

\bibitem[Gab21]{Gabai:2021-1}
David Gabai, \emph{Self-referential discs and the light bulb lemma}, Comment. Math. Helv. \textbf{96} (2021), no.~3, 483--513. \MR{4344778}

\bibitem[Gal24]{Galvin:2024-1}
Daniel A.~P. Galvin, \emph{The {C}asson-{S}ullivan invariant for homeomorphisms of 4-manifolds}, arXiv:2405.07928, to appear in Compos. Math., 2024.

\bibitem[GGH{\etalchar{+}}26]{Gabai-Gay-Hartman-Krushkal-Powell:2023-1}
David Gabai, David Gay, Daniel Hartman, Vyacheslav Krushkal, and Mark Powell, \emph{Pseudo-isotopies of simply connected 4-manifolds}, Forum Math. Pi \textbf{14} (2026), Paper No. e9, 25. \MR{5038934}

\bibitem[GN25]{Galvin-Nonino:2025-1}
Daniel Galvin and Isacco Nonino, \emph{Pseudo-isotopy versus isotopy for homeomorphisms of 4-manifolds}, arXiv:2506.11905, 2025.

\bibitem[Hed07]{Hedden:2007-1}
Matthew Hedden, \emph{Knot {F}loer homology of {W}hitehead doubles}, Geom. Topol. \textbf{11} (2007), 2277--2338. \MR{2372849 (2008m:57030)}

\bibitem[HK88]{Hambleton-Kreck:1988-1}
Ian Hambleton and Matthias Kreck, \emph{Smooth structures on algebraic surfaces with cyclic fundamental group}, Invent. Math. \textbf{91} (1988), no.~1, 53--59. \MR{918236}

\bibitem[HK12]{Hedden-Kirk:2010-1}
Matthew Hedden and Paul Kirk, \emph{Instantons, concordance, and {W}hitehead doubling}, J. Differential Geom. \textbf{91} (2012), no.~2, 281--319. \MR{2971290}

\bibitem[HPC21]{Hedden-Pinzon-Caicedo:2021-1}
Matthew Hedden and Juanita Pinz\'{o}n-Caicedo, \emph{Satellites of infinite rank in the smooth concordance group}, Invent. Math. \textbf{225} (2021), no.~1, 131--157. \MR{4270665}

\bibitem[KK24]{Krannich-Kupers:2024-1}
Manuel Krannich and Alexander Kupers, \emph{A note on homotopy and pseudoisotopy of diffeomorphisms of 4-manifolds}, C. R. Math. Acad. Sci. Paris \textbf{362} (2024), 1515--1520. \MR{4824946}

\bibitem[KM21]{Klug-Miller:2021-1}
Michael~R. Klug and Maggie Miller, \emph{Concordance of surfaces in 4-manifolds and the {F}reedman-{Q}uinn invariant}, J. Topol. \textbf{14} (2021), no.~2, 560--586. \MR{4286049}

\bibitem[KM22]{Klug-Miller:2022-1}
Michael Klug and Maggie Miller, \emph{Concordance of spheres in 4-manifolds with an immersed dual sphere}, arXiv:2211.07177, 2022.

\bibitem[KT24]{Kosanovic-Teichner:2024-2}
Danica Kosanovi\'c and Peter Teichner, \emph{A new approach to light bulb tricks: disks in 4-manifolds}, Duke Math. J. \textbf{173} (2024), no.~4, 673--721. \MR{4734552}

\bibitem[Lev12]{LevineA:2009-1}
Adam~Simon Levine, \emph{Slicing mixed {B}ing-{W}hitehead doubles}, J. Topol. \textbf{5} (2012), no.~3, 713--726. \MR{2971611}

\bibitem[Lic62]{Lickorish:1962-1}
W.~B.~R. Lickorish, \emph{A representation of orientable combinatorial {$3$}-manifolds}, Ann. of Math. (2) \textbf{76} (1962), 531--540. \MR{0151948 (27 \#1929)}

\bibitem[Loo21]{Looijenga:2021-1}
Eduard Looijenga, \emph{Teichm\"uller spaces and {T}orelli theorems for hyperk\"ahler manifolds}, Math. Z. \textbf{298} (2021), no.~1-2, 261--279. \MR{4257085}

\bibitem[LVC11]{Livingston-VanCott:2011-1}
Charles Livingston and Cornelia~A. Van~Cott, \emph{Concordance of {B}ing doubles and boundary genus}, Math. Proc. Cambridge Philos. Soc. \textbf{151} (2011), no.~3, 459--470. \MR{2838347}

\bibitem[LWXZ26]{Lin-Wu-Xie-Zhang:2026-1}
Jianfeng Lin, Weiwei Wu, Yi~Xie, and Boyu Zhang, \emph{Dax invariants, light bulbs, and isotopies of symplectic structures}, arXiv:2501.16083v5, 2026.

\bibitem[LXZ26a]{Lin-Xie-Zhang:2026-2}
Jianfeng Lin, Yi~Xie, and Boyu Zhang, \emph{The dax isomorphism for topological $4$-manifolds}, arXiv:2609.15792, 2026.

\bibitem[LXZ26b]{Lin-Xie-Zhang:2026-1}
\bysame, \emph{Pseudo-isotopies of 3-manifolds with infinite fundamental groups}, arXiv:2602.09454, 2026.

\bibitem[Man16]{Manolescu:2016-1}
Ciprian Manolescu, \emph{Pin(2)-equivariant {S}eiberg-{W}itten {F}loer homology and the triangulation conjecture}, J. Amer. Math. Soc. \textbf{29} (2016), no.~1, 147--176. \MR{3402697}

\bibitem[MO07]{Manolescu-Owens:2007-1}
Ciprian Manolescu and Brendan Owens, \emph{A concordance invariant from the {F}loer homology of double branched covers}, Int. Math. Res. Not. IMRN (2007), no.~20, Art. ID rnm077, 21. \MR{2363303}

\bibitem[Nak10]{Nakamura:2010-1}
Nobuhiro Nakamura, \emph{Smoothability of {$\Bbb Z\times\Bbb Z$}-actions on 4-manifolds}, Proc. Amer. Math. Soc. \textbf{138} (2010), no.~8, 2973--2978. \MR{2644908}

\bibitem[NST24]{Nozaki-Sato-Taniguchi:2024-1}
Yuta Nozaki, Kouki Sato, and Masaki Taniguchi, \emph{Filtered instanton {F}loer homology and the homology cobordism group}, J. Eur. Math. Soc. (JEMS) \textbf{26} (2024), no.~12, 4699--4761. \MR{4780492}

\bibitem[OP25]{Orson-Powell:2022-1}
Patrick Orson and Mark Powell, \emph{Mapping class groups of simply connected 4-manifolds with boundary}, J. Differential Geom. \textbf{131} (2025), no.~1, 199--275. \MR{4947553}

\bibitem[OS03]{Ozsvath-Szabo:2003-1}
Peter Ozsv{\'a}th and Zolt{\'a}n Szab{\'o}, \emph{Knot {F}loer homology and the four-ball genus}, Geom. Topol. \textbf{7} (2003), 615--639 (electronic). \MR{MR2026543 (2004i:57036)}

\bibitem[Per86]{Perron:1986-1}
B.~Perron, \emph{Pseudo-isotopies et isotopies en dimension quatre dans la cat\'egorie topologique}, Topology \textbf{25} (1986), no.~4, 381--397. \MR{862426}

\bibitem[Qui82]{Quinn:1982-1}
Frank Quinn, \emph{Ends of maps. {III}. {D}imensions {$4$} and {$5$}}, J. Differential Geom. \textbf{17} (1982), no.~3, 503--521. \MR{679069 (84j:57012)}

\bibitem[Qui86]{Quinn:1986-1}
\bysame, \emph{Isotopy of {$4$}-manifolds}, J. Differential Geom. \textbf{24} (1986), no.~3, 343--372. \MR{868975 (88f:57020)}

\bibitem[Ram25]{Raman:2025-1}
Sidhanth Raman, \emph{A smooth {B}irman-{H}ilden theory for hyper-{K}\"ahler manifolds}, J. Reine Angew. Math. \textbf{822} (2025), 147--178. \MR{4899235}

\bibitem[Rub98]{Ruberman:1998-1}
Daniel Ruberman, \emph{An obstruction to smooth isotopy in dimension {$4$}}, Math. Res. Lett. \textbf{5} (1998), no.~6, 743--758. \MR{1671187}

\bibitem[Rub99]{Ruberman:1999-1}
\bysame, \emph{A polynomial invariant of diffeomorphisms of 4-manifolds}, Proceedings of the {K}irbyfest ({B}erkeley, {CA}, 1998), Geom. Topol. Monogr., vol.~2, Geom. Topol. Publ., Coventry, 1999, pp.~473--488. \MR{1734421}

\bibitem[ST22]{Schneiderman-Teichner:2022-1}
Rob Schneiderman and Peter Teichner, \emph{Homotopy versus isotopy: spheres with duals in 4-manifolds}, Duke Math. J. \textbf{171} (2022), no.~2, 273--325. \MR{4375617}

\bibitem[ST26]{Salvatore-Turchin:2026-1}
Paolo Salvatore and Victor Turchin, \emph{On the smoothing theory delooping of disc diffeomorphism and embedding spaces}, Math. Z. \textbf{313} (2026), no.~2, Paper No. 30, 56. \MR{5076882}

\bibitem[Tei97]{Teichner:1997-1}
Peter Teichner, \emph{On the star-construction for topological {$4$}-manifolds}, Geometric topology ({A}thens, {GA}, 1993), AMS/IP Stud. Adv. Math., vol. 2.1, Amer. Math. Soc., Providence, RI, 1997, pp.~300--312. \MR{1470734}

\bibitem[Wal99]{Wall-Ranicki:1999-1}
C.~Terence~C. Wall, \emph{Surgery on compact manifolds}, second ed., Mathematical Surveys and Monographs, vol.~69, American Mathematical Society, Providence, RI, 1999, Edited and with a foreword by A. A. Ranicki. \MR{2000a:57089}

\bibitem[Whi50]{Whitehead:1950-1}
J.~H.~C. Whitehead, \emph{A certain exact sequence}, Ann. of Math. (2) \textbf{52} (1950), 51--110. \MR{35997}

\end{thebibliography}

\end{document}